\documentclass[11pt]{amsart}
    \pdfoutput=1
\usepackage[maxbibnames=99]{biblatex}
    \usepackage{graphicx}
    \usepackage[english]{babel}
    \usepackage{framed}
    \usepackage[normalem]{ulem}
    \usepackage{amsmath}
    \usepackage{dynkin-diagrams}
    \usepackage{amsthm}
    \usepackage{amssymb}
    \usepackage{hyperref} 
    \usepackage[capitalize]{cleveref}
    \usepackage{comment}
    \usepackage{tikz}
    \usetikzlibrary{matrix,calc}
    \usepackage{mathtools}
    \usepackage{tikz-cd}
    \usepackage{amsfonts}
    \usepackage{enumerate}
    \usepackage[utf8]{inputenc}
    \usepackage[top=1 in,bottom=1in, left=1 in, right=1 in]{geometry}
       \usepackage{stmaryrd}

    \usepackage[colorinlistoftodos]{todonotes}
    \usepackage[all]{xy}
    \usepackage{mathrsfs}

    \newcommand{\R}{\mathbb{R}} 
    \newcommand{\N}{\mathcal{N}} 
    \newcommand{\Q}{\mathbb{Q}} 
    \newcommand{\Z}{\mathbb{Z}} 
    \newcommand{\A}{\mathbb{A}} 
    \newcommand{\G}{\mathbb{G}} 
    
    \newcommand{\SL}{\operatorname{SL}}

    \newcommand{\LGL}{\mathfrak{gl}}
    
    \newcommand{\LG}{\mathfrak{g}}
    
    \newcommand{\LT}{\mathfrak{t}}
    \newcommand{\LN}{\mathfrak{n}}
    \newcommand{\LB}{\mathfrak{b}}

    \newcommand{\QCoh}{\operatorname{QCoh}}

    \renewcommand{\O}{\mathcal{O}} 
    \newcommand{\F}{\mathcal{F}} 
    \renewcommand{\\}{\backslash}
    \renewcommand{\subset}{\subseteq}

    \theoremstyle{definition}
    \newtheorem{Theorem}{Theorem}[section]

    \newtheorem{Corollary}[Theorem]{Corollary}
    
    \newtheorem{Proposition}[Theorem]{Proposition}
    
    \newtheorem{Remark}[Theorem]{Remark}
    \newtheorem{Example}[Theorem]{Example}
    \newtheorem{Lemma}[Theorem]{Lemma}

    \title{Crepant Partial Resolutions of the Nilpotent Cone via Hamiltonian Reduction}

    \newcommand{\LTd}{\LT^{\ast}}
\newcommand{\Spec}{\mathrm{Spec}}
\newcommand{\characterlatticeforT}{X^{\bullet}(T)}
    
\begin{document}

\newcommand{\Cox}{\mathrm{Cox}}
\newcommand{\affineClosureofBasicAffineSpace}{\overline{G/U}}
\newcommand{\affineClosureOfCotangentBundleofBasicAffineSpace}{\overline{T^*(G/U)}}
\newcommand{\SpecOfSymofSectionsOfTangentBundleOfBasicAffineSpace}{\text{Spec}(\text{Sym}_{\affineClosureofBasicAffineSpace}^{\bullet}(\mathcal{T}_{\overline{G/U}}))}
\newcommand{\ringOfFunctionsForBasicAffineSpace}{A}
\newcommand{\ringOfFunctionsForCOTANGENTBUNDLEOfBasicAffineSpace}{R}
\newcommand{\tangentSheafForBasicAffineSpace}{\mathcal{T}_{G/U}}
\newcommand{\symOfDirectSumOfRepsofFundamentalWeights}{\text{Sym}(\oplus_i E(\omega_i))}
\newcommand{\Sym}{\text{Sym}}
\newcommand{\projectionFromAffineClosureofCotangentBundleToAffineClosureofSpace}{\overline{\pi}}
\newcommand{\momentMapFromAFFINECLOSUREofCotangentSpaceWithGROUPG}{\overline{\mu}_G}
\newcommand{\momentMapFromAFFINECLOSUREofCotangentSpaceWithGROUPT}{\overline{\mu}_T}
\newcommand{\momentMapFromAFFINECLOSUREofCotangentSpacewithTOTALGROUP}{\overline{\mu}}
\newcommand{\singularLocusOfAffineClosureOfCOTANGENTBUNDLEofBasicAffineSpace}{Z}
\newcommand{\singularLocusOfAffineClosureofBasicAffineSpace}{Z_0}
\newcommand{\idealForSingularLocusOfAffineClosureOfCOTANGENTBUNDLEofBasicAffineSpace}{I_Z}
\newcommand{\groundfield}{k}
\newcommand{\LGd}{\mathfrak{g}^*}
\newcommand{\rhocheck}{\rho^{\vee}}
\newcommand{\unipotentRadicalOfPARABOLICSUBGROUP}{U_P}
\newcommand{\lieAlgebraOfUnipotentRadicalOfPARABOLICSUBGROUP}{\mathfrak{u}_P}
\newcommand{\affinizationOfGrothendieckSpringerResolution}{\LGd \times_{\LTd\sslash W} \LTd}
\newcommand{\smoothLocusOfAffineclosureBasicAffineSpace}{\mathcal{S}}
\newcommand{\Proj}{\mathrm{Proj}}
\newcommand\blfootnote[1]{%
  \begingroup
  \renewcommand\thefootnote{}\footnote{#1}%
  \addtocounter{footnote}{-1}%
  \endgroup
}
\newcommand{\dominantCharactersForT}{X^+(T)}
\newcommand{\Pic}{\mathrm{Pic}}
\newcommand{\Mov}{\mathrm{Mov}}
\newcommand{\tcN}{\widetilde{\mathcal{N}}}
\newcommand{\LUP}{\mathfrak{u}_P}
\newcommand{\LP}{\mathfrak{p}}
\newcommand{\Symt}{\mathrm{Sym}(\mathfrak{t})}
\renewcommand{\L}{\mathcal{L}}
\newcommand{\LU}{\mathfrak{u}}
\newcommand{\LL}{\mathfrak{l}}
\newcommand{\tg}{\widetilde{\mathfrak{g}}}
\newcommand{\tgP}{\tg_P}

\author{Gwyn Bellamy}
\address{School of Mathematics and Statistics, University Place, Glasgow, G12 8QQ, Glasgow, UK.}
\email{gwyn.bellamy@glasgow.ac.uk}
\urladdr{https://www.gla.ac.uk/schools/mathematicsstatistics/staff/gwynbellamy/}

\author{Tom Gannon} 
\address{University of California, Riverside: 900 University Ave. Riverside, CA, USA.}
\email{tom.gannon@ucr.edu}
\urladdr{https://profiles.ucr.edu/app/home/profile/tomg}

\begin{abstract}
We show that the nilpotent cone associated to a simply connected semisimple algebraic group, together with all its crepant projective partial resolutions, can be constructed as Hamiltonian reductions of the affine closure of the cotangent bundle of base affine space for suitable choices of stability parameter of a maximal torus. We also realize the base change of the universal Poisson deformation of each of these crepant projective partial resolutions as the GIT quotient of the affine closure of the cotangent bundle of base affine space at the same stability parameter.
\end{abstract}

\maketitle
\section{Introduction}


In work in preparation \cite{BCSHamiltonian}, Bellamy-Craw-Schedler associate to each conic symplectic singularity $X$ its \textit{Hamiltonian Cox space} $Z$. This is defined to be the spectrum of the Cox ring of the universal graded Poisson deformation of any $\Q$-factorial terminalization of $X$. It is conjectured that $Z$ should be a conic symplectic singularity with $\Q$-factorial terminal singularities, equipped with a Hamiltonian action of a torus $T$ such that the original singularity $X$ and all crepant projective\footnote{We say that a morphism of schemes $X \to Y$ is \textit{projective} if it satisfies \cite[Definition 17.3.1]{VakilRisingSeaFoundationsofAlgebraicGeometry}. If $Y$ is a quasiprojective scheme, this is equivalent to the existence of a closed embedding $X \xhookrightarrow{} \mathbb{P}_Y^n$ for some $n$ \cite[Exercise 17.3.A(c)]{VakilRisingSeaFoundationsofAlgebraicGeometry}.}  partial resolutions of $X$ can be obtained as Hamiltonian reductions of $Z$ by $T$ at suitable choices of the stability parameter. In this article, we confirm this conjecture in the case where $X$ is the nilpotent cone of a semisimple (complex) Lie algebra. 

Let $G$ be a simply connected 
semisimple group scheme over $k := \mathbb{C}$; choose a maximal torus $T$ contained inside a fixed Borel $B \subseteq G$, and let $U$ denote the unipotent radical of $B$.  It is standard (see, for example, \cite[Section 2.1]{GannonProofOftheGinzburgKazhdanConjecture} and the references therein) that the varieties $G/U$ and $T^*(G/U)$ are \textit{quasi-affine}; in other words, the affinization maps naturally exhibit $G/U$, respectively $T^*(G/U)$, as open subsets of \[\affineClosureofBasicAffineSpace := \Spec(\O(G/U))\text{, respectively } \affineClosureOfCotangentBundleofBasicAffineSpace := \Spec(\O(T^*(G/U))).\]
The algebraic nature of the definition of the Hamiltonian Cox space means that it is very difficult to identify the space $Z$ in examples. Since Grothendieck's simultaneous resolution $\tg := G \times^B (\LG/\LU)^*$ gives the universal graded Poisson deformation of the $\Q$-factorial terminalization $\tcN$ of the nilpotent cone, \cite[Lemma 3.6.1]{GinzburgRicheDifferentialOperatorsOnBasicAffineSpaceandtheAffineGrassmannian} implies that:

\begin{Theorem}
    The Hamiltonian Cox space of the nilpotent cone $\mathcal{N}$ is isomorphic to the affine closure $\affineClosureOfCotangentBundleofBasicAffineSpace$ of the cotangent bundle of base affine space. 
\end{Theorem}

With this identification, the main result of \cite{GannonProofOftheGinzburgKazhdanConjecture} 
shows that the Hamiltonian Cox space $\affineClosureOfCotangentBundleofBasicAffineSpace$ of the nilpotent cone $\mathcal{N}$ is a conic symplectic singularity with $\Q$-factorial terminal singularities. The focus of this article is therefore to confirm the second half of the conjecture by Bellamy-Craw-Schedler, that the original singularity $\mathcal{N}$ and all crepant partial resolutions of $\mathcal{N}$ can be obtained as Hamiltonian reductions of $\affineClosureOfCotangentBundleofBasicAffineSpace$ by $T$ at suitable choices of the stability parameter.


\subsection{The main result} We now make this more precise. Let $\Pi^{\vee}$ denote the set of simple coroots associated to the inclusion $T \subseteq B$. The vanishing and non-vanishing of various $\alpha^{\vee} \in \Pi^{\vee}$ induce the structure of a fan on the set of dominant weights $X^+(T)_{\R}$ in $\characterlatticeforT_{\R}$--the associated set of strongly convex rational polyhedral cones is cut out by the vanishing and non-vanishing of the various simple coroots $\alpha^{\vee}$. 
To any cone $C$ in this fan there exists an associated parabolic subgroup $P_C$ containing $B$, and it is a classical result (see for example, \cite[Chapter 30.1]{HumphreysLinearAlgebraicGroups}) that this gives a complete classification of all subgroups of $G$ containing $B$. We will also label parabolic subgroups $P_C := P_{\lambda}$ by a fixed $\lambda$ in the (relative) interior of the cone $C$ of this fan.

We recall how the crepant (projective) partial resolutions of $\mathcal{N}$ are constructed. Assume $P$ is any parabolic subgroup of $G$. One can construct the variety 
\[\tcN^P := G \times^P (\mathcal{N} \cap (\LG/\mathfrak{u}_P)^*),\] 
where $\mathfrak{u}_P$ is the Lie algebra of the unipotent radical of $P$. Observe that a choice of $G$-invariant bilinear form induces an isomorphism 
\[
\tcN^P \cong G \times^P \mathfrak{p}_{\mathrm{nilp}},
\] where $\mathfrak{p}_{\mathrm{nilp}}$ denotes the set of nilpotent elements in the Lie algebra of $P$. Observe further that if $P \subseteq Q$ are parabolic subgroups containing $B$ then there is a natural $G$-equivariant map 
\[\nu^{Q}_P: \tcN^P \to \tcN^{Q}.
\]
It is known (and, as we explain in \cref{Partial Crepant Resolutions Are Fces of Mori Fan}, is a consequence of our results below) that every crepant partial resolution of the nilpotent cone $\mathcal{N} = \tcN^G$ is given by some $\nu_P^G$ for a parabolic subgroup $P$ containing $B$.   

There is a natural action of $T$ on $G/U$ which induces a Hamiltonian $T$-action on $T^*(G/U)$. The moment map for $T^*(G/U)$ induces a $T$-equivariant map \[\momentMapFromAFFINECLOSUREofCotangentSpaceWithGROUPT: \affineClosureOfCotangentBundleofBasicAffineSpace \to \LTd\] on its affinization. For any character $\lambda \in \characterlatticeforT$, we define \begin{equation}\label{Definition of Hamiltonian Reduction by Group Character}\affineClosureOfCotangentBundleofBasicAffineSpace/\!/\!/_{\lambda} T := \momentMapFromAFFINECLOSUREofCotangentSpaceWithGROUPT^{-1}(0)\sslash_{\lambda} T := \Proj \left(\oplus_{n \ge 0} \, \O(\momentMapFromAFFINECLOSUREofCotangentSpaceWithGROUPT^{-1}(0))_{n\lambda} \right)\end{equation} where we equip $\oplus_{n \ge0} \, \O(\momentMapFromAFFINECLOSUREofCotangentSpaceWithGROUPT^{-1}(0))_{n\lambda}$ with the obvious multiplication map. Our main result says that the varieties \labelcref{Definition of Hamiltonian Reduction by Group Character} precisely give the crepant partial resolutions of the nilpotent cone.

\begin{Theorem}\label{Intro Main Theorem}
With the above notation, we have the following: \begin{enumerate}
    \item For any dominant character $\lambda \in \characterlatticeforT$, there is an isomorphism 
    \[
    \affineClosureOfCotangentBundleofBasicAffineSpace/\!/\!/_{\lambda} T \cong \tcN^{P_{\lambda}}
    \]
    which intertwines the variation of GIT quotient maps with the maps $\nu_{P_\lambda}^{P_{\lambda'}}$ for any $\lambda' \in \characterlatticeforT$ such that $P_{\lambda} \subseteq P_{\lambda'}$. In particular, the variety \labelcref{Definition of Hamiltonian Reduction by Group Character} is independent of character $\lambda$ in the (relative) interior of a given cone in the above fan. 
    \item Every crepant partial resolution of $\mathcal{N}$ has the form \labelcref{Definition of Hamiltonian Reduction by Group Character} for some $\lambda$ and, moreover, the assignment $C \mapsto \tcN^{P_C}$ induces an order-preserving bijective correspondence \[\{\text{crepant partial resolutions of } \mathcal{N}\} \leftrightarrow \{\text{(relative) interior of cones in }\dominantCharactersForT_{\mathbb{R}}\}.\]
    \end{enumerate}
\end{Theorem}

\cref{Intro Main Theorem}(2) recovers \cite[Corollary 4.18]{SanMiguelMalaneyPartialResolutionsOfAffineSymplecticSingularities}. 
\cref{Intro Main Theorem} follows from the more general \cref{Copy of BCS Theorem in Our SetupP}, where we describe the Mori fan of the relative Mori dream space $\tcN^P \to \mathcal{N}$. 

\subsection{The geometry of $\momentMapFromAFFINECLOSUREofCotangentSpaceWithGROUPT^{-1}(0)$}

Though not required for the main results of this paper, we also consider in \cref{sec:muinv0} the geometry of $\momentMapFromAFFINECLOSUREofCotangentSpaceWithGROUPT^{-1}(0)$. It is well-known that the preimage of zero under the moment map for a Hamiltonian action of a reductive group on a Poisson variety can be very badly behaved. For instance, it can fail to be irreducible or reduced, let alone normal. See for instance \cite{SchwarzHamred}. However, even though $\overline{\mu}_T$ is the moment map for the action of $T$ on a highly singular variety, it turns out that $\momentMapFromAFFINECLOSUREofCotangentSpaceWithGROUPT^{-1}(0)$ is remarkably well behaved. 

\begin{Theorem}\label{Preimage of affinized moment map is integral}
$\momentMapFromAFFINECLOSUREofCotangentSpaceWithGROUPT^{-1}(0)$ is a normal integral Cohen-Macaulay scheme.
\end{Theorem}

The fact that $\momentMapFromAFFINECLOSUREofCotangentSpaceWithGROUPT^{-1}(0)$ is Cohen-Macaulay follows immediately from the fact that it is defined as the subscheme cut out by a regular sequence on the Cohen-Macaulay scheme $    \affineClosureOfCotangentBundleofBasicAffineSpace$. Irreducibility and normality are much harder to show and occupy the majority of \cref{sec:muinv0}.

\subsection{Universal Poisson deformations via GIT}

The universal graded Poisson deformation of the nilpotent cone $\mathcal{N}$, respectively of the Springer resolution $\tcN$, is given by $\mathfrak{g}^* \to \mathfrak{t}^*/W$ and $\tg \to \mathfrak{t}^*$ respectively. So that they have the same base $\mathfrak{t}^*$, we base change $\mathfrak{g}^*$ to $\mathcal{X}' := \mathfrak{g}^* \times_{\mathfrak{t}^*/W} \mathfrak{t}^*$. Once again, we can recover these spaces and all other crepant (partial) resolutions of $\mathcal{X}'$ by VGIT of the affine closure of the cotangent bundle of base affine space; the difference is that we do not need to perform Hamiltonian reduction this time. 

For each parabolic $P \subset G$ with Lie algebra $\LP$, whose nilpotent radical is $\mathfrak{u}_P$, let $\tgP := G \times^P (\LG / \mathfrak{u}_P)^*$. Fixing a Levi subgroup $T \subset L$ with Lie algebra $\mathfrak{l}$, the scheme $\tgP$ admits a map to the quotient $\LL^*\sslash \! L$ and we\newcommand{\tgTHISBaseChanged}[1]{V_{#1}} set $\tgTHISBaseChanged{P} := \tgP \times_{\LL^*\sslash L} \LT^*$. The morphism $V_P \to \mathcal{X}'$ is a crepant partial resolution. If $P \subset Q$ then, as previously, there is a $G$-equivariant map $\tilde{\nu}^Q_P \colon \tgTHISBaseChanged{P} \to \tgTHISBaseChanged{Q}$, commuting with the morphisms to $\mathcal{X}'$. We will also use the fan structure on $\characterlatticeforT_\R$ cut out by the vanishing and non-vanishing of \textit{all} (not necessarily simple) coroots.

\begin{Theorem}\label{Intro Main Theorem2}
With the above notation, we have the following: \begin{enumerate}
    \item For any dominant character $\lambda \in \characterlatticeforT$, there is an isomorphism 
    \[
    \affineClosureOfCotangentBundleofBasicAffineSpace/\!/_{\lambda} T \cong V_{P_{\lambda}}
    \]
    which intertwines the variation of GIT quotient maps with the maps $\tilde{\nu}_{P_{\lambda}}^{P_{\lambda'}}$ for any $\lambda'$ such that $P_{\lambda} \subseteq P_{\lambda'}$.
    \item The domain of every crepant partial resolution of $ \mathcal{X}'$ is isomorphic to $V_{P_{\lambda^+}}$ for some dominant weight $\lambda^+$. Moreover, any crepant partial resolution of $\mathcal{X}'$ is isomorphic to the composite \[V_{P_{\lambda^+}} \to \mathcal{X}' := \LGd \times_{\LTd\sslash W} \LTd \xrightarrow{\mathrm{id} \times w} \LGd \times_{\LTd\sslash W} \LTd\] for some $w \in W$. This construction induces an order-preserving bijective correspondence \[ \{\text{crepant partial resolutions of } \mathcal{X}'\} \xrightarrow{\sim}\{\text{(relative) interior of cones in }\characterlatticeforT_{\mathbb{R}}\}.\]
    \end{enumerate}
\end{Theorem}

We also relate the \textit{Gelfand-Graev action} of $W$ on $\affineClosureOfCotangentBundleofBasicAffineSpace$ to the Namikawa-Weyl group action on $\LTd$ obtained via deformation theory, see \cref{Gelfand Graev and Namikawa Weyl Subsection}.

\subsection*{Acknowledgments} We would like to thank Alberto San Miguel Malaney and Victor Ginzburg for interesting and useful discussions. The first author would also like to thank Alastair Craw and Travis Schedler for teaching him many things about symplectic singularities. The second author would also like to thank the University of Glasgow, where much of this project was completed, for their hospitality. He would also like to thank the first author for having him over for dinner during his visit to Glasgow.

The first author is supported by EPSRC grants
EP-W013053-1 and EP-R034826-1, and the second author is supported by an AMS-Simons Travel Grant.

A draft of this work was checked by generative artificial intelligence (more specifically, ChatGPT 5.6 Sol) which identified some minor errors that have been corrected. Some of these errors were corrected partly in conversations with the model: specifically, the proofs of \cref{prop:semiample-induced-morphism}, \cref{lem:technicalnonmovable}, and \cref{prop:smalGrothediecksimnew} were outlined by the model. All text in this document is human generated, and the authors assume full responsibility for its contents.

\section{Notation}
\subsection{Notation for Groups} Let $G_{\mathbb{Z}}$ denote some simply connected split semisimple group over $\mathbb{Z}$ with a choice of maximal torus $T_{\mathbb{Z}}$, and let $G := G_{k}$ and $T := T_k$ denote the respective base changes to $k := \mathbb{C}$; note that our assumption that $G_\Z$ is simply connected automatically implies that $G$ is semisimple. We let $X_{\bullet}(T)$ denote the lattice of cocharacters and let $\characterlatticeforT$ denote the lattice of characters for $T$.  Choose some Borel subgroup $B \supseteq T$, and let $U$ denote the unipotent radical of $B$. We let $\LG, \LB, \LU, \LT$ denote the Lie algebras of $G, B, U$, and $T$ respectively and, for any vector space $V$, we let $V^*$ denote the dual vector space. We fix a non-degenerate associative $G$-invariant bilinear form on $\LG$, which allows us, when necessary, to identify $\LG \cong \LG^*$ as $G$-representations. We will occasionally abuse notation by denoting $\LT(\mathbb{Q}) := \text{Lie}(T_{\mathbb{Q}})(\mathbb{Q})$ and $\LTd(\mathbb{Q}) := \text{Lie}(T_{\mathbb{Q}})^*(\mathbb{Q})$. With this notation, we have isomorphisms \begin{equation}\label{characterlatticeforT tensors to give Lie Algebra}
            \characterlatticeforT \otimes_{\mathbb{Z}} \mathbb{Q} \xrightarrow{\sim} \LTd(\mathbb{Q})
        \end{equation} and \begin{equation}\label{cocharacterlatticeforT tensors to give Lie Algebra}
            X_{\bullet}(T) \otimes_{\mathbb{Z}} \mathbb{Q} \xrightarrow{\sim} \LT(\mathbb{Q})
        \end{equation} both induced by the differential. 

Let $\Pi^{\vee} = \{ \alpha_1^{\vee}, ..., \alpha_r^{\vee} \}$ denote the set of simple coroots and $\Pi = \{ \alpha_1, \dots, \alpha_r \}$ the set of simple roots with respect to $T \subset B$. Let $\omega_i \in \LTd$ denote the \textit{fundamental weights}, that is, the dual basis to the set of simple coroots. Since $G$ is simply connected, 
$\omega_i \in \characterlatticeforT$ for every $i$ and the map $\Z^r \to \characterlatticeforT$ given by $\vec{n} \mapsto \sum_{i = 1}^rn_i\omega_i$ is an isomorphism. We say a weight is dominant if each $n_i$ is nonnegative under this isomorphism. The Weyl group $W := N_G(T)/T$ naturally acts on $T$, and there is a natural induced action on $\LT$ and $\LTd$. 

Recall that a parabolic subgroup $P \subset G$ is \textit{standard} if $B \subset P$. Each parabolic subgroup of $G$ is conjugate to a (unique) standard parabolic and the latter are parameterized by subsets $S$ of $\Pi^{\vee}$. At times, it is convenient to use the bijection $\Pi \stackrel{\sim}{\to} \Pi^{\vee}$, $\alpha \mapsto \alpha^{\vee}$, to parameterize the standard parabolics by subsets of $\Pi$. Then $P_S := P_{S^{\vee}}$ for $S \subset \Pi$. 

\subsection{Reminders on Affine Closure of $T^*(G/U)$}\label{Reminders on Affine Closure of Cotangent Bundle of Base Affine Space Subsection}
The projection map $T^*(G/U) \to G/U$ induces a map $\projectionFromAffineClosureofCotangentBundleToAffineClosureofSpace: \affineClosureOfCotangentBundleofBasicAffineSpace \to \affineClosureofBasicAffineSpace$ on affinizations. The following proposition follows from exactly the same argument as in \cite[Proposition 4.3]{GannonProofOftheGinzburgKazhdanConjecture}, replacing $\mathcal{S}$ with $G/U$: 

\begin{Proposition}\label{Preimage of G Mod U Under Projection Is Its Cotangent Bundle}
The natural inclusion $T^*(G/U) \subseteq \projectionFromAffineClosureofCotangentBundleToAffineClosureofSpace^{-1}(G/U)$ is in fact an equality.
\end{Proposition}

\newcommand{\zeroSectionAffinized}{\overline{z}}
The zero section $G/U \to T^*(G/U)$ also induces a map $\zeroSectionAffinized: \affineClosureofBasicAffineSpace \to \affineClosureOfCotangentBundleofBasicAffineSpace$. 

Since $T$ is abelian, there is a $T$-action on $G/U$ given by the formula $t(gU) := gtU$. This action gives rise to a moment map $T^*(G/U) \to \LTd$. This extends to a moment map $\momentMapFromAFFINECLOSUREofCotangentSpaceWithGROUPT: \affineClosureOfCotangentBundleofBasicAffineSpace \to \LTd$ for the Poisson variety $\affineClosureOfCotangentBundleofBasicAffineSpace$.

\subsection{Line Bundles Associated to $G$}\label{Line Bundle Notation Subsection} For any character $\lambda: T \to \mathbb{G}_m$, we let $\mathbb{A}^1_{\lambda}$ denote the one-dimensional representation of $B$ obtained by pullback by the natural quotient map. Let $\mathcal{L}_{\lambda}^{G/B} \in \QCoh(G/B)$ denote the sheaf of sections of the line bundle \[G \times^B \mathbb{A}^1_{-\lambda} \to G/B\] on $G/B$. 
For any scheme $E \to G/B$ over $G/B$, such as $\widetilde{\LG} := G \times^B (\LG/\mathfrak{u})^*$ or its subbundle 
\[\tilde{\mathcal{N}} := G \times^B (\LG/\LB)^* \cong T^*(G/B),
\]
we let $\mathcal{L}^E_{\lambda}$ denote the pullback of the line bundle $\mathcal{L}_{\lambda}^{G/B}$ under the map to $G/B$. 

We set similar notation for a more general case: to any standard parabolic $P$ (i.e. a parabolic containing $B$) with associated set of simple coroots $S_P$ and any character $\lambda \colon P \to \mathbb{G}_m$ (or, equivalently, any character $\lambda \colon T \to \mathbb{G}_m$ such that $\lambda(\alpha^{\vee}) = 0$ for all $\alpha^{\vee} \in S_P$) we let $\mathcal{L}_{\lambda}^{G/P} \in \QCoh(G/P)$ denote the sheaf of sections of the line bundle \[G \times^P \mathbb{A}^1_{-\lambda} \to G/P.\] More generally, let $\mathcal{L}_{\lambda}^{E}$ denote the pullback of the line bundle $\mathcal{L}_{\lambda}^{G/P}$ for any scheme $E$ equipped with a map to $G/P$.

We recall the following well known result that, up to isomorphism, this construction gives all line bundles on $E$ if $E$ is a vector bundle over $G/P$:

\begin{Proposition}\label{Picard Group of Vector BUndles on Partial Flag Variety} With the above notation, we have the following: \begin{enumerate}
    \item The assignment $\lambda \mapsto \mathcal{L}^E_{\lambda}$ induces an isomorphism $\characterlatticeforT \xrightarrow{\sim} \Pic(E)$ for any vector bundle $E$ on $G/B$.
    \item More generally, if $E$ is a vector bundle on $G/P$ for a parabolic subgroup $P$ and $S_P$ is the associated set of simple coroots, then the Picard group of $E$ can be identified with the subset of $\lambda \in \characterlatticeforT$ for which $\lambda(\alpha^{\vee}) = 0$ for all $\alpha^{\vee} \in S_P$. 
\end{enumerate}
\end{Proposition}

\begin{proof}
In the case where $E = G/P$, (2), and thus (1), follows directly from standard results, see for example the short exact sequence at the end of \cite[Chapter 18f]{MilneAlgebraicGroupsBook}. 
We next claim that if $E$ is any vector bundle over $G/P$, pullback induces an isomorphism \[\Pic(G/P) \xrightarrow{\sim} \Pic(E).\] Indeed, this map is injective since a left inverse is given by pullback by the zero section and the surjectivity follows from \cite[Corollaire IV.21.4.11, Erratum]{DieudonneGrothendieckEGA}, as the projection map is a faithfully flat morphism to a normal variety whose fibers are vector spaces. 
\end{proof}

If $f \colon Y \to X$ is a morphism then $\Pic(Y/X)$, the relative Picard group, is defined to be the quotient of $\Pic(Y)$ by the image of the pull-back map $f^* \colon \Pic(X) \to \Pic(Y)$. We recall the following fact: 

\begin{Proposition}
    The Picard groups of $\mathcal{N}$ and $\LGd \times_{\LTd\sslash W} \LTd$ are trivial. In particular, $\Pic(\tcN/\mathcal{N}) = \mathrm{Pic}(\tcN)$ and $\Pic(\tg/(\LGd \times_{\LTd\sslash W} \LTd)) \cong \Pic(\tg)$.
\end{Proposition}

\begin{proof}
Recall that $\mathcal{N}$, respectively $\LGd \times_{\LTd\sslash W} \LTd$, is normal by \cite[Theorem 0.8]{KostantLieGroupRepresentationsonPolynomialRings}, respectively \cite[Lemma~14]{BorhoBrylinski2}. The fact that the Picard groups are trivial follows from the fact that any normal affine variety $\Spec(A)$ which admits a contracting $\G_m$-action (or, equivalently, $A$ admits a non-negative grading with $A_0 = k$) has a trivial Picard group, see for example \cite[Lemma 7.1.1]{LosevMasonBrownMatvieievskyiUnipotentIdealsandHarishChandraBimodules}. 
\end{proof}

We also recall the following well known result, which follows from standard arguments, see for example \cite[II.4.4,Remarks 1)]{JantzenRepresentationsofAlgebraicGroups}: 

\begin{Proposition}\label{Ample Cone for Partial Flag Variety}
    If $P$ is a parabolic subgroup of $G$ and $S_P$ is its associated set of simple coroots, then, under the embedding $\mathrm{Pic}(G/P) \subseteq \characterlatticeforT$ of \cref{Picard Group of Vector BUndles on Partial Flag Variety}, the ample cone is the set of all $\lambda \in \characterlatticeforT_{\R}$ with $\lambda(\alpha^{\vee}) = 0$ for all $\alpha^{\vee} \in S_P$ and $\lambda(\alpha^{\vee}) > 0$ for all $\alpha^{\vee} \notin S_P$.
\end{Proposition}

\subsection{Line Bundles In General}

We recall the following straightforward extension of a standard result in algebraic geometry (in the case $\L = \O_X$) to the setting of line bundles:

\begin{Lemma}\label{For Proper Birational Maps Unit Map Is Iso on Line Bundles}
    Assume $p: \mathcal{Y} \to Y$ is a proper birational map from an integral scheme $\mathcal{Y}$ onto a normal quasi-compact quasi-separated locally Noetherian scheme $Y$, and $\mathcal{L}$ is a locally free sheaf of rank one on $Y$. Then the unit map $u_p(\L): \L \to p_*p^*(\L)$ is an isomorphism.
\end{Lemma}

\begin{proof}
Assume $j: U \to Y$ is any open embedding, and let $U' := p^{-1}(U)$. Assume $j'$ is the open embedding of $U'$ into $\mathcal{Y}$ and that $p': U' \to U$ is the restriction of $p$ to $U'$ so that the diagram \begin{equation}\label{Open Embedding Diagram}\begin{tikzcd}
      U' \ar[r, "p'"] \ar[d, "j'"] & U \ar[d, "j"] \\
        \mathcal{Y} \ar[r,"p"] & Y  
    \end{tikzcd}\end{equation} is Cartesian. Observe that, for any quasicoherent sheaf $\F$ on $Y$, the diagram \begin{equation}\begin{tikzcd}
      j^*(\F) \ar[r, "j^*u_{p}(\F)"] \ar[d, "u_{p'}(j^*(\F))"]  & j^*p_*p^*(\F) \ar[d, "{b}(p^*\F)"] \\
        p'_*p'^*j^*(\F) \ar[r,"\cong"] & p'_*j'^*p^*(\F) 
    \end{tikzcd}\end{equation} commutes, where $b$ is the base change isomorphism and the horizontal arrow labeled \lq $\cong$\rq{} is given by the functoriality of pullback. Therefore, since it suffices to show $u_p(\L)$ is an isomorphism when restricted to any affine open cover, it suffices to prove that $u_p(\L)$ is an isomorphism in the case where $Y$ is affine and for which there is an isomorphism $\L \cong \O_Y$. In this case, the fact that $u_p(\L)$ is an isomorphism is standard: for example, it follows from the proof of \cite[Corollary III.11.4]{HartshorneAlgebraicGeometry}, using the fact that the pushforward of a coherent sheaf by a proper morphism is coherent \cite[Proposition 30.19.1, Tag 02O5]{StacksProject}.
\end{proof}

\begin{Remark}
Although we will not use this in what follows, \cref{For Proper Birational Maps Unit Map Is Iso on Line Bundles} holds more generally when $\L$ is any vector bundle on $Y$, with a similar proof.
\end{Remark}


Let $X,Y$ be Noetherian schemes and $f \colon Y \to X$ a projective morphism, which we assume factors as a composite \begin{equation}\label{eq:factorfgh}
        \begin{tikzcd}
            Y \ar[dr,"f"'] \ar[rr,"g"] & & Z \ar[dl,"h"] \\
            & X & 
        \end{tikzcd}
    \end{equation}
   for some Noetherian scheme $Z$ and projective morphisms $g, h$. The following results are relative versions of Zariski's lemma on the finite generation of section rings. We include proofs since we have not found references for the results in the generality we require. We recall that a quasi-coherent sheaf $M$ on $Y$ is $f$-generated if the adjunction $f^* f_* M \to M$ is surjective. We say that a line bundle $L$ is $f$-semiample if some power of $L$ is $f$-generated. 

\begin{Lemma}\label{lem:pullbackgeneratedisgenerated}
With the above notation, if $M \in \QCoh(Z)$ is $h$-generated then $L := g^* M$ is $f$-generated.
\end{Lemma}

\begin{proof}
Observe that the diagram     \begin{equation}\label{eq:commutative adjunction diagram}
        \begin{tikzcd}
            f^*h_*(M) \ar[rr, "f^*h_*(c_g(M))"] \ar[d, "\sim"] & & f^*h_*g_*g^*(M) = f^*f_*g^*(M) \ar[d, "u_f(g^*(M))"] \\
           g^*h^*h_*(M) \ar[rr, "g^*(c_h(M))"] & & g^*(M) 
        \end{tikzcd}
    \end{equation} commutes, where the unlabeled arrow is given by functoriality of pullback, $u_g$, respectively $u_h$, is the unit of the adjunction whose left adjoint is $g^*$, respectively $h^*$, and $c_f$ is the counit for the adjunction whose left adjoint is $f^*$. By assumption, $c_h(M)$ is surjective. Therefore, by the right exactness of $g^*$, $g^*(c_h(M))$ is also surjective. Since \labelcref{eq:commutative adjunction diagram} commutes, we deduce that $u_f(g^*(M)) = u_f(L)$ is surjective as well. 

\end{proof}

We note that if $L$ is $f$-ample then it is $f$-semiample since \cite[Lemma 29.39.7, Tag 01VR]{StacksProject} implies that if $L$ is $f$-very ample then it is $f$-generated. The relative version of Zariski's lemma says:

\begin{Lemma}\label{lem:Rsectionringfg}
    If $L$ is $f$-semiample then $\mathcal{B} := \bigoplus_{n \ge 0} f_* L^{\otimes n}$ is a finitely generated sheaf of $\mathcal{O}_X$-algebras. 
\end{Lemma}

\begin{proof}
    This statement is local on $X$, therefore we may assume that $X = \Spec \, A$ is affine. Then we need to show that $\bigoplus_{n \ge 0} H^0(Y,L^{\otimes n})$ is a finitely generated $A$-algebra. By assumption, there exists $m > 0$ such that $f^* f_* L^{\otimes m} \to L^{\otimes m}$ is surjective. Let $N = f_* L^{\otimes m}$ and note that $N$ is a finitely generated $A$-module since $f$ is projective. Let $S = \Sym_A N$, a finitely generated $A$-algebra. The surjection $f^* N \to L^{\otimes m}$ implies that each $F_j = \bigoplus_{n \ge 0} L^{\otimes (nm + j)}$ is generated as a graded $f^* S$-module by a coherent $\mathcal{O}_Y$-submodule. In fact, since 
    \[
    f^* \Sym^n_A N = \Sym^n_{\mathcal{O}_Y} f^*N \twoheadrightarrow \Sym^n L^{\otimes m} = L^{\otimes mn},
    \]
    where the final equality is because $L$ has rank one, $F_j$ is actually generated as a graded $f^* S$-module by $L^{\otimes j}$. Therefore, \cite[Lemma 69.20.4, Tag 08AP]{StacksProject} says that $f_* F_j$ is a finitely generated graded $S$-module. This implies that
    \[
    \bigoplus_{j = 0}^{m-1} f_* F_j = \bigoplus_{n \ge 0} H^0(Y,L^{\otimes n})  
    \]
     is a finitely generated graded $S$-module. It follows that it is a finitely generated graded $A$-algebra. 
\end{proof}

Since $g \colon Y \to Z$ is a projective morphism of Noetherian schemes, it factors through the Stein factorization 
    \begin{equation}\label{eq:gStein}
        \begin{tikzcd}
            Y \ar[dr,"g"'] \ar[rr,"{g'}"] & & Z' \ar[dl,"k"] \\
            & Z & 
        \end{tikzcd}
    \end{equation}
    where $Z' = \Spec_{Z} \, g_* \mathcal{O}_Y$. 

\begin{Proposition}\label{prop:Steinamplefactor}
    Assume that $M$ is a $h$-ample line bundle on $Z$ and let $L = g^* M$. Then there is an isomorphism $Z' \stackrel{\sim}{\longrightarrow} \mathrm{Proj}_X \mathcal{B}$, where $\mathcal{B} := \bigoplus_{m \ge 0} f_* L^{\otimes m}$, such that the diagram 
    \[
     \begin{tikzcd}
     & Y \ar[dr] \ar[dl,"{g'}"'] & \\
     Z' \ar[rr,"\sim"] \ar[dr,"q"'] & & \mathrm{Proj}_X \mathcal{B} \ar[dl] \\
            & X & 
        \end{tikzcd}
    \]
        is commutative, where $q := h \circ k$. 
\end{Proposition}

\begin{proof}
    Recall that, in the Stein factorization \eqref{eq:gStein}, $g'$ is projective with connected fibers and $g_*' \mathcal{O}_Y = \mathcal{O}_{Z'}$, whilst $k$ is a finite morphism \cite[III, Corollary~11.5]{HartshorneAlgebraicGeometry}. Let $M' = k^* M$. Then $M'$ is ample relative to $q$ by \cite[Lemma 29.38.7(2), Tag 0892]{StacksProject}. The projection formula \cite[Lemma 20.54.2, Tag 01E8]{StacksProject} gives an isomorphism
    \begin{align*}
  \bigoplus_{m \ge 0} f_* L^{\otimes m} & = \bigoplus_{m \ge 0} f_* ((g')^* M')^{\otimes m} \cong \bigoplus_{m \ge 0} q_* (g'_* ((g')^* M')^{\otimes m}) \\& \cong \bigoplus_{m \ge 0} q_* ((M')^{\otimes m} \otimes g'_* (\mathcal{O}_Y)) = \bigoplus_{m \ge 0} q_* (M')^{\otimes m},       
    \end{align*}
    since $q$ is a projective morphism. It can be checked that these isomorphisms are compatible with multiplication. Since $M'$ is $q$-ample, there is an isomorphism $Z' \cong \mathrm{Proj}_X  \bigoplus_{m \ge 0} q_* (M')^{\otimes m}$ as schemes over $X$. The claim follows. 
\end{proof}

The converse to Proposition~\ref{prop:Steinamplefactor} says:

\begin{Proposition}\label{prop:semiample-induced-morphism}
Let $f\colon Y\longrightarrow X$ be a projective morphism of Noetherian schemes and \(L\) an
\(f\)-semiample line bundle on $Y$. Then $\mathcal{B} :=
  \bigoplus_{n\geq 0} f_*L^{\otimes n}$ is a finitely generated graded \(\mathcal O_X\)-algebra. If $Z:=\operatorname{Proj}_X \mathcal{B}$ and $h\colon Z\longrightarrow X$, then there is a projective morphism $g\colon Y\longrightarrow Z$ such that the diagram 
\[
\begin{tikzcd}
    Y \ar[dr,"f"'] \ar[rr,"g"] & & Z \ar[dl,"h"] \\
    & X & 
\end{tikzcd}
\]
is commutative. 

Moreover, for sufficiently divisible \(m>0\), the sheaf $M:=\mathcal O_Z(m)$ is invertible, \(h\)-ample and $L^{\otimes m}\simeq g^*M$.
\end{Proposition}

\begin{proof}
The fact that $\mathcal{B}$ is a finitely generated graded \(\mathcal O_X\)-algebra is precisely Lemma~\ref{lem:Rsectionringfg}. Choose \(m>0\), sufficiently divisible, such that \(L^{\otimes m}\) is \(f\)-generated and the \(m\)-th Veronese algebra
\[
  \mathcal{B}^{(m)}
  :=
  \bigoplus_{n\geq0} f_*L^{\otimes mn}
\]
is generated in degree one.  The latter is possible by
\cite[Lemma~10.56.2, Tag~0EGH]{StacksProject} since $X$ is assumed to be Noetherian.

Passing to a Veronese does not change relative Proj, so there is a
canonical isomorphism
\[
  Z
  =
  \operatorname{Proj}_X \mathcal{B}
  \cong
  \operatorname{Proj}_X \mathcal{B}^{(m)};
\]
see \cite[Lemma~27.11.8, Tag~0B5J]{StacksProject}.  Under this
isomorphism, $M:=\mathcal O_Z(m)$ corresponds to the twisting sheaf
\(\mathcal O_{\operatorname{Proj}_X \mathcal{B}^{(m)}}(1)\).  Since
\(\mathcal{B}^{(m)}\) is generated in degree one, this is an invertible
\(h\)-very ample sheaf.  In particular, $M$ is \(h\)-ample.

The graded algebra morphism
\[
  f^*\mathcal{B}^{(m)}
  \longrightarrow
  \bigoplus_{n\geq0}L^{\otimes mn}
\]
has degree-one component $f^*f_*L^{\otimes m} \longrightarrow L^{\otimes m}$, which is surjective because \(L^{\otimes m}\) is \(f\)-generated. The universal property of Proj (see \cite[Lemma~27.12.1, Tag~01N8]{StacksProject}) therefore gives an \(X\)-morphism
\[
  g\colon Y\longrightarrow
  \operatorname{Proj}_X \mathcal{B}^{(m)}
  \cong Z
\]
such that $g^*\mathcal O_Z(m)\cong L^{\otimes m}$. Since \(g\) is an
\(X\)-morphism, $f=h\circ g$. That is, the diagram is commutative. 

Finally, \(h\) is projective and hence separated, while
\(h\circ g=f\) is projective.  Hence \(g\) is projective by
\cite[Lemma~29.44.15, Tag~0C4Q]{StacksProject}.
\end{proof}

In applications, it is useful to note that, under the hypothesis of Proposition~\ref{prop:semiample-induced-morphism}:

\begin{Corollary}\label{cor:semiample-induced-morphism-corollary}
    If, in addition $X,Y$ are integral, $Y$ is normal and $f$ is birational then $Z$ is integral and normal, the morphisms $g,h$ are birational and $g_* \mathcal{O}_Y \cong \mathcal{O}_Z$. 
\end{Corollary}

We recall that a $\mathbb{Q}$-Cartier divisor $D$ is $f$-movable if there exists some $m > 0$ such that $mD$ is a Cartier divisor and the support of the cokernel of $f^* f_* \mathcal{O}(mD) \to \mathcal{O}(mD)$ has codimension at least two in $Y$. The following technical result is required in the proof of Corollary~\ref{cor:MoveconeGIT}.  

\begin{Lemma}\label{lem:technicalnonmovable}
Let $f,g,h$ be as in \eqref{eq:factorfgh} and assume that $Y,Z$ are integral and normal and $g$ birational. Assume, moreover, that $Y$ is $\mathbb{Q}$-factorial. If $E \subset Y$ is a $g$-exceptional prime divisor and $A$ a $h$-ample Cartier divisor on $Z$ then $g^* A + \epsilon E$ is not $f$-movable for all rational $\epsilon > 0$. 
\end{Lemma}

\begin{proof}
    Choose $m > 0$ such that $m \epsilon$ is an integer and $m \epsilon E$ is Cartier (possible since $Y$ is assumed to be $\mathbb{Q}$-factorial). Let $A' = m A$. Since $m \epsilon E$ is an effective Cartier divisor it has a canonical section $s \in H^0(Y,\mathcal{O}_Y(m \epsilon E))$ \cite[Definition~31.15.1,~Tag 01WX]{StacksProject}. Equivalently, $s$ is an injection $\mathcal{O}_Y \hookrightarrow \mathcal{O}_Y(m \epsilon E)$. 
    
    Since $Z$ is assumed normal and $E$ is $g$-exceptional, the composite 
    \[
    \mathcal{O}_Z \stackrel{\sim}{\longrightarrow} g_* \mathcal{O}_Y \stackrel{g_* s}{\longrightarrow} g_* \mathcal{O}_Y(m \epsilon E)
    \]
    is an isomorphism. Indeed, the first isomorphism is a consequence of Zariski's main theorem and the composite is an isomorphism away from the closed set $g(E)$. Since $g(E)$ has codimension at least two in $Z$, which is normal and hence $S_2$, and $g_* \mathcal{O}_Y(m \epsilon E)$ is a reflexive rank one sheaf, it follows that the composite is an isomorphism. 
    
Multiplication by $s$ also defines an injection $j \colon \mathcal{O}_Y(g^* A') \hookrightarrow \mathcal{O}_Y(g^* A' + m \epsilon E)$ and it follows from the previous isomorphism together with the projection formula that $g_* j : g_* \mathcal{O}_Y(g^* A') \to g_* \mathcal{O}_Y(g^* A' + m \epsilon E)$ is an isomorphism. Therefore, we have a commutative diagram 
    \[
    \begin{tikzcd}
        g^* g_* \mathcal{O}_Y(g^* A') \ar[r,"g^* g_* j"] \ar[d] & g^* g_* \mathcal{O}_Y(g^* A' + m \epsilon E) \ar[d,"\phi"] \\
       \mathcal{O}_Y(g^* A') \ar[r,"j"] & \mathcal{O}_Y(g^* A' + m \epsilon E)
    \end{tikzcd}
    \]
    where the top horizontal morphism is an isomorphism and the left vertical morphism is surjective since $g^* A'$ is $g$-generated. This implies that the image of $\phi$ is contained in $s \cdot \mathcal{O}_Y(g^* A')$. But $s$ vanishes along $E$, therefore every local section of the image of $\phi$ also vanishes along $E$. Finally, since $f = h \circ g$, the adjunction $f^* f_* \mathcal{O}_Y(g^* A' + m \epsilon E) \to \mathcal{O}_Y(g^* A' + m \epsilon E)$ factors through $g^* g_* \mathcal{O}_Y(g^* A' + m \epsilon E) \to \mathcal{O}_Y(g^* A' + m \epsilon E)$. Therefore, $E$ is contained in the support of the cokernel of $f^* f_* \mathcal{O}_Y(g^* A' + m \epsilon E) \to \mathcal{O}_Y(g^* A' + m \epsilon E)$. Since $g^* A' + m \epsilon E = m (g^* A + \epsilon E)$ and $m$ was arbitrary subject to $m \epsilon \in \mathbb{Z}$ and $m \epsilon E$ being Cartier, we deduce that $g^* A + \epsilon E$ is not $f$-movable.  
\end{proof}

\section{Geometry of $\widetilde{\mathcal{N}}^P$} In this section, we set some notation and record some results on the geometry of $\tcN^P$ and related varieties. Since it simplifies the exposition, in this section (only) we identify $\LG \cong \LG^*$ and think of $\mathcal{N}$ as a closed subvariety of $\LG$. 

\subsection{Preliminary Geometry of $\tcN^P$} First, we let $\LP \cap \N$ denote the intersection of $\LP$ and $\N$ in $\LG$ \textit{with its reduced scheme structure}. 

\begin{Proposition}\label{Integrality of Intersection}
There is an isomorphism $\LP \cap \N \cong \N_L \times \LU_P$. In particular, the reduced scheme $\LP \cap \N$ is integral.
\end{Proposition}

\begin{proof}
Fix a Levi subalgebra $\mathfrak{l} \subset \mathfrak{p}$ so that $\mathfrak{p} = \mathfrak{l} \oplus \LU_P$, where $\LU_P$ is the nilpotent radical of $\mathfrak{p}$. Let $L \subset P$ be the connected closed subgroup whose Lie algebra is $\mathfrak{l}$. We claim that the diagram
\begin{equation}\label{The Rep Theory Diagram for P}
\begin{tikzcd}
\mathfrak{p} \ar[d] \ar[rr,hook] & & \mathfrak{g} \ar[d] \\
 \mathfrak{l} \ar[r] & \mathfrak{l} \sslash L \ar[r] & \mathfrak{g} \sslash G 
\end{tikzcd}
\end{equation} obtained from the obvious closed embedding and quotient maps commutes. If we denote by $\phi \colon \mathfrak{p} \to \mathfrak{g} \sslash G$ the composite going right then down and $\psi \colon \mathfrak{p} \to \mathfrak{g} \sslash G$ the composite going down then right, then for $f \in \mathcal{O}(\mathfrak{g})^G$ and $x \in \mathfrak{p}$, $\phi^*(f)(x) = f(x)$ and $\psi^*(f)(x) = f(x_{\mathfrak{l}})$ where $x = x_{\mathfrak{l}} + x_{\LU_P}$ is the decomposition in $\mathfrak{p} =  \mathfrak{l} \oplus \LU_P$. We must show that $f(x) = f(x_{\mathfrak{l}})$. We may assume that $f \in \mathcal{O}(\mathfrak{p})^L = (\mathcal{O}(\mathfrak{l}) \oplus \LU_P^* \mathcal{O}(\mathfrak{p}))^L$. Let $h \in \mathrm{Lie}(Z(L))$ such that the eigenvalues of $\mathrm{ad}(h)$ on $\LU_P$ are all strictly positive. Then $(\LU_P^* \mathcal{O}(\mathfrak{p}))^{\mathrm{ad}(h)} = 0$ and hence $\mathcal{O}(\mathfrak{p})^L = \mathcal{O}(\mathfrak{l})^L$. Then it is clear that $f(x) = f(x_{\mathfrak{l}})$.

Observe that both squares in the diagram 
\begin{equation}\label{The Rep Theory Diagram for P Over Quotient of 0}
\begin{tikzcd}
\mathfrak{p} \times_{\mathfrak{l}\sslash L} \Spec(S) \ar[d] \ar[r,hook] & \Spec(S) \ar[d, hook]\ar[r]  & \overline{0} \ar[d, hook] \\
 \mathfrak{p} \ar[r] & \mathfrak{l}\sslash L \ar[r] & \LG\sslash G 
\end{tikzcd}
\end{equation} are Cartesian, where $S := \O(\mathfrak{l})^L \otimes_{\O(\LG)^G} k$. Because of this, we obtain a canonical isomorphism \[\mathfrak{p} \times_{\mathfrak{l}\sslash L} \Spec(S)  \xrightarrow{\sim} \LP \times_{\LG\sslash G} \overline{0}\] from which we deduce \begin{equation}\label{Cartesian Diagram and Reduction}\LP \cap \N := (\LP \times_{\LG \sslash G} \overline{0})^{\mathrm{red}} \xleftarrow{\sim} (\LP \times_{\mathfrak{l}\sslash L} \Spec(S))^{\mathrm{red}}.\end{equation} 

Next, observe that $S$ is a finite extension of $k$: this follows, for example, from the fact that the composite $\LT \to \mathfrak{l}\sslash L \to \LG\sslash G$ is a finite map and \cite[Theorem 11.1.1]{VakilRisingSeaFoundationsofAlgebraicGeometry}. Therefore all positively graded elements $S_+$ of $S$ are nilpotent. We therefore obtain that \begin{equation}\label{Reduced Is Integral}\N_L \times \LU_P \cong (\LP \times_{\mathfrak{l}\sslash L} \Spec(S))^{\mathrm{red}}\end{equation} from the fact that \[\Spec((\O(\LP) \otimes_{\O(\mathfrak{l})^L} S)/S_+) \cong \LP \times_{\mathfrak{l}\sslash L} \overline{0} \cong \N_L \times \LU_P\] is reduced. Combining \labelcref{Cartesian Diagram and Reduction} and \labelcref{Reduced Is Integral} we deduce that $\LP \cap \N \cong \N_L \times \LU_P$ as desired.
%
\end{proof}

Let $\tcN^P := G \times^P (\LP \cap \N).$ Observe that an inclusion $Q \subseteq P$ of parabolic subgroups naturally induces a map $\nu_{Q}^{P}: \tcN^{Q} \to \tcN^{P}$. In particular, if $Q = B$ we have constructed a map $\nu_B^{P}: \widetilde{\mathcal{N}} = \widetilde{\mathcal{N}}^{B} \to \widetilde{\mathcal{N}}^{P}$ which we will also denote by $\nu^{P}$. 

We now record some results on the geometry of $\tcN^P$, which are likely well known, but which we were unable to locate proofs in the literature outside of the case $B = P$ \cite{SpringerTheUnipotentVarietyofaSemisimpleGroup} and $P = G$ \cite{KostantLieGroupRepresentationsonPolynomialRings}:

\begin{Lemma}\label{tcNP is normal} For any parabolic subgroup $P$, we have:
\begin{enumerate}
    \item The scheme $\tcN^P$ is integral, separated, quasiprojective, and of finite type over $k$. 
    \item The dimension of $\tcN^P$ is equal to that of the nilpotent cone. 
    \item The maps $\nu^P_{Q}$ are projective and birational for any parabolic $Q \subseteq P$.
    \item The variety $\tcN^P$ is normal and a local complete intersection; in particular, $\tcN^P$ is Gorenstein and Cohen-Macaulay.
    \item The variety $\tcN^P$ is $\mathbb{Q}$-factorial. 
\end{enumerate}
\end{Lemma}

\begin{proof}[Proof of \cref{tcNP is normal}]
If $P = B$, then $\tcN^P = \tcN = T^*(G/B)$ is smooth and therefore normal. This shows claims (1), (2), (4), and (5) for $P = B$. 

Now assume $P \neq B$. Let $\overline{U}_P$ be the unipotent radical of the parabolic subgroup containing the Borel $\overline{B}$ in opposite position to $B$ and containing $T$. Observe that there is an open cover given by the big cell $\overline{U}_P \cong \overline{U}_PP/P$, together with finitely many $G$-translates of $\overline{U}_P$, such that the preimage of these open sets under the projection map are isomorphic to $\overline{U}_P \times \LP \cap \N$. The scheme $\overline{U}_P \times \LP \cap \N$ is quasicompact and of finite type since it is the product of two finite type affine $k$-schemes and thus is a finite type affine $k$-scheme. Moreover, since the product of integral schemes (over our algebraically closed field $k$) is integral, we obtain from \cref{Integrality of Intersection} that the product $\overline{U}_P  \times \LP \cap \N$ is integral. This proves $\tcN^P$ is integral, and of finite type over $k$. 

We now prove that $\tcN^P$ is quasiprojective and separated. To this end, we claim that the canonical map $\varphi: \tcN^P \xrightarrow{} G/P \times \mathfrak{g}$ given by the projection and the action map is a closed embedding. Since we may cover $\tcN^P$ by $G$-translates of the big cell, it suffices to verify that the restriction $\varphi|_{\overline{U}_P \times \LP \cap \mathcal{N}}$ is a closed embedding. But observe that $\varphi_{\overline{U}_P \times \LP \cap \N}$ factors as the composite \[\overline{U}_P \times \mathfrak{p} \cap \mathcal{N} \xrightarrow{\mathrm{id} \times \iota} \overline{U}_P \times \mathfrak{g} \xrightarrow{(u, \xi) \mapsto (u, u\xi)} \overline{U}_P \times \mathfrak{g}\] where $\iota$ is the inclusion map. The second map in this composite is an isomorphism (with inverse $(u, \nu) \mapsto (u, u^{-1}\nu)$) and the first map is evidently a closed embedding, so we see that $\varphi|_{\overline{U}_P \times \LP \cap \mathcal{N}}$ is a closed embedding. This proves that $\varphi$ is a closed embedding. Since $G/P$ itself is projective (see, for example, the discussion below the proof of \cite[Corollary 21.3.A]{HumphreysLinearAlgebraicGroups}), it follows that $\tcN^P$ is a closed subscheme of a quasiprojective scheme and therefore is quasiprojective. This proves (1).

Next, we compute the dimension of $\tcN^P$. Observe that we have a canonical map \[\LP \cap \N \to \N_L = \N \cap \mathfrak{l}\] induced by the quotient map $\LP \to \mathfrak{l}$ and that the fiber of this map over any $x \in \N_L$ is $\mathfrak{u}_P$. We obtain 
\[\mathrm{dim}(\LP \cap \N) = \mathrm{dim}(\N_L) + \mathrm{dim}(\LUP) = 2\, \mathrm{dim}(\mathfrak{u}_L) + \mathrm{dim}(\mathfrak{u}_P),
\]
by \cite[Corollary 12.4.2]{VakilRisingSeaFoundationsofAlgebraicGeometry}, and \cite[Theorem 0.7]{KostantLieGroupRepresentationsonPolynomialRings} respectively. Therefore,
\[
    \mathrm{dim}(\tcN^P) = \mathrm{dim}(G/P) + \mathrm{dim}(\N_L) + \mathrm{dim}(\LUP) = 2\, \mathrm{dim}(\LUP) + \mathrm{dim}(\N_L) = 2(\mathrm{dim}(\LUP) + \mathrm{dim}(\mathfrak{u}_L)) = 2\, \mathrm{dim}(\mathfrak{u})
\]
by, say, the parabolic Bruhat decomposition, \cite[Theorem 0.7]{KostantLieGroupRepresentationsonPolynomialRings} for the reductive group $L$, and the fact that $\LU = \LU_L \oplus \LU_P$, respectively. This shows (2).

We now show (3), and first show that $\nu^P_{Q}$ is projective for any parabolic subgroups $Q \subseteq P$. Observe that $\nu_{Q}^{G}$ factors as the composite \[\tcN^{Q} := G \times^{Q} (\mathfrak{q} \cap \N) \xhookrightarrow{} G/{Q} \times \mathcal{N} \xrightarrow{\mathrm{pr}_\N} \N\] of a closed embedding and the projection map, which is projective since $G/{Q}$ is projective. Therefore $\nu^G_{Q}$ is projective for any parabolic $Q$. 

We next claim that the (relative) diagonal morphism $\Delta_{\nu^G_P}: \tcN^P \to \tcN^P \times_{\N} \tcN^P$ is a closed embedding. To see this, first observe that the canonical map $f: \tcN^P \times_{\N} \tcN^P \to \tcN^P \times \tcN^P$ is a closed embedding. Indeed, this follows from the standard fact that \begin{equation}
\begin{tikzcd}
\tcN^P \times_{\N} \tcN^P \ar[d] \ar[r, "f"] & \tcN^P \times \tcN^P \ar[d, "\nu^G_P \times \nu^G_P"]  \\
\N \arrow[r, "\Delta_\N"] &  \N \times \N
\end{tikzcd}
\end{equation} is Cartesian (see for example \cite[Exercise 1.2.S]{VakilRisingSeaFoundationsofAlgebraicGeometry}) as well as the fact that $\Delta_\N$ is a closed embedding (since $\N$ is affine) and that the property of a morphism being a closed embedding is stable under base change. Moreover, since $\tcN^P$ is separated by our above analysis, the diagonal morphism $\Delta: \tcN^P \to \tcN^P \times \tcN^P$ is a closed embedding. Now observe that $f \Delta_{\nu^G_P} = \Delta$. We have shown that $f$ is a closed embedding; in particular, the diagonal morphism $\Delta_f$ associated to $f$ is a closed embedding since all affine morphisms are separated. This fact, as well as the fact that $\Delta$ itself is a closed embedding, implies that $\Delta_{\nu^G_P}$ is a closed embedding by \cite[Theorem 11.1.1]{VakilRisingSeaFoundationsofAlgebraicGeometry}. Finally, since $\nu_{P}^G\nu^P_{Q} = \nu_{Q}^G$, $\nu_{Q}^G$ is projective by our above analysis, and $\Delta_{\nu^G_P}$ is a closed embedding (and, in particular, projective), we obtain that $\nu_{Q}^{P}$ is projective, by \cite[Exercise 17.3.C]{VakilRisingSeaFoundationsofAlgebraicGeometry}. 

We next prove that $\nu^{P}_{Q}$ is birational in the special case where $Q = B$. When $P = G$, we recall that birationality is a celebrated result of Springer \cite{SpringerTheUnipotentVarietyofaSemisimpleGroup}. 
Now if $Q$ is a standard parabolic subgroup distinct from $G$, we can write $\nu_B^G = \nu_{Q}^{G}\nu_{B}^{Q}$. Since $\nu_B^G$ is a dominant morphism of integral schemes, the morphism $\nu_{B}^{Q}$ is dominant as well. Therefore $\nu_{B}^{Q}$ induces a map on the respective function fields. Since $\nu_B^G$ is birational, the generic point of $\mathcal{N}$ is in the image of $\nu_B^G$. Since $\nu^G_B = \nu_{Q}^{G}\nu_{B}^{Q}$, the generic point of $\mathcal{N}$ lies in the image of $\nu_B^{Q}$. Since the dimensions of the spaces $\tcN, \tcN^{Q}$, and $\tcN^G = \mathcal{N}$ all agree, we see that the \textit{only} point of $\tcN^{Q}$ that can map to the generic point of $\mathcal{N}$ is the generic point of $\tcN^{Q}$, and so $\nu_{Q}^G$ is dominant as well. Therefore $\nu_{Q}^{G}$ and $\nu_{B}^{Q}$ induce maps on the respective function fields. Since $\nu^G_B = \nu_{Q}^{G}\nu_{B}^{Q}$ and $\nu_B^G$ induces an isomorphism of function fields, the maps $\nu_{Q}^G$ and $\nu_B^{Q}$ induce isomorphisms of function fields as well. Therefore, $\nu_B^{Q}$ and $\nu_{Q}^G$ are birational. Repeating the same argument but replacing the birationality of $\nu_B^G$ with the birationality of $\nu_B^{Q}$ shows that $\nu^{Q}_P$ is also birational.


Next, we show (4). Since $\tcN^P$ is Zariski locally isomorphic to the product $U \times \N_L \times \LU_P$ for some open subset $U \subseteq G/P$, we deduce that $\tcN^P$ is normal from the fact that $\N_L$ is normal \cite[Theorem 0.8]{KostantLieGroupRepresentationsonPolynomialRings}. By the same argument, we deduce from the fact that $\N_L$ is a complete intersection \cite[Theorem 0.7]{KostantLieGroupRepresentationsonPolynomialRings} that $\tcN^P$ is a local complete intersection. One can also use Stein factorization together with the fact that there is an isomorphism $\O_{\tcN^P} \xrightarrow{\sim} \nu_*^P(\O_{\tcN})$ by \cite[Lemma 3.2(2)]{ChanRiderSobajeExoticTStructureForPartialResolutionsofNilpotentCone} to deduce that $\tcN^P$ is normal.


 As in the previous paragraph, to show that $\tcN^P$ is $\mathbb{Q}$-factorial it suffices to show that $\mathcal{N}_L$ is $\mathbb{Q}$-factorial. Since the nilpotent cone of any split reductive group equals the nilpotent cone of its derived subgroup, we may assume $L$ is semisimple. Since $\mathcal{N}_L$ is an affine cone, its Picard group is trivial, see for example \cite[Lemma 7.1.1]{LosevMasonBrownMatvieievskyiUnipotentIdealsandHarishChandraBimodules}. Therefore, we need to show that its class group is finite. This class group is isomorphic to the class group of the regular nilpotent orbit since the irregular nilpotent elements have codimension two in $\N_L$. Write the regular orbit as $L/L_x$ where $x$ is some regular nilpotent element. Since $L/L_x$ is smooth, the class group and Picard group of $L/L_x$ agree. Moreover, we have an exact sequence \begin{equation}\label{Picard SES}X^{\bullet}(L_x) \to \mathrm{Pic}(L/L_x) \to \mathrm{Pic}(L)\end{equation} by for example \cite[Theorem 18.32]{MilneAlgebraicGroupsBook}. Since $L$ is semisimple, its Picard group is finite. Moreover, by \cite[Theorem 4.11]{SpringerSomeArithmeticalResultsonSemisimpleLieAlgebras}, $L_x$ is the product of a unipotent group $U'$ and the center of $L$. Therefore $X^{\bullet}(L_x) = X^{\bullet}(Z(L))$ is finite since $Z(L)$ is a finite group. 
\end{proof}

We refer the reader to \cite{HuKeel} for the definition of a Mori dream space and \cite{Ohta} for the relative version. We note that $\tcN \to \mathcal{N}$ is a (relative) Mori dream space because it is a projective symplectic resolution \cite{NamikawaPoissonDeformationsAndBirationalGeometry}. As a consequence of Lemma~\ref{tcNP is normal}, we note that:

\begin{Proposition}\label{prop:Moridreamspace}
    The variety $\tcN^P \to \N$ is a (relative) Mori dream space. 
\end{Proposition}

\begin{proof}
First we note that both $\tcN$ and $\tcN^P$ are normal and $\Q$-factorial by \cref{tcNP is normal}. Since $\tcN$ is a (relative) Mori dream space, the claim follows from \cite[Theorem~1.3]{Ohta} since both the morphisms $\tcN \to \tcN^P$ and $\tcN^P \to \mathcal{N}$ are \textit{algebraic fibre spaces}, meaning that they are projective morphisms $f \colon X \to Y$ with $f_* \mathcal{O}_X = \mathcal{O}_Y$. The latter condition holds in our case because the target is normal and $f$ is birational and proper by Lemma~\ref{tcNP is normal}.     
\end{proof}

It is very unusual that all crepant partial resolutions of a symplectic singularity are (relative) Mori dream spaces. Even if a symplectic singularity admits a projective symplectic resolution, the crepant partial resolutions are usually not $\mathbb{Q}$-factorial and hence cannot be (relative) Mori dream spaces.

\subsection{An Explicit Divisorial Contraction} Fix a simple coroot $\alpha^{\vee}$ and let $P := P_{\alpha^{\vee}}$ denote the parabolic associated to the subset $\{\alpha^{\vee}\}$. Let $\mathfrak{u}_{\alpha^{\vee}}$ denote the Lie algebra of the unipotent radical of $P_{\alpha^{\vee}}$, or equivalently the nilradical of $\LP := \mathrm{Lie}(P)$. In this case, $[L,L]$ has rank one and hence the singular locus of $\tcN^{P}$ equals $G \times^P \mathfrak{u}_{\alpha^{\vee}}$, which is precisely the image of the irreducible divisor 
\[
D_{\alpha^{\vee}} := G \times^B \mathfrak{u}_{\alpha^{\vee}} \subset \tcN
\]
under the proper map $\nu^{P}$. 

\begin{Proposition}\label{Quotient Map is Divisorial Contraction}
    The map $\nu^{P}$ contracts the divisor $D_{\alpha^{\vee}}$. 
\end{Proposition}

\begin{proof}
Since our parabolic subgroup has semisimple rank one, the closed subset $\mathfrak{u}_{\alpha^{\vee}} \subseteq \LU$ has codimension one. Therefore $G \times^B \mathfrak{u}_{\alpha^{\vee}} \subseteq G \times^B \LU$ has codimension one. Moreover, $\nu^{P}$ fits into the following commutative diagram: \begin{equation*}\xymatrix@R+2em@C+2em{G \times^B \mathfrak{u}_{\alpha^{\vee}} \ar[r]^{\subseteq} \ar[d] & G \times^B \LU \ar[d]_{\nu^{P}} \\
G \times^P \mathfrak{u}_{\alpha^{\vee}} \ar[r]^{\subseteq} &  G \times^P (\LP \cap \mathcal{N}) 
  }\end{equation*} where the unlabeled arrow is the natural quotient map. The natural quotient map is a smooth morphism of relative dimension 1 $(= \mathrm{dim}(P/B) = \mathrm{dim}(\mathbb{P}^1)$) so we see that the divisor $G \times^B \mathfrak{u}_{\alpha^{\vee}}$ is contracted onto a subscheme of smaller dimension. Since the dimension of the domain and codomain of $\nu^P$ agree by \cref{tcNP is normal}(2), our claim follows.
  \end{proof}

We noted in the proof of Lemma~\ref{tcNP is normal}(5) that the nilpotent cone is $\Q$-factorial. Therefore, van der Waerden's purity theorem says that every irreducible component of the exceptional locus of $\nu \colon \tcN \to \mathcal{N}$ is a divisor.  

\begin{Proposition}\label{Divisors of tcN and Which Are Contracted by nuP} The irreducible codimension one components of the exceptional divisor of $\nu \colon \tcN \to \mathcal{N}$ are precisely the closed subschemes $G \times^B \LU_{\alpha^{\vee}}$ for $\alpha^{\vee} \in \Pi^{\vee}$. If $Q$ is an arbitrary parabolic subgroup of $G$ then the image of $G \times^B \LU_{\alpha^{\vee}}$ under $\nu_B^Q$ has codimension one if and only if $\alpha^{\vee} \notin S_Q$.
\end{Proposition}

We prove this after proving two Lemmas:

\begin{Lemma}\label{G Invariant Divisors on tcN}
    The $G$-invariant (prime) divisors on $\tcN$ are in bijective correspondence with the simple coroots of $G$: more precisely, to any simple coroot $\alpha^{\vee}$, the subscheme $G \times^B \mathfrak{u}_{{\alpha^{\vee}}}$ is a $G$-invariant divisor and any $G$-invariant divisor has this form.
\end{Lemma}

\begin{proof}
Recall that $p \colon \tcN = G \times^B \LU \to G/B$. We identify $\LU$ with the closed subscheme $p^{-1}(1B)$ of $\tcN$. Then, to any integral $G$-invariant codimension one subscheme $Z$ of $\tcN$, we can construct a subscheme $Z \cap \LU$ by taking the reduced subscheme of $Z \times_{\tcN} \, \LU$. Conversely, to any integral $B$-invariant codimension one subscheme $Z'$ of $\LU$, there is a natural closed embedding $G \times^B Z' \to \tcN = G \times^B \LU$; the fact that this is a closed embedding can be proved on the open cell and its $G$-translates. We claim that this gives a bijection: \[\{G\text{-stable divisors on }\tcN\}\leftrightarrow \{B\text{-stable divisors on }\LU\}.\]

We now prove this. Assume $Z$ is an integral nonempty $G$-invariant codimension one subscheme of $\tcN$. We first claim that the induced map $\dot{p}: Z \to G/B$ is flat. Indeed, by $G$-equivariance, this map is dominant, and so by generic flatness (\cite[Proposition 29.27.1, Tag 04PW]{StacksProject}) there exists some nonempty open subscheme $X \subseteq G/B$ such that $\dot{p}|_{\dot{p}^{-1}(X)}$ is flat, and by $G$-equivariance we have that $\dot{p}|_{\dot{p}^{-1}(gX)}$ is also flat for any $g \in G(k)$. Since $\dot{p}^{-1}(gX)$ is an open cover of $Z$ and flatness is local on the target, we see that indeed $\dot{p}$ is flat. In particular, $\dot{p}^{-1}(1B)$ has dimension $\mathrm{dim}(\LU) - 1$, and is some $B$-invariant codimension one subscheme of $\LU \cong p^{-1}(1B)$.

By our above analysis, the map $G \times^B (Z \cap \LU) \to Z$ is a closed embedding. Moreover, $\dim(G \times^B (Z \cap \LU)) = \dim(Z)$ and so this map is in fact an isomorphism. This can be used to show that $Z \cap \LU$ is integral: since $G \times^B (Z \cap \LU)$ is locally isomorphic to $G/B \times (Z \cap \LU)$ it is reduced as it is locally isomorphic to the product of two reduced schemes \cite[Lemma 28.3.2, Tag 01OL]{StacksProject}. Moreover, the $G$-translates of $\dot{p}^{-1}(Bw_0B/B)$ give a nonempty open cover on which $Z$ is locally isomorphic to $G/B \times Z \cap \LU$ and so $Z \cap \LU$ must be integral as well by \cite[Lemma 28.3.3, Tag 01OM]{StacksProject}. Thus $Z \cap \LU$ is integral \cite[Lemma 28.3.4, Tag 01ON]{StacksProject}. Conversely, it is not difficult to check directly that $\LU \cap (G \times^B Z') = Z'$ for any $B$-stable divisor of $\LU$. This gives our desired bijection; it remains to classify the $B$-stable divisors on $\LU$. 

Let $Z'$ denote any $B$-stable integral codimension one subscheme of $\LU$. Then, since the regular elements of $\LU$ are an open $B$-orbit \cite[Lemma 3.2.12]{ChrissGinzburgRepresentationTheoryandComplexGeometry}, $Z'$ is a closed subscheme of $\LU\setminus \LU_{\mathrm{reg}}$. Choose a basis of root vectors in $\LU$, and let $e_{\alpha^{\vee}}^*$ denote the dual basis. Then $\LU\setminus \LU_{\mathrm{reg}} = \cup_{\alpha^{\vee} \in \Pi^{\vee}}V(e_{\alpha^{\vee}}^*)$ by \cite[Theorem 5.3]{KostantThePrincipalThreeDimensionalSubgroupandtheBettiNumbersofaComplexSimpleLieGroup}. Therefore we see that $Z' \subseteq \cup_{\alpha^{\vee} \in \Pi^{\vee}}V(e_{\alpha^{\vee}}^*)$ and so, by irreducibility of $Z'$, $Z' = V(e_{\alpha^{\vee}}^{*})$ for some simple coroot $\alpha^{\vee}$. 

Thus, from our bijection above, we see that any $G$-stable divisor on $\tcN$ has our desired form.
\end{proof}

\begin{Lemma}\label{Acting on Regular Elements in u by Ubar is Open embedding}
For any reductive group $G$ with a choice of Borel subgroups $B, \overline{B}$ in opposite position and $U, \overline{U}$ their respective unipotent radicals, the action map $\overline{U} \times \LU^{\mathrm{reg}} \to \N$ is an open embedding.
\end{Lemma}

\begin{proof}
Let $e$ denote a regular nilpotent element contained in $\LU$. Since $e$ is regular, the action map induces an open embedding $G/G_e \to \N$, where $G_e$ is the stabilizer of $e$. It is known (see for example \cite[Chapter 3, \S 1.14(a)]{SpringerSteinbergConjugacyClasses}) that $G_e \subseteq B$. Therefore $G/G_e$ contains the open subset $\overline{U} \times B/G_e$. The diagram \begin{equation*}\xymatrix@R+2em@C+2em{\overline{U} \times B/G_e \times \{e\} \ar[r]^{\mathrm{id} \times \mathrm{act}} \ar[d]^{\subseteq} & \overline{U} \times \LU_{\mathrm{reg}} \ar[d]^{\mathrm{act}} \\
G/G_e \times \{e\} \ar[r]^{\mathrm{act}} & \N
  }\end{equation*} obviously commutes, and $\mathrm{id} \times \mathrm{act}$ is an isomorphism by say \cite[Lemma 3.2.12]{ChrissGinzburgRepresentationTheoryandComplexGeometry}. Therefore, since $\mathrm{act}$ is the composite of the inverse of an isomorphism and two open embeddings, it is an open embedding. 
\end{proof}
\newcommand{\orbit}{\mathbb{O}}
\begin{proof}[Proof of \cref{Divisors of tcN and Which Are Contracted by nuP}]
Let $Z$ be an irreducible codimension one subscheme contained in the exceptional set. Observe that $Z$ is $G$-invariant: if not, the image of $G \times Z$ in $\tcN$ of the map $(g,z) \mapsto g z$ would be dense since it is irreducible and properly contains $Z$. But this implies that there exists $(g,z)$ such that $\nu(g z) \in \mathcal{O}_{\mathrm{reg}} \subset \mathcal{N}$ since $\nu$ is birational. Since $\nu$ is $G$-equivariant, this implies that $\nu(z) \in \mathcal{O}_{\mathrm{reg}}$, contradicting the fact that $\nu$ is an isomorphism over $\mathcal{O}_{\mathrm{reg}}$. Therefore, $Z$ is one of the irreducible codimension one $G$-invariant subschemes of $\tcN$. By \cref{G Invariant Divisors on tcN}, therefore, set theoretically we have $Z = G \times^B \LU_{{\alpha^{\vee}}}$ for some $\alpha^{\vee}$. Conversely, \cref{Quotient Map is Divisorial Contraction} implies that for $P = P_{\alpha^{\vee}}$, the morphism $\nu^P$ contracts $G \times^B \LU_{{\alpha^{\vee}}}$ and so $\nu = \nu_P^G\nu^P_B$ contracts $G \times^B \LU_{{\alpha^{\vee}}}$ as well. This shows the first claim. 

Assume that $\alpha^{\vee} \in S_Q$. Then, since $\nu_B^Q = \nu_{P}^Q\nu^{P}_B$ we see that $\nu_B^Q$ contracts $G \times^{B}\LU_{{\alpha^{\vee}}}$ by \cref{Quotient Map is Divisorial Contraction}. 

Now assume $\alpha^{\vee} \notin S_Q$. We will show that \[X_{\alpha^{\vee}} := \nu_B^Q(G \times^B \LU_{{\alpha^{\vee}}}) = G \times^Q \orbit\] has codimension one in $\tcN^Q$, where $\orbit$ is the $Q$-saturation of $\LU_{{\alpha^{\vee}}} \subseteq \N \cap \mathfrak{q}$. More specifically, we will show that \begin{equation}\label{Dimension Inequality for Xalpha}
    \mathrm{dim}(X_{\alpha^{\vee}}) \geq 2\mathrm{dim}(U) - 1
\end{equation} from which our codimension claim will follow since we certainly have \[\mathrm{dim}(X_{\alpha^{\vee}}) \leq 2\mathrm{dim}(U) - 1\] as $X_{\alpha^{\vee}}$ is the image of a subscheme whose dimension is exactly $2\mathrm{dim}(U) - 1$. 

Observe that there is an open cover of $X_{\alpha^{\vee}}$ given by open subsets of the form $\overline{U}_Q \times \orbit$ where $\overline{U}_Q$ is the unipotent radical of the parabolic subgroup containing the Borel $\overline{B}$ in opposite position to $B$ and containing $T$. To prove \labelcref{Dimension Inequality for Xalpha}, we will give a lower bound on the dimension of $\orbit$. 

Write $\mathfrak{q} = \LL \oplus \LU_{Q}$, where $\LL$ is the Lie algebra of some choice of Levi subgroup for $Q$. Since $\LU_{\alpha^{\vee}} \cap \LL = \LU_{\alpha^{\vee}} \cap \LU_L$, where $\LU_L$ is the Lie algebra of the unipotent radical of $B \cap L$, it induces a direct sum decomposition $\LU_{{\alpha^{\vee}}} = (\LU_{\alpha^{\vee}} \cap \LU_L) \oplus (\LU_{{\alpha^{\vee}}} \cap \LU_{Q})$. Since $\alpha^{\vee} \notin S_Q$, the open subset $Y \subseteq \LU_{{\alpha^{\vee}}}$ of those $y \in \LU_{{\alpha^{\vee}}}$ for which the projection onto $\LU_{L}$ lies in $\N_L^{\mathrm{reg}}$ is nonempty. Moreover, we observe that the action map $a: L \times \LU_{{\alpha^{\vee}}} \to \mathfrak{q}$ respects our above direct sum decomposition in the sense that the rightmost square in the diagram \begin{equation}\label{Summary Diagram for Direct Sum Decomposition of p}\xymatrix@R+2em@C+2em{\overline{U}_{L} \times Y \ar[r]^{\subseteq} \ar[d]^{=} & \overline{U}_{L} \times \LU_{P} \ar[r]^{a} \ar[d]^{=} & \mathfrak{q} \ar[d]^{=} \\
\overline{U}_{L} \times (\LU_{\alpha^{\vee}} \cap \LU_L^{\mathrm{reg}}) \oplus (\LU_{{\alpha^{\vee}}} \cap \LU_{Q}) \ar[r]^{\subseteq} & \overline{U}_{L} \times (\LU_{\alpha^{\vee}} \cap \LU_L) \oplus (\LU_{{\alpha^{\vee}}} \cap \LU_{Q}) \ar[r] &  \LL \oplus \LU_Q
  }\end{equation} commutes, where the unlabeled map is $(g, \xi_{\LL}, \xi_{\LU_Q}) \mapsto (g\xi_{\LL}, g\xi_{\LU_Q})$. By \cref{Acting on Regular Elements in u by Ubar is Open embedding}, then, the composite $\varphi$ of the lower two horizontal arrows is injective. 
  Therefore, this generic injectivity guarantees \[\mathrm{dim}(\orbit) \geq \mathrm{dim}(\mathrm{im}(a)) \ge \mathrm{dim}(\mathrm{im}(\varphi)) = \mathrm{dim}(\overline{U}_L \times Y) = \mathrm{dim}(U_L) + \mathrm{dim}(U_{{\alpha^{\vee}}})\] by say \cite[Corollary 12.4.2]{VakilRisingSeaFoundationsofAlgebraicGeometry}. Therefore \[\mathrm{dim}(X_{\alpha^{\vee}}) = \mathrm{dim}(\overline{U}_Q) + \mathrm{dim}(\orbit) \geq \mathrm{dim}(\overline{U}_Q) + \mathrm{dim}(U_L) + \mathrm{dim}(U_{{\alpha^{\vee}}}) = \mathrm{dim}(U) + \mathrm{dim}(U_{{\alpha^{\vee}}}) = 2\mathrm{dim}(U) - 1\] which proves \labelcref{Dimension Inequality for Xalpha}, as desired.
\end{proof}

%
%
%
\begin{Corollary}\label{cor:divisorialcontraction}
If $P \subsetneq Q$ then $\nu_P^{Q}$ has divisorial exceptional locus. More precisely, for any parabolic subgroups $P \subseteq Q$ containing $B$ and any $\alpha^{\vee} \in S_{Q}\setminus S_P$, the subscheme $\nu_B^P(G \times^B \LU_{{\alpha^{\vee}}})$ is an irreducible codimension one subscheme of $\tcN^P$ which is contracted by $\nu^{Q}_P$. 
\end{Corollary}

\begin{proof}
With the above notation, the fact that $\nu_B^P(G \times^B \LU_{{\alpha^{\vee}}})$ has codimension one follows immediately from \cref{Divisors of tcN and Which Are Contracted by nuP}. Moreover, the fact that $\nu_P^{P'}(\nu_B^P(G \times^B \LU_{{\alpha^{\vee}}}))$ is not a divisor follows immediately from the equality 
$\nu_P^{Q}\nu_B^P = \nu_B^{Q}$ and \cref{Divisors of tcN and Which Are Contracted by nuP}.
\end{proof}

We recall that a proper birational morphism $f \colon Y \to X$ is a \textit{divisorial contraction} if the exceptional locus $\mathrm{Exc}(f) \subset Y$ is an irreducible divisor. \cref{cor:divisorialcontraction} shows that for $P \subsetneq Q$, the morphism $\nu_P^{Q}$ is a divisorial contraction if and only if $|S_Q \setminus S_P| = 1$. 

\section{Geometric Invariant Theory for $\overline{T^*(G/U)}$}

\subsection{The GIT fan}

Let $T$ be a torus acting on an affine variety $X$ and let $\theta \in \characterlatticeforT$ be a character of $T$. A function $f \in \mathcal{O}(X)$ is said to be $T_{\theta}$-invariant if $t \cdot f = \theta(t) f$ for all $t \in T$. The character $\theta$ is \textit{effective} if there exists an $n > 0$ and non-zero $T_{n\theta}$-invariant function on $X$. Given a closed point $x \in X$ and $\theta \in \characterlatticeforT$, we say that $x$ is $\theta$\textit{-semistable} if there is an $n > 0$ and $T_{n\theta}$-invariant function on $X$ such that $f(x) \neq 0$; we will use the notation $X^{\theta}$ or $X^{\theta\text{-ss}}$ to denote the semistable points. A point $x$ is said to be \textit{stable} if it is semistable, the $T$-stabilizer of $x$ is finite, and the orbit of $x$ is closed in  $X^{\theta\text{-ss}}$. We say that $\theta$ is \textit{generic} if every $\theta$-semistable point is $\theta$-stable. 

Let $\omega(X) \subset \characterlatticeforT_{\R}$ be the weight cone of $X$, given by the $\R^{\geq 0}$-span of those $\lambda \in \characterlatticeforT$ such that $\mathcal{O}(X)_{\lambda} \neq 0$. An element $\theta \in \characterlatticeforT_{\R}$ is \textit{integral} if $\theta \in \characterlatticeforT$. 
Notice that $\theta \in \characterlatticeforT$ is effective if and only if $\theta \in \omega(X)$. Let $\theta \in \omega(X)$. The set
\[
\left\{ \sigma \in \characterlatticeforT \;\middle|\; X^{\sigma} = X^{\theta} \right\}
\]
is the set of integral points in the relative interior of a rational polyhedral cone in $\characterlatticeforT_{\R}$. The \textit{GIT fan} $\Sigma(X) \subset \characterlatticeforT_{\R}$ of $X$ is the union of all such cones as $\theta$ varies over all integral elements in $\omega(X)$. By \cite[Theorem~3.2]{AH}, the cone $\omega(X)$ is the support of the GIT fan, which consists of only finitely many cones. The relative interior $\mathrm{Int}(C)$ of a GIT cone $C$ is said to be a \textit{GIT chamber} if all integral points in $\mathrm{Int}(C)$ are generic. Let $\theta$ be integral and 
\[
Y_{\theta} := X^{\theta\text{-ss}} \sslash T = \mathrm{Proj} \bigoplus_{n \ge 0} \mathcal{O}(X)_{n \theta}.
\]
Since $X$ is a variety, so too is each $Y_{\theta}$. There exists $m > 0$ such that the Serre twisting sheaf $\mathcal{O}(m)$ is an ample line bundle on $Y_{\theta}$ (see \cite[Lemma 27.10, Tag 01MW]{StacksProject}, \cite[Lemma 28.27.12, Tag 01Q2]{StacksProject}) and we set $\mathcal{O}_{Y_{\theta}}(1) = \frac{1}{m} \mathcal{O}(m) \in \mathrm{Pic}(Y_{\theta})_{\mathbb{R}}$, a \textit{fractional line bundle}. Inclusion of cones defines a partial ordering $\le$ on integral points of $\omega(X)$ such that $X^{\theta} \subseteq X^{\lambda}$ if and only if $\lambda \le \theta$. If $\lambda \le \theta$ then the embedding $X^{\theta} \subseteq X^{\lambda}$ induces a projective morphism (\textit{variation of GIT}) $\xi_{\theta}^{\lambda} \colon Y_{\theta} \to Y_{\lambda}$. 

Note that the set of one parameter subgroups of $T$ that act trivially on $X$ is $\omega(X)^{\perp}$. Therefore, if we assume that $T$ acts faithfully then $\omega(X)^{\perp} = 0$, meaning that $\omega(X)$ is a full dimensional cone in $\characterlatticeforT_{\R}$.  In particular, there must exist a GIT cone $C$ whose dimension equals $\dim \characterlatticeforT_{\R}$. 

\begin{Lemma}\label{Elements in Interior of GIT Cone Are Generic}
If $\dim C = \dim \characterlatticeforT_{\R}$ then the integral elements in the interior of $C$ are generic. In other words, if $\theta \in \mathrm{Int}(C)$ then every $\theta$-semistable point is $\theta$-stable. 
\end{Lemma}

\begin{proof}
    Let $\theta \in \mathrm{Int}(C)$ and $x \in X^{\theta}$. If $\mathrm{Stab}_{T}(x)$ is not finite then we can find a one-parameter subgroup $\lambda \in X_{\bullet}(T)$ such that $\lambda(t) x = x$ for all $t$. But since $\dim C = \dim \characterlatticeforT_{\Q}$, there exists $\theta_0 \in \mathrm{Int}(C)$ such that $\langle \lambda, \theta_0 \rangle \neq 0$. Choose such a $\theta_0$. Since $X^{\theta} = X^{\theta_0}$, there exists $k > 0$ and $f \in \mathcal{O}(X)_{k \theta_0}$ such that $f(x) \neq 0$. This is a contradiction. Hence $\mathrm{Stab}_{T}(x)$ is finite. Since this is true for all $x \in X^{\theta}$, we deduce that all orbits are closed in $X^{\theta\text{-ss}}$ (otherwise their closure would contain an orbit of dimension $< \dim T$) and hence every point is stable.  
\end{proof}

Lemma~\ref{Elements in Interior of GIT Cone Are Generic} says that the interior of every top dimensional GIT cone is a GIT chamber. The lemma is false if $T$ is replaced by an arbitrary connected reductive group; see \cite{RessayreThickWall}.

\subsection{Movable cones}

Let $X$ be an affine $T$-variety, $\theta \in \characterlatticeforT$ and $Y_\theta := X^{\theta\text{-ss}} /\!/ \, T$ the GIT quotient. Let $C$ be the GIT cone whose relative interior $C^{\circ} := \mathrm{Int}(C)$ contains $\theta$. Define
\[
\mathrm{St}(\theta) := \{ H \subset T \mid \exists \, x \in X^\theta \text{ such that } T \cdot x \text{ is a closed orbit in } X^\theta \text{ and } H = \operatorname{Stab}_T(x) \}.
\]
We define the orthogonal set $\mathrm{St}(\theta)^\perp \subset \characterlatticeforT$ to be all characters $\lambda$ such that $ \lambda(H) = 1 $ for all $ H \in \mathrm{St}(\theta)$. We note that there exists an integer $m > 0$ such that $m \theta \in \mathrm{St}(\theta)^{\perp}$ since there exists $m$ and $f_1, \dots, f_k \in \mathcal{O}(X)_{m \theta}$ such that $D(f_1), \dots, D(f_k)$ cover $X^{\theta}$. Descent \cite[Proposition~4.2]{KKV} defines a $\mathbb{Z}$-linear map $L_{\mathbb{Z}} \colon \mathrm{St}(\theta)^\perp \to \operatorname{Pic}(Y_{\theta})$ such that $\pi^*_{\theta} L_{\Z}(\sigma) = \mathcal{O}_{X^{\theta}} \otimes \sigma$. We write 
\[
L \colon \mathrm{St}(\theta)^{\perp} \otimes_{\mathbb{Z}} \mathbb{R} \to \operatorname{Pic}(Y_{\theta} / Y_0)_{\mathbb{R}}
\]
for the composite of $L_{\mathbb{Z}}$ with the quotient map, tensored over $\mathbb{R}$. 

\begin{Lemma}\label{lem:pullbackOm}
    If $\lambda \in C$ is integral, then $(\xi_{\theta}^{\lambda})^* \mathcal{O}_{Y_{\lambda}}(1) = L(\lambda)$. 
\end{Lemma}

\begin{proof}
    We must show that there exists $m > 0$ such that $\mathcal{O}_{Y_{\lambda}}(m)$ is a line bundle on $Y_{\lambda}$ and $(\xi_{\theta}^{\lambda})^* \mathcal{O}_{Y_{\lambda}}(m) \cong L_{\Z}(m\lambda)$. Since the morphism $\xi_{\theta}^{\lambda}$ comes from the inclusion $X^{\theta} \hookrightarrow X^{\lambda}$ there is a commutative diagram
    \[
    \begin{tikzcd}
        X^{\theta} \ar[r,hook] \ar[d,"\pi_{\theta}"] & X^{\lambda} \ar[d,"\pi_{\lambda}"] \\
        Y_{\theta} \ar[r,"\xi_{\theta}^{\lambda}"] & Y_{\lambda}.
    \end{tikzcd}
    \]
    Pull-back is an embedding $\pi^*_{\theta} \colon \mathrm{Pic}(Y_{\theta}) \to \mathrm{Pic}_T(X^{\theta})$ \cite[Proposition~4.2]{KKV} and, by definition, $\pi^*_{\theta} L_{\Z}(\sigma) \cong \mathcal{O}_{X^{\theta}} \otimes \sigma$. Therefore, it suffices to show that $\pi^*_{\lambda} \mathcal{O}_{Y_{\lambda}}(m) \cong \mathcal{O}_{X^{\lambda}} \otimes (m \lambda)$. This is equivalent to $(\pi_{\lambda*}( \mathcal{O}_{X^{\lambda}} \otimes m \lambda))^T \cong \mathcal{O}_{Y_{\lambda}}(m)$, which can be checked on the affine open sets $D_+(f)$ for $f \in \mathcal{O}(X)_{k \lambda}$, where we assume $m$ divides $k$. 
\end{proof}


In particular, the proof of Lemma~\ref{lem:pullbackOm} shows that $L_{\mathbb{Z}}(m\theta) = \mathcal{O}_{Y_{\theta}}(m)$ is $\xi_{\theta}^0$-ample. For each $ \mathcal{L} \in \operatorname{Pic}(Y_{\theta}) $, let 
\[
R(Y_{\theta}, \mathcal{L}) := \bigoplus_{n \geq 0} \Gamma(Y_{\theta}, \mathcal{L}^{\otimes n})
\]
denote the associated section ring. If $\mathcal{L}$ is effective, i.e., $\Gamma(Y_{\theta}, \mathcal{L}) \neq 0$, and $R(Y_{\theta}, \mathcal{L})$ is finitely generated then we have a rational map
\[
\psi_{\mathcal{L}} \colon Y_{\theta} \dashrightarrow Y(\mathcal{L}) := \operatorname{Proj} \, R(Y_{\theta}, \mathcal{L})
\]
of schemes over $Y_0$. 


\begin{Lemma}\label{lem:steinfactorizationxithetalambda}
Let $\theta \in C^{\circ}$ and $\lambda \in C$ be integral and $m > 0$ such that $m \lambda \in \mathrm{St}(\theta)^{\perp}$. Then $R(Y_{\theta},{L}_{\Z}(m\lambda))$ is finitely generated and there is an isomorphism $\widetilde{Y}_{\lambda} \cong Y({L}_{\Z}(m\lambda))$, where $\widetilde{Y}_{\lambda} := \Spec_{Y_{\lambda}} (\xi_{\theta}^{\lambda})_* \mathcal{O}_{Y_{\theta}}$ is the Stein factorization of $\xi_{\theta}^{\lambda}$, such that the following diagram commutes:
\begin{equation}\label{eq:commdiagramLtheta}
\begin{tikzcd}
& Y_{\theta} \arrow[dr, "\xi_{\theta}^{\lambda}"] \arrow[dl,"\psi_{{L}_{\Z}(m\lambda)}"'] &  \\
Y(L_{\Z}(m\lambda)) \arrow[r, "\sim"] & \widetilde{Y}_{\lambda} \ar[r] &  Y_{\lambda}.
\end{tikzcd}
\end{equation}
\end{Lemma}

\begin{proof}
The proof is essentially the same as the proof of \cite[Lemma~3.9(iii)]{BellamyCrawSchedlerBirationalGeometryofQuiverVarietiesandOtherGITQuotients}. By Lemma~\ref{lem:pullbackOm}, $L_{\Z}(m \lambda) = (\xi_{\theta}^{\lambda})^* \mathcal{O}_{Y_{\lambda}}(m)$ for the morphism $\xi_{\theta}^{\lambda} \colon Y_{\theta} \to Y_{ \lambda}$ over $Y_0$ induced by VGIT. Since $\xi_{\theta}^{\lambda}$ is a projective morphism and $\mathcal{O}_{Y_{\lambda}}(m)$ is a $\xi_{\lambda}^{0}$-ample line bundle, $L_{\Z}(m \lambda)$ is $\xi_{\theta}^{0}$-semiample line bundle by Lemma~\ref{lem:pullbackgeneratedisgenerated}.  This implies that the section ring $R(Y_{\theta},{L}_{\Z}(m\lambda))$ is finitely generated by Lemma~\ref{lem:Rsectionringfg} and hence $Y(L_{\Z}(m \lambda))$ is well-defined. Finally, the existence of the commutative diagram follows from Proposition~\ref{prop:Steinamplefactor}. 

\end{proof}


\begin{Lemma}\label{lem:NefconeC0}
Assume that the map $L \colon \mathrm{St}(\theta)^{\perp} \otimes_{\mathbb{Z}} \mathbb{R} \to \operatorname{Pic}(Y_{\theta} / Y_0)_{\mathbb{R}}$ is an isomorphism. If, for every integral $ \lambda \in C\setminus C^{\circ}$, the morphism $ \xi_{\theta}^{\lambda}$ contracts a curve, then
\[
L(C^{\circ}) = \operatorname{Amp}(Y_{\theta} / Y_0), \quad L(C) = \operatorname{Nef}(Y_{\theta} / Y_{0}).
\]
In particular, $\operatorname{Nef}(Y_{\theta} / Y_0)$ is a rational polyhedral cone. 
\end{Lemma}

\begin{proof}
If $\lambda \in C^{\circ}$, then the map $\xi_{\theta}^{\lambda} \colon Y_{\theta} \to Y_{\lambda} $ is an isomorphism, and $L(\lambda) = (\xi_{\theta}^{\lambda})^* \mathcal{O}(1)$ is an ample fractional line bundle on $Y_{\theta} $ because $ \mathcal{O}(1) $ is ample on $ Y_{\lambda} $. Therefore, $L(C^{\circ}) \subseteq \operatorname{Amp}(Y_{\theta} / Y_{0})$. We claim that $\dim C = \dim \mathrm{St}(\theta)^{\perp} \otimes_{\mathbb{Z}} \mathbb{R}$. If this is not the case then there exists a one-parameter subgroup $\eta \in X_{\bullet}(T)$ such that $\langle \eta, \lambda \rangle = 0$ for all $\lambda \in C$ but $\langle \eta, \lambda' \rangle \neq 0$ for some $\lambda' \in L^{-1}(\operatorname{Amp}(Y_{\theta} / Y_{0}))$ since $\operatorname{Amp}(Y_{\theta} / Y_{0})$ is top dimensional and contains $L(C^{\circ})$. Since $L(\lambda')$ is $\xi_{\theta}^0$-ample, $X^{\theta} \subset X^{\lambda'}$. But this implies that $\lambda' \in C$ and hence $\langle \eta, \lambda' \rangle =0$, a contradiction. 

If $\lambda \in C \setminus C^{\circ}$, then by assumption the map $\xi^{\lambda}_{\theta}$ contracts a curve. Hence, diagram \eqref{eq:commdiagramLtheta} implies that the map $ \psi_{{L}(\lambda)} $ contracts a curve since the morphism $ \widetilde{Y}_{\lambda} \to Y_{\lambda}$ is finite and so doesn't contract any curves. This implies that the line bundle $L(\lambda)$ pairs to zero with the contracted curve. Thus, ${L}(\lambda)$ is not in the interior of $\operatorname{Nef}(Y_{\theta} / Y_0)$. But we have $L(C^{\circ}) \subset \operatorname{Amp}(Y_{\theta} / Y_0) \subseteq \operatorname{int}(\operatorname{Nef}(Y_{\theta} / Y_0))$. Therefore, $L(C) = \overline{L(C^{\circ})} \subseteq \operatorname{Nef}(Y_{\theta} / Y_0)$. Since $\operatorname{Nef}(Y_{\theta} / Y_0)$ is a convex cone this implies that $L(C) = \operatorname{Nef}(Y_{\theta} / Y_{0})$ by Lemma~\ref{lem:intclosureequal} below since $C^{\circ}$ is open in $\mathrm{St}(\theta)^{\perp}$ and we have assumed that $L$ is an isomorphism, which implies that $\dim C = \dim \operatorname{Pic}(Y_{\theta} / Y_0)_{\mathbb{R}}$.
\end{proof}

\begin{Corollary}\label{cor:MoveconeGIT}
   Assume that $Y_{\theta}$ is normal and $\Q$-factorial and that $\xi^0_{\theta}$ is birational. If the map $L \colon \mathrm{St}(\theta)^{\perp} \otimes_{\mathbb{Z}} \mathbb{R} \to \operatorname{Pic}(Y_{\theta} / Y_0)_{\mathbb{R}}$ is an isomorphism and, for every $\lambda \in C \setminus C^{\circ}$, the morphism $\xi^{\lambda}_{\theta}$ contracts a divisor, then $L(C) = \operatorname{Nef}(Y_{\theta} / Y_{0}) = \operatorname{Mov}(Y_{\theta} / Y_0)$.
\end{Corollary}

\begin{proof}
The proof is based on the proof of \cite[Proposition~5.4]{Ohta}, though we do not need to assume $Y_{\theta}$ is a relative Mori dream space. If, for every $\lambda \in C \setminus C^{\circ}$, the morphism $\xi^{\lambda}_{\theta}$ contracts a divisor then $\xi^{\lambda}_{\theta}$ contracts at least one curve. Therefore, Lemma~\ref{lem:NefconeC0} implies that $\operatorname{Nef}(Y_{\theta} / Y_{0}) = L(C)$ is a top dimensional rational polyhedral cone. 

Since $\operatorname{Nef}(Y_{\theta} / Y_{0}) \subset \operatorname{Mov}(Y_{\theta} / Y_0)$ and $\operatorname{Nef}(Y_{\theta} / Y_{0})$ is top dimensional, the cones are equal by Lemma~\ref{lem:intclosureequal} if we can show that any point in the boundary of $\operatorname{Nef}(Y_{\theta} / Y_{0})$ is also in the boundary of $\operatorname{Mov}(Y_{\theta} / Y_0)$. For any cone $\sigma$ in the boundary of $\operatorname{Nef}(Y_{\theta} / Y_{0})$, we can choose $\lambda \in \mathrm{St}(\theta)^{\perp}$ such that $L(\lambda)$ is in the interior of $\sigma$ because $\operatorname{Nef}(Y_{\theta} / Y_{0}) = L(C)$. Then Lemma~\ref{lem:pullbackOm} says that $L(\lambda) = (\xi_{\theta}^{\lambda})^* \mathcal{O}_{Y_{\lambda}}(1)$. We wish to apply Lemma~\ref{lem:technicalnonmovable}. To do so, we take $Y = Y_{\theta}$, $Z = \widetilde{Y}_{\lambda}$ and $g = \psi_{L_{\mathbb{Z}}(\lambda)}$ as in the statement of Lemma~\ref{lem:steinfactorizationxithetalambda}. Then $g$ contracts a divisor (since $k \colon \widetilde{Y}_{\lambda} \to Y_{\lambda}$ is finite) and $L(\lambda) = \psi_{L_{\mathbb{Z}}(\lambda)}^* M$, where $M := k^* \mathcal{O}_{Y_{\lambda}}(1)$ is $(\xi_{\lambda}^0 \circ k)$-ample by \cite[Lemma 29.38.7(2), Tag 0892]{StacksProject}. In particular, $L(\lambda)$ is $\xi_{\theta}^0$-moveable so $L(\lambda) \in \operatorname{Mov}(Y_{\theta} / Y_0)$. To use the language of divisors, write $L(\lambda) = \mathcal{O}_{Y_{\theta}}(D)$. If $E$ is a prime divisor contracted by $\psi_{L_{\mathbb{Z}}(\lambda)}$ then Lemma~\ref{lem:technicalnonmovable} says that $D + \epsilon E$ is not $\xi_{\theta}^0$-movable for all rational $\epsilon > 0$. Thus, $D$ is $\xi_{\theta}^0$-movable but the $\Q$-divisors $D + \epsilon E$ lie arbitrarily close to $D$ and are not movable. This means that $D$ (equivalent, $L(\lambda)$) must be in the boundary of $\operatorname{Mov}(Y_{\theta} / Y_0)$.

\end{proof}

\subsubsection{Elementary Cone Lemma}

The following elementary result was used in the proof of Lemma~\ref{lem:NefconeC0}. 

\begin{Lemma}\label{lem:intclosureequal}
If $C,D \subseteq \R^n$ are closed convex $n$-dimensional cones for which $\mathrm{Int}(C) \subseteq \mathrm{Int}(D)$ and $\partial C \subseteq \partial D$ then $C = D$.
\end{Lemma}

\begin{proof}
Observe that our assumption implies that $C = \overline{\operatorname{Int}(C)} \subseteq \overline{\operatorname{Int}(D)} = D$. Suppose for the sake of contradiction there exists a $d \in D\setminus C$. Observe this implies that there is a $d \in \mathrm{Int}(D)$ that is not in $C$: indeed, otherwise we would have that $D = \overline{\operatorname{Int}(D)} \subseteq \overline{\operatorname{Int}(C)} = C$.
%

Notice that $C$ is closed and $\operatorname{Int}(D)$ is open.  
So $C^c \cap \operatorname{Int}(D)$ is also open in $\mathbb{R}^n$.  

Now, $\operatorname{Int}(C)$ is defined by a collection of inequalities:
\[
\operatorname{Int}(C) = \bigcap_{i \in I} \{ x : \lambda_i(x) > 0 \}
\]
for some linear functionals $\lambda_i$.  

Choose $d \in C^c \cap \operatorname{Int}(D)$ and $c \in \operatorname{Int}(C)$. Since $d \notin C$, there exists some $i$ for which 
$\lambda_i(d) < 0$ and, since $c \in \mathrm{Int}(C)$, $\lambda_i(c) > 0$. Moreover, since $D$ is a convex cone and $c,d \in \operatorname{Int}(D)$, the line 
\[
[c,d] := \{\, t c + (1-t) d \mid t \in [0,1] \,\}
\]
is contained in $\operatorname{Int}(D)$.  

We claim that there is some point $c_0 \in [c,d]$ such that $c_0 \in \partial C$.  
We have $\operatorname{Int}(C) \cap [c,d]$ and $(C^c \cap \operatorname{Int}(D)) \cap [c,d]$ open in $[c,d]$, but
\[
[c,d] = (C \cap [c,d]) \,\sqcup\, \big( (C^c \cap \operatorname{Int}(D)) \cap [c,d] \big)
\]
with $C \cap [c,d]$ closed in $[c,d]$.  
Therefore, there must be a point in $(C \cap [c,d])$ but not in $(\operatorname{Int}(C) \cap [c,d])$. Therefore $c_0 \in \partial C \subseteq \partial D$, and so $c_0$ lies both in the interior and the boundary of $D$, which is a contradiction.
\end{proof}

\subsection{Wall-and-Chamber Structure for $\overline{G/U}$}

Now we consider the case of $T$ acting on $\overline{G/U}$. 

\begin{Proposition}\label{Wall and Chamber Decomposition of affineClosureofBasicAffineSpace is Same for Parabolic Subgroups}
A character $\theta \in \characterlatticeforT$ is effective for the action of $T$ on $\affineClosureofBasicAffineSpace$ if and only if $\theta$ is dominant. Moreover, an effective character $\theta \in \characterlatticeforT$ is generic if and only if $\theta$ is regular. 

More generally, the vanishing and non-vanishing of the simple coroots give the wall-and-chamber structure of the GIT fan for this action.
\end{Proposition}

Before proving \cref{Wall and Chamber Decomposition of affineClosureofBasicAffineSpace is Same for Parabolic Subgroups} we will recall the following standard lemma (see for example \cite[Lemma 6.13]{GuilleminJeffreySjamaarSymplecticImplosion}, \cite[Section 5.1.2]{GannonProofOftheGinzburgKazhdanConjecture}), which uses the assumption that $G$ is simply connected: 

\begin{Lemma}\label{Stratification on affineClosureofBasicAffineSpace}
The ring $\ringOfFunctionsForBasicAffineSpace := \O(G/U)$ is generated by $\ringOfFunctionsForBasicAffineSpace_{\omega_i}$, where $\omega_i$ varies over the fundamental weights of $G$, and there is a $G \times T$-equivariant stratification 
\[
\overline{G/U} \cong \bigsqcup _{S \subseteq \Pi^{\vee}}G/[P_{S}, P_S]\text{ such that  } G/[P_{S}, P_S] \cong \bigcap_{\alpha_i^{\vee} \in S} V(\ringOfFunctionsForBasicAffineSpace_{\omega_i}) \cap \bigcap_{\alpha_i^{\vee} \notin S} D(\ringOfFunctionsForBasicAffineSpace_{\omega_i}),
\] where $\Pi^{\vee}$ is the set of simple coroots and $P_S$ is the parabolic subgroup whose associated simple coroots are exactly those in $S$.
\end{Lemma}

\begin{proof}[Proof of \cref{Wall and Chamber Decomposition of affineClosureofBasicAffineSpace is Same for Parabolic Subgroups}] Since $\ringOfFunctionsForBasicAffineSpace$ is generated by $A_{\lambda}$, for $\lambda$ dominant, we see that $A_{\theta} = 0$ if $\theta$ is not dominant and so the semistable locus of such $\theta$ is empty.

Label the simple coroots $\alpha_i^{\vee}$ and fundamental weights $\omega_i$ so that $\langle \alpha_i^{\vee},\omega_j \rangle = \delta_{i,j}$. Now assume $\theta \in \characterlatticeforT$ is dominant. Write $\theta = \sum_i m_i\omega_i$ for some non-negative integers $m_i$, and let $S_{\theta}$ denote the subset of simple coroots for which $m_i \neq 0$. From the above stratification, we see that the complement of the semistable locus is given by \[\cup_{\alpha_i^{\vee} \in S_{\theta}}V(\ringOfFunctionsForBasicAffineSpace_{\omega_i}) = \cup_{\alpha_i^{\vee} \in S_\theta} \overline{G/[P_{ \{ \alpha_i^{\vee} \} }, P_{  \{ \alpha_i^{\vee} \} }] }\] set theoretically. Clearly, then, the semistable locus for $\theta$ agrees with that of $\theta'$ if and only if $S_{\theta} = S_{\theta'}$, which gives our desired wall-and-chamber decomposition. 

Fix some dominant $\theta \in \characterlatticeforT$. We will show that $\theta$ is generic if and only if $\theta$ is regular. If $\theta$ is regular, our above analysis gives that the $\theta$-semistable points are precisely the open subset $G/U$. Since every point of $G/U$ has trivial $T$-stabilizer and has closed $T$-orbit 
we see that every point of $G/U$ is also $\theta$-stable, and so $\theta$ is generic. On the other hand, if $\theta$ is \textit{not} regular, then it lies on some wall, and our above computation gives that the $\theta$-semistable locus is strictly larger than $G/U$. Since every point of the complement of $G/U$ in $\affineClosureofBasicAffineSpace$ has a positive dimensional $T$-stabilizer by our above stratification, we see that there must exist $\theta$-semistable points which are not $\theta$-stable in this case, and so $\theta$ is not generic.
\end{proof}

\subsection{Stable Characters for $T \curvearrowright\affineClosureOfCotangentBundleofBasicAffineSpace$} 

We now prove that the (GIT) notion of generic for $T \curvearrowright \affineClosureOfCotangentBundleofBasicAffineSpace$ agrees with the usual representation theoretic notion of regular weights. Let $R := \mathcal{O}(T^* (G/U))$ so that $\affineClosureOfCotangentBundleofBasicAffineSpace = \Spec \, R$. 

\begin{Proposition}\label{Generic Character for GIT iff Generic Character Rep Theoretically}
Fix $\theta \in \characterlatticeforT$. Then $\theta$ is generic if and only if $\theta(\alpha^{\vee}) \neq 0$ for all coroots $\alpha$. Moreover, if $\theta$ is not generic, there exists some $y \in \affineClosureofBasicAffineSpace$ and $w \in W$ such that $w\zeroSectionAffinized(y) \in \affineClosureOfCotangentBundleofBasicAffineSpace$ is $\theta$-semistable but not $\theta$-stable. 
\end{Proposition}

We prove this after proving the following lemma and deriving a corollary from it.

\begin{Lemma}\label{Free Action If Weight Has No Coroot Vanishing}
Assume $\theta$ has the property that $\theta(\alpha^{\vee})$ is nonzero for every coroot $\alpha^{\vee}$. Then the $T$-stabilizer of any $\theta$-semistable point is trivial. If $\theta$ is moreover dominant, then $x \in \affineClosureOfCotangentBundleofBasicAffineSpace$ is $\theta$-semistable if and only if $\projectionFromAffineClosureofCotangentBundleToAffineClosureofSpace(x) \in G/U \subseteq \affineClosureofBasicAffineSpace$.
\end{Lemma}

\begin{proof}
Assume $\theta \in \characterlatticeforT$ has the property that $\theta(\alpha^{\vee}) \neq 0$ for all coroots $\alpha^{\vee}$, and fix $x$ semistable. After translating by an element of $W$, we may assume that $\theta$ is dominant. Assume there exists some $f \in \ringOfFunctionsForCOTANGENTBUNDLEOfBasicAffineSpace_{n\theta}$ such that $f(x) \neq 0$. By \cite[Lemma 3.6]{GinzburgRicheDifferentialOperatorsOnBasicAffineSpaceandtheAffineGrassmannian} (see also \cite[Lemma 2.2]{GannonProofOftheGinzburgKazhdanConjecture}) there exists some $f \in \ringOfFunctionsForBasicAffineSpace_{n\theta}$ for which $f(x) \neq 0$. This implies that $\projectionFromAffineClosureofCotangentBundleToAffineClosureofSpace(x) \in G/U$ and hence the action of $T$ on $x$ is free.

Conversely, assume $x \in \affineClosureOfCotangentBundleofBasicAffineSpace$ has the property that $\projectionFromAffineClosureofCotangentBundleToAffineClosureofSpace(x) \in G/U$. Then, for each fundamental weight 
$\omega_i$, there exists some $f_i \in \ringOfFunctionsForBasicAffineSpace_{\omega_i}$ such that $f_i(x) \neq 0$. Write $\theta = \sum_i n_i \omega_i$ for $n_i \in \mathbb{Z} > 0$, and set $g := \prod_if_i^{n_i} \in \ringOfFunctionsForBasicAffineSpace_{\theta}$. Since $\ringOfFunctionsForBasicAffineSpace$ is an integral domain, $g(x) \neq 0$. Therefore $x$ is $\theta$-semistable.
\end{proof}

\begin{Corollary}\label{Orbit of Theta Semistable Point is Closed}Assume $\theta \in \characterlatticeforT$ has the property that $\theta(\alpha^{\vee}) \neq 0$ for all coroots $\alpha^{\vee}$. The orbit of any $\theta$-semistable point is closed in the semistable locus.
\end{Corollary}

\begin{proof}
The closure of the orbit is the union of the orbit itself and locally closed subvarieties of strictly smaller dimension which are themselves orbits. However, we have shown that any $\theta$-semistable point has a free $T$-action, and so no such smaller dimension can occur.
\end{proof}

\begin{Remark}
    Unlike for $\overline{G/U}$, every character of $T$ is effective for its action on $\affineClosureOfCotangentBundleofBasicAffineSpace$.
\end{Remark}

\begin{proof}[Proof of \cref{Generic Character for GIT iff Generic Character Rep Theoretically}]
    Assume $\theta(\alpha^{\vee}) \neq 0$ for all coroots $\alpha^{\vee}$. By \cref{Free Action If Weight Has No Coroot Vanishing}, the stabilizer of every $\theta$-semistable point is trivial, and so, in particular, it is finite, and by \cref{Orbit of Theta Semistable Point is Closed}, every orbit in the $\theta$-semistable locus is closed in this locus.  
    
    Conversely, assume $\langle \theta, \alpha^{\vee} \rangle = 0$ for some coroot $\alpha^{\vee}$. We wish to show that there exists a $\theta$-semistable point which is not $\theta$-stable. Observe that, since $T \rtimes W$ acts on $\affineClosureOfCotangentBundleofBasicAffineSpace$, a point $x$ of $\affineClosureOfCotangentBundleofBasicAffineSpace$ is $\theta$-(semi)stable if and only if $wx$ is $w\theta$-(semi)stable. Therefore, it suffices to show that $w\theta$ is not generic for some $w \in W$. We choose $w \in W$ such that $\theta' := w\theta$ is dominant. 
    
    Our assumption on $\theta$ implies $\theta'(\beta^{\vee}) = 0$ for some simple coroot $\beta^{\vee}$. Therefore we may write $\theta'$ as a nonnegative integral linear combination of the set $S$ of fundamental weights distinct from the fundamental weight dual to $\beta^{\vee}$. Therefore, since for each point $p$ in the $G$-orbit $G/[P_{\beta^{\vee}}, P_{\beta^{\vee}}]$ in $\overline{G/U}$ there is a function $f \in \ringOfFunctionsForBasicAffineSpace_{\omega_i}$ for $\omega_i \in S$ with $f(p) \neq 0$ (see \cref{Stratification on affineClosureofBasicAffineSpace}), we obtain that any point of $G/[P_{\beta^{\vee}}, P_{\beta^{\vee}}]$ is $\theta'$-semistable. Moreover, since the $\theta'$-semistable locus is open and $G$-invariant, we see that it contains $G/U$ as well. 
    On the other hand, no point of $G/U$ is $\theta'$-stable since, with respect to the $\G_m$-action through the coroot $\beta^{\vee}: \G_m \to T$, we have $\lim_{t \to 0}tx \in G/[P_{\beta^{\vee}}, P_{\beta^{\vee}}]$ for any $x \in G/U$, and so the $T$-orbit of any point of $G/U$ is not closed in the $\theta'$-semistable locus. 
\end{proof}

\subsection{Corollaries for $T^*(G/U)$}
\begin{Corollary}\label{Theta Stable Points Equals Theta Semistable Points is Cotangent Bundle}
Assume $\theta$ is dominant regular. Then $T^*(G/U) \subseteq \affineClosureOfCotangentBundleofBasicAffineSpace$ is precisely the set of $\theta$-stable points. More generally, if $\theta \in \characterlatticeforT$ is regular and $w \in W$ has the property that $w^{-1}\theta$ is dominant, then $wT^*(G/U) \subseteq \affineClosureOfCotangentBundleofBasicAffineSpace$ is precisely the $\theta$-stable locus.
\end{Corollary}

\begin{proof}
Since there is a $T \rtimes W$-action on $\affineClosureOfCotangentBundleofBasicAffineSpace$, the second claim follows immediately from the first, so we may now assume $\theta$ itself is dominant.

If $x$ is $\theta$-stable, then by \cref{Free Action If Weight Has No Coroot Vanishing}, $\projectionFromAffineClosureofCotangentBundleToAffineClosureofSpace(x) \in G/U$. Therefore, by \cref{Preimage of G Mod U Under Projection Is Its Cotangent Bundle}, $x \in T^*(G/U)$. Conversely, assume $x \in T^*(G/U)$. Then $\projectionFromAffineClosureofCotangentBundleToAffineClosureofSpace(x) \in G/U$, and so the $T$-stabilizer of $\projectionFromAffineClosureofCotangentBundleToAffineClosureofSpace(x) $ is trivial. Moreover, by \cref{Free Action If Weight Has No Coroot Vanishing}, $x$ is $\theta$-semistable and, by \cref{Orbit of Theta Semistable Point is Closed}, the orbit of $x$ is closed in the $\theta$-semistable locus. Therefore $x$ is $\theta$-stable. 
\end{proof}

\begin{Corollary}\label{Hamiltonian Reduction at Dominant Regular Level for Affine Closure of Cotangent Bundle}
The Hamiltonian reduction of $\affineClosureOfCotangentBundleofBasicAffineSpace$ with respect to a dominant regular $\theta \in \characterlatticeforT$ is $T^*(G/B)$. Moreover, the GIT quotient $\affineClosureOfCotangentBundleofBasicAffineSpace\sslash_\theta T$ is isomorphic to $\tg$. In particular, these varieties are smooth and hence normal. 
\end{Corollary}

\begin{proof}
    We have 
    \begin{multline*}
        \affineClosureOfCotangentBundleofBasicAffineSpace/\!/\!/_{\theta} T =: \momentMapFromAFFINECLOSUREofCotangentSpaceWithGROUPT^{-1}(0)^{\theta\mathrm{-ss}}\sslash T =  \left(\momentMapFromAFFINECLOSUREofCotangentSpaceWithGROUPT^{-1}(0) \cap \affineClosureOfCotangentBundleofBasicAffineSpace^{\theta\mathrm{-ss}} \right) \sslash T \\ = (\momentMapFromAFFINECLOSUREofCotangentSpaceWithGROUPT^{-1}(0) \cap T^*(G/U)) \sslash T = \mu_T^{-1}(0)\sslash T = T^*(G/B)
    \end{multline*}
    given by the definition of $/\!/\!/_{\theta}$, the fact that the $\theta$-semistable locus of a closed $T$-invariant subscheme is the intersection of the closed subscheme with the $\theta$-semistable locus of the ambient scheme, \cref{Free Action If Weight Has No Coroot Vanishing}, the definition of $\momentMapFromAFFINECLOSUREofCotangentSpaceWithGROUPT$, and the fact that the Hamiltonian reduction of a cotangent bundle of a space with a free $T$-action gives the cotangent bundle of the quotient. The proof of the isomorphism $\affineClosureOfCotangentBundleofBasicAffineSpace\sslash_\theta T \cong \tg$ is proved in a completely parallel manner. 
\end{proof}

\begin{Corollary}\label{Regular Dominant Weights are GIT Chamber}
    The set of regular dominant weights is a GIT chamber for the $T$-action on $\momentMapFromAFFINECLOSUREofCotangentSpaceWithGROUPT^{-1}(0)$.
\end{Corollary}

\begin{proof}
By the definition of GIT chamber, it suffices to show that any regular dominant $\theta \in \characterlatticeforT$ is generic and any dominant $\lambda \in \characterlatticeforT$ which is not regular is not generic. Fix such a $\theta$; we will prove it is generic. To this end, fix a $\theta$-semistable point $x$. By \cref{Free Action If Weight Has No Coroot Vanishing} and the fact that $\affineClosureOfCotangentBundleofBasicAffineSpace^{\theta\mathrm{-ss}} \cap \momentMapFromAFFINECLOSUREofCotangentSpaceWithGROUPT^{-1}(0) = \momentMapFromAFFINECLOSUREofCotangentSpaceWithGROUPT^{-1}(0)^{\theta\mathrm{-ss}}$, $\projectionFromAffineClosureofCotangentBundleToAffineClosureofSpace$ maps $x$ into $G/U$. Therefore, by \cref{Theta Stable Points Equals Theta Semistable Points is Cotangent Bundle}, $x$ is $\theta$-stable. Moreover, \cref{Wall and Chamber Decomposition of affineClosureofBasicAffineSpace is Same for Parabolic Subgroups} shows that, for regular dominant characters, the set  $\momentMapFromAFFINECLOSUREofCotangentSpaceWithGROUPT^{-1}(0)^{\theta\mathrm{-ss}}$ is independent of $\theta$ and hence all regular dominant characters lie in the same GIT chamber. 

Now fix \textit{any} non-regular $\lambda$; we will show it is not generic. By \cref{Generic Character for GIT iff Generic Character Rep Theoretically}, we see that there exists some $w \in W$ and $y \in \affineClosureofBasicAffineSpace$ such that $w\zeroSectionAffinized(y)$ is $\lambda$-semistable but not $\lambda$-stable. Observe that $\momentMapFromAFFINECLOSUREofCotangentSpaceWithGROUPT(\zeroSectionAffinized(y)) = 0$ since the moment map applied to the zero section is zero. This, along with the fact that the moment map is $W$-equivariant, gives our claim
\end{proof}

\begin{Corollary}\label{Wall and Chamber Structure for T action on affineClosureOfCotangentBundleofBasicAffineSpace}
The $W$-translates of the walls and chambers of \cref{Wall and Chamber Decomposition of affineClosureofBasicAffineSpace is Same for Parabolic Subgroups} give the GIT wall-and-chamber structure for the $T$-action on $\affineClosureOfCotangentBundleofBasicAffineSpace$ and for the $T$-action on $\momentMapFromAFFINECLOSUREofCotangentSpaceWithGROUPT^{-1}(0)$.
\end{Corollary}

\begin{proof}
We prove the claim for $\affineClosureOfCotangentBundleofBasicAffineSpace$; the claim for $\momentMapFromAFFINECLOSUREofCotangentSpaceWithGROUPT^{-1}(0)$ follows by essentially identical arguments.

Assume $\theta$ is GIT equivalent to $\theta'$. Then $\theta$ and $\theta'$ must lie in the relative interior of the same GIT cone. By \cref{Generic Character for GIT iff Generic Character Rep Theoretically}, the GIT cone is a $W$-translate of a cone of dominant weights. Choose a $w \in W$ which maps this cone onto the set of dominant weights, and let $\lambda := w\theta$ and $\lambda' := w\theta'$. Then, since the Gelfand-Graev action upgrades to an action of $T \rtimes W$, $\lambda$ being GIT equivalent to $\lambda'$ implies that 
\[\affineClosureofBasicAffineSpace^{\lambda\mathrm{-ss}} = \affineClosureofBasicAffineSpace \cap \affineClosureOfCotangentBundleofBasicAffineSpace^{\lambda\mathrm{-ss}} = \affineClosureofBasicAffineSpace \cap \affineClosureOfCotangentBundleofBasicAffineSpace^{\lambda'\mathrm{-ss}} = \affineClosureofBasicAffineSpace^{\lambda'\mathrm{-ss}}
\]
and so $\lambda(\alpha^{\vee}) = 0$ for some simple coroot $\alpha^{\vee}$ if and only if $\lambda'(\alpha^{\vee}) = 0$ by \cref{Wall and Chamber Decomposition of affineClosureofBasicAffineSpace is Same for Parabolic Subgroups}.

Conversely, assume that there is some $w \in W$ such that $w\theta(\alpha^{\vee}) = 0$ for some simple coroot $\alpha^{\vee}$ if and only if $w\theta'(\alpha^{\vee}) = 0$. As above, let $\lambda := w\theta$ and $\lambda' := w\theta'$. By using the Gelfand-Graev action on $\affineClosureOfCotangentBundleofBasicAffineSpace$ we see that it suffices to show the $\lambda$ and $\lambda'$ semistable subsets of $\affineClosureOfCotangentBundleofBasicAffineSpace$ agree. By \cite[Lemma 3.6]{GinzburgRicheDifferentialOperatorsOnBasicAffineSpaceandtheAffineGrassmannian}, for any dominant weight $\lambda \in \characterlatticeforT$ the $\ringOfFunctionsForCOTANGENTBUNDLEOfBasicAffineSpace_0$-module $\ringOfFunctionsForCOTANGENTBUNDLEOfBasicAffineSpace_{\lambda}$ is generated by $\ringOfFunctionsForBasicAffineSpace_{\lambda}$. Therefore the set $V(\ringOfFunctionsForCOTANGENTBUNDLEOfBasicAffineSpace_{\lambda})$ is the preimage of $V(\ringOfFunctionsForBasicAffineSpace_{\lambda})$ under $\projectionFromAffineClosureofCotangentBundleToAffineClosureofSpace$. As we have already recalled in \cref{Wall and Chamber Decomposition of affineClosureofBasicAffineSpace is Same for Parabolic Subgroups}, $V(\ringOfFunctionsForBasicAffineSpace_{\lambda}) = \cup_{\omega_i \in S_{\lambda}}V(\ringOfFunctionsForBasicAffineSpace_{\omega_i})$ as subsets of $\affineClosureofBasicAffineSpace$. Therefore, set theoretically, \[V(\ringOfFunctionsForCOTANGENTBUNDLEOfBasicAffineSpace_{\lambda}) = \projectionFromAffineClosureofCotangentBundleToAffineClosureofSpace^{-1}(\cup_{\omega_i \in S_{\lambda}}V(\ringOfFunctionsForBasicAffineSpace_{\omega_i})) = \cup_{\omega_i \in S_{\lambda}}\projectionFromAffineClosureofCotangentBundleToAffineClosureofSpace^{-1}(V(\ringOfFunctionsForBasicAffineSpace_{\omega_i})) = \cup_{\omega_i \in S_{\lambda}}V(\ringOfFunctionsForCOTANGENTBUNDLEOfBasicAffineSpace_{\omega_i}).\]
This implies that the $\lambda$-semistable locus $D(\ringOfFunctionsForCOTANGENTBUNDLEOfBasicAffineSpace_{\lambda})$ agrees with the $\lambda'$-semistable locus $D(\ringOfFunctionsForCOTANGENTBUNDLEOfBasicAffineSpace_{\lambda'})$ if and only if $S_{\lambda} = S_{\lambda'}$. 
\end{proof}

\subsection{(Semi-)stable points in $\momentMapFromAFFINECLOSUREofCotangentSpaceWithGROUPT^{-1}(0)$} 


\begin{Proposition}\label{There Are Zero-Stable Points in Moment Map Preimage}
For any $\theta \in \characterlatticeforT$, there are $\theta$-stable points in $\momentMapFromAFFINECLOSUREofCotangentSpaceWithGROUPT^{-1}(0)$.
  \end{Proposition}

  \begin{proof}
      Note that for any $\theta \in \characterlatticeforT$, \[\momentMapFromAFFINECLOSUREofCotangentSpaceWithGROUPT^{-1}(0)^{\theta\mathrm{-stable}} = \affineClosureOfCotangentBundleofBasicAffineSpace^{\theta\mathrm{-stable}} \cap \momentMapFromAFFINECLOSUREofCotangentSpaceWithGROUPT^{-1}(0),\] so it suffices to prove that $\momentMapFromAFFINECLOSUREofCotangentSpaceWithGROUPT^{-1}(0)$ contains a point which is $\theta$-stable in $\affineClosureOfCotangentBundleofBasicAffineSpace$. Note that a point is $\theta$-stable if and only if that point has a closed $T$-orbit inside the $\theta$-semistable locus and the stabilizer of that point is finite. Choose a regular element $e \in \LU$, and consider the point $(1, e) \in G \times^U \LB \cong T^*(G/U)$. Since the moment map to $\LG$ is given by multiplication, the moment map for the Hamiltonian $G$-action on $T^*(G/U)$ sends $(1, e)$ to the regular element $e$ of $\LG$. By \cite[Proposition 5.1.4]{Gin}, the $T$-orbit of $(1, e)$ is therefore closed in $\affineClosureOfCotangentBundleofBasicAffineSpace$ (and, in particular, closed inside the $\theta$-semistable locus of $\affineClosureOfCotangentBundleofBasicAffineSpace$) and the stabilizer in $T$ has dimension zero. Moreover, one can check directly (or appeal to \cite[Proposition 5.19]{GannonProofOftheGinzburgKazhdanConjecture}) that the $T$-stabilizer of $(1, e)$ is in fact trivial. Points in $\affineClosureOfCotangentBundleofBasicAffineSpace$ whose orbit is free and closed in $\affineClosureOfCotangentBundleofBasicAffineSpace$ are stable with respect to all stability conditions. In particular, $(1,e)$ is $\theta$-stable.
  \end{proof}



\begin{Proposition}\label{Torus Stabilizer of lamnbda Semistable Points}
If $\lambda$ is dominant and $P := P_{\lambda}$, the $T$-stabilizer of any point of $\momentMapFromAFFINECLOSUREofCotangentSpaceWithGROUPT^{-1}(0)^{\lambda\mathrm{-ss}}$ is contained in the kernel $T'$ of all $\omega_i$ corresponding to simple coroots $\alpha_i^{\vee} \notin S_P$. Moreover, there exists some point of $\momentMapFromAFFINECLOSUREofCotangentSpaceWithGROUPT^{-1}(0)^{\lambda\mathrm{-ss}}$ whose stabilizer is exactly $T'$.
\end{Proposition}

\begin{proof}
Let $\overline{\ringOfFunctionsForCOTANGENTBUNDLEOfBasicAffineSpace} := \ringOfFunctionsForCOTANGENTBUNDLEOfBasicAffineSpace \otimes_{\Symt} k$ denote the ring of functions on $\momentMapFromAFFINECLOSUREofCotangentSpaceWithGROUPT^{-1}(0)$.\footnote{This is distinct from the ring of functions on $\mu_T^{-1}(0) \cong G \times^U \LU$, even when $G = \SL_2$, see \cref{Example for SL2}.} The semistable locus is given by the non-vanishing of $\overline{\ringOfFunctionsForCOTANGENTBUNDLEOfBasicAffineSpace}_{n\lambda}$ for positive integers $n$. Since $\overline{\ringOfFunctionsForCOTANGENTBUNDLEOfBasicAffineSpace}_{n\lambda}$ is contained in the subring generated by $\ringOfFunctionsForCOTANGENTBUNDLEOfBasicAffineSpace_{\omega_i}$ for $\alpha_i^{\vee} \notin S_P$ \cite[Lemma 3.6]{GinzburgRicheDifferentialOperatorsOnBasicAffineSpaceandtheAffineGrassmannian} the $\lambda$-semistable locus is equivalently given by the intersection of the non-vanishing of $R_{\omega_i}$ for each $\alpha_i^{\vee} \notin S_P$. Therefore the stabilizer of any point in the $\lambda$-semistable locus is a closed subscheme of $\cap \mathrm{ker}(\omega_i) = \prod_{\alpha^{\vee}_i \in S_P}\G_m$, which shows the first claim. Now, for the second, we observe that any point of $G/[P, P]$ in the stratification of \cref{Stratification on affineClosureofBasicAffineSpace}, viewed as a locally closed subset of $\affineClosureOfCotangentBundleofBasicAffineSpace$ by the map $\zeroSectionAffinized$, has stabilizer exactly $T'$. Since this locally closed subset is in the non-vanishing set of $\ringOfFunctionsForBasicAffineSpace_{\omega_i}$ for $\omega_i$ corresponding to simple coroots not in $S_P$, we see that $G/[P, P]$ is $\lambda$-semistable as well. 
\end{proof}

\begin{Remark}
    Of course, using the Gelfand-Graev action, one can immediately derive an analogue of \cref{Torus Stabilizer of lamnbda Semistable Points} for an arbitrary $\lambda \in \characterlatticeforT$.
\end{Remark}

\begin{Corollary}\label{cor:stthetaequalset}
For any $\theta \in \characterlatticeforT$, the set $\{\lambda \in \characterlatticeforT : \lambda(\alpha^{\vee}) = 0\text{ for all } \alpha^{\vee} \in S_{P_{\theta
}}\}$ equals the orthogonal set $\mathrm{St}(\theta)^{\perp}$. 
\end{Corollary}

\begin{proof}
Assume $\lambda$ lies in the former set. If $S \subseteq T$ stabilizes some point of $\momentMapFromAFFINECLOSUREofCotangentSpaceWithGROUPT^{-1}(0)^{\theta\mathrm{-ss}}$, then \cref{Torus Stabilizer of lamnbda Semistable Points} says that $S \subseteq T'$. Now, since $T' \cong \prod_{\alpha^{\vee} \in S_{P_\theta}}\G_m$, we see that $\lambda|_{T'}$ must be trivial since by assumption $\lambda(\alpha^{\vee}) = 0$. Hence $\lambda \in \mathrm{St}(\theta)^{\perp}$.  

Conversely, assume $\lambda$ lies in $\mathrm{St}(\theta)^{\perp}$. Since there is a point of $\momentMapFromAFFINECLOSUREofCotangentSpaceWithGROUPT^{-1}(0)^{\theta \mathrm{-ss}}$ whose stabilizer is exactly $T'$, we have $\lambda |_{T'} = 0$. Since $T' \cong \prod_{\alpha^{\vee} \in S_{P_{\theta}}}\G_m$, this implies that $\lambda(\alpha^{\vee}) = 0$ for all $\alpha^{\vee} \in S_{P_{\theta}}$.
\end{proof}

\section{Proof of Main Theorem} 

\subsection{Explicit Description of GIT Resolutions for the Nilpotent Cone}\label{Explicit Description of GIT Resolutions Subsection} Letting \[Y_{\lambda} := \momentMapFromAFFINECLOSUREofCotangentSpaceWithGROUPT^{-1}(0)\sslash_{\lambda} T := \momentMapFromAFFINECLOSUREofCotangentSpaceWithGROUPT^{-1}(0)^{\lambda\text{-ss}}\sslash T,
\] we observe that if $\theta$ is in the interior of a GIT cone and $\lambda$ on a wall of the cone then there is an open embedding $\momentMapFromAFFINECLOSUREofCotangentSpaceWithGROUPT^{-1}(0)^{\theta\text{-ss}} \subseteq \momentMapFromAFFINECLOSUREofCotangentSpaceWithGROUPT^{-1}(0)^{\lambda\text{-ss}}$ which induces the morphism of variation of GIT (VGIT) $\xi_{\theta}^{\lambda} \colon Y_{\theta} \to Y_{\lambda}$. We now prove the following result, which informally states that, if $\lambda$ is dominant, then $Y_{\lambda}$ can be identified with an appropriate $\tcN^P$ in a way that is compatible with the map from the Springer resolution:

\begin{Proposition}\label{Generically Isomorphisms Are What We Want}
For any parabolic subgroup $P$ and any dominant $\lambda \in \characterlatticeforT$ such that $\lambda(\alpha^{\vee}) = 0$ if and only if $\alpha^{\vee} \in S_{P}$, there is an isomorphism $\zeta_{\lambda} \colon \tcN^{P}\xrightarrow{\sim} Y_{\lambda}$. Moreover, if $\theta \in \characterlatticeforT$ is dominant regular, the diagram \begin{equation}\label{Compatibility of Key Isomorphism}\xymatrix@R+2em@C+2em{\tcN \ar[r]^{\nu^{P}} \ar[d]^{\zeta_{\theta}} & \tcN^{P} \ar[d]^{\zeta_{\lambda}}\\
\momentMapFromAFFINECLOSUREofCotangentSpaceWithGROUPT^{-1}(0)\sslash_{\theta} T \ar[r]^{\xi^{\lambda}_{\theta}} & Y_{\lambda} 
  }\end{equation} commutes, where $\zeta_{\theta}$ is the isomorphism of \cref{Hamiltonian Reduction at Dominant Regular Level for Affine Closure of Cotangent Bundle}. More generally, if $P'$ is a parabolic subgroup containing $P$ and $\lambda' \in \characterlatticeforT$ is a dominant weight such that $\lambda'(\alpha^{\vee}) = 0$ if and only if $\alpha^{\vee} \in S_{P'}$ then \begin{equation}\label{Compatibility of Key Isomorphism For Any Parabolics}\xymatrix@R+2em@C+2em{\tcN^{P} \ar[r]^{\nu^{P'}_{P}} \ar[d]^{\zeta_{\lambda}} & \tcN^{P'} \ar[d]^{\zeta_{\lambda'}}\\
Y_{\lambda} \ar[r]^{\xi^{\lambda'}_{\lambda}} & Y_{\lambda'} 
  }\end{equation} commutes. 
\end{Proposition}

Before proving \cref{Generically Isomorphisms Are What We Want}, we highlight a special case that we will use in what follows, which follows immediately from the fact that $\tcN^G = \mathcal{N}$:

\begin{Corollary}\label{Hamiltonian Reduction of affineClosureOfCotangentBundleofBasicAffineSpace for T is Nilcone}
The Hamiltonian reduction of $\affineClosureOfCotangentBundleofBasicAffineSpace$ for the Hamiltonian $T$-action $\momentMapFromAFFINECLOSUREofCotangentSpaceWithGROUPT^{-1}(0)\sslash T$ is isomorphic to $\mathcal{N}$.
\end{Corollary}

We now show \cref{Generically Isomorphisms Are What We Want}; this proof will occupy the entirety of \cref{Explicit Description of GIT Resolutions Subsection}. We first prove the following:

\begin{Lemma}\label{Appropriate Bundles Are Ample or Basepoint Free}
Fix an arbitrary parabolic subgroup $P$ with associated set of simple coroots $S_P$, and let $\lambda$ denote some dominant weight for which $\lambda(\alpha^{\vee}) = 0$ for all $\alpha^{\vee} \in S_P$ and $\lambda(\alpha^{\vee}) > 0$ when $\alpha^{\vee} \notin S_P$. The line bundle $\L_{\lambda}^{\tcN^P}$ is ample. Moreover, if $\lambda'$ is any dominant weight for which $\lambda'(\alpha^{\vee}) = 0$ for all $\alpha^{\vee} \in S_P$, $\L_{\lambda'}^{\tcN^P}$ is semiample.
\end{Lemma}

\begin{proof}
    By \cref{Ample Cone for Partial Flag Variety}, the line bundle $\L_{\lambda}^{G/P}$ is ample on $G/P$. Since the morphism $\tcN^P \to G/P$ is affine (which can be checked on $G$-translates of the big cell of $G/P$, where this assertion is obvious) the bundle $\L_{\lambda}^{\tcN^P}$ is ample since the pullback of any ample bundle by an affine morphism is ample \cite[Proposition 5.1.12]{EGAII}. We deduce that $\L_{\lambda}^{\tcN^P}$ is ample. Let $P' = P_{\lambda'}$. The final claim follows from Lemma~\ref{lem:pullbackgeneratedisgenerated} since $\L_{\lambda'}^{G/P'}$ is ample and $\L_{\lambda'}^{\tcN^P}$ is the pullback of $\L_{\lambda'}^{G/P'}$ under $\tcN^P \to G/P \to G/P'$.
\end{proof}

For any dominant weight $\lambda \in \characterlatticeforT$ such that $\lambda(\alpha^{\vee}) = 0$ for all $\alpha^{\vee} \in S_{P}$, respectively any dominant weight $\theta$, we let \[i^{P}_{\lambda}: \tcN^{P} \xrightarrow{} \Proj_\N(R(\tcN^P,\mathcal{L}_{\lambda}^{\tcN^{P}}))\text{ and  }i_{\theta}: \tcN \xrightarrow{} \Proj_\N(R(\tcN, \mathcal{L}_{\theta}^{\tcN}))\] denote the corresponding morphisms given by the semiample locally free sheaves $\mathcal{L}_{\lambda}^{\tcN^{P}}$ and $\mathcal{L}_{\theta}^{\tcN}$ respectively. 

\begin{Lemma}\label{Connection of Hamiltonian Reductions and Resolutions of Nilpotent Cones}
For any dominant weight $\lambda \in \characterlatticeforT$, there is an isomorphism \[\varphi_{\lambda}: \Proj_\N(R(\tcN, \mathcal{L}^{\tcN}_{\lambda})) \xrightarrow{\sim} Y_{\lambda}\] such that \begin{equation}\label{Commutative Diagram from Restriction}\xymatrix@R+2em@C+2em{\tcN  \ar[d]^{i_{\lambda}} \ar[r]^{\zeta_{\theta}} & \momentMapFromAFFINECLOSUREofCotangentSpaceWithGROUPT^{-1}(0)\sslash_{\theta} T \ar[d]^{\xi^{\lambda}_{\theta}}  \\
\Proj_{\mathcal{N}}(R(\tcN, \L_{\lambda}^{\tcN})) \ar[r]^{\varphi_{\lambda}}  & Y_\lambda
  }\end{equation} commutes for any regular dominant $\theta \in \characterlatticeforT$. 
\end{Lemma}

\begin{proof} 
Recall that $\widetilde{\mathfrak{g}} = G \times^B (\LG / \mathfrak{n})^*$ is the domain of Grothendieck's simultaneous resolution. Recall also our notation for line bundles from \cref{Line Bundle Notation Subsection}. Since $n\lambda$ is dominant, the natural map 
\[
\Gamma(\widetilde{\mathfrak{g}},\mathcal{L}_{n\lambda}^{\widetilde{\LG}}) \otimes_{\mathrm{Sym}(\LT)} k \to \Gamma(\tcN,\mathcal{L}_{n\lambda}^{\tcN}) 
\]
is an isomorphism by the arguments of \cite[Proposition~3.2.3]{GinzburgRicheDifferentialOperatorsOnBasicAffineSpaceandtheAffineGrassmannian}; see also the proof of \cite[Lemma~3.6.2]{GinzburgRicheDifferentialOperatorsOnBasicAffineSpaceandtheAffineGrassmannian}. Thus the composite 
\[
\Proj_{\mathcal{N}}(R(\tcN, \mathcal{L}^{\tcN}_{\lambda})) = \Proj_{\mathcal{N}}(\oplus_{n \geq 0}\Gamma(\mathcal{L}^{\tcN}_{n\lambda})) \xrightarrow{\sim} \Proj_{\mathcal{N}}(\oplus_{n \geq 0}\Gamma(\widetilde{\LG},\mathcal{L}^{\widetilde{\LG}}_{n\lambda}) \otimes_{\Symt} k)\] \[= \Proj_{\mathcal{N}}(\oplus_{n \geq 0} (\ringOfFunctionsForCOTANGENTBUNDLEOfBasicAffineSpace_{n\lambda} \otimes_{\Symt} k)) =  \Proj_{\mathcal{N}}(\oplus_{n \geq 0}(\ringOfFunctionsForCOTANGENTBUNDLEOfBasicAffineSpace \otimes_{\Symt} k)_{n\lambda}) =: Y_{\lambda}\]
gives our desired isomorphism. Here we have used the fact, as shown in the proof of \cite[Lemma~3.6.2]{GinzburgRicheDifferentialOperatorsOnBasicAffineSpaceandtheAffineGrassmannian}, that $\oplus_nR_{n\lambda} = \oplus_{n \geq 0}\Gamma(\widetilde{\LG},\mathcal{L}^{\widetilde{\LG}}_{n\lambda})$ and that $\mathcal{O}(\overline{\mu}_T^{-1}(0))_{\lambda} = (R \otimes_{\Symt} k)_{\lambda}$.    
\end{proof}

\begin{Lemma}\label{isotildeNP}
There is an isomorphism $(\nu^{P})^* \mathcal{L}^{\widetilde{\mathcal{N}}^{P}}_{\lambda} \cong \mathcal{L}^{\widetilde{\mathcal{N}}}_{\lambda}$. 
\end{Lemma}

\begin{proof}
    Notice that in addition to the commutative diagram
    \[
    \begin{tikzcd}
        \tcN \ar[rr,"\nu^{P}"] \ar[dr] & & \tcN^{P} \ar[dl] \\
        & \mathcal{N} &
    \end{tikzcd}
    \]
    we also have a commutative diagram 
        \begin{equation}\label{eq:GBtoGP}
    \begin{tikzcd}
        \tcN \ar[r,"\nu^{P}"] \ar[d,"\pi_B"] & \tcN^{P} \ar[d,"\pi_P"'] \\
        G / B \ar[r,"q"] \arrow[u, bend left=50,"i_B"] & G/P \arrow[u, bend right=50,"i_P"']
    \end{tikzcd}         
        \end{equation}
    and $\mathcal{L}^{\tcN^{P}}_{\lambda} = \pi_P^* \mathcal{L}^{G/P}_{\lambda}$, $\mathcal{L}^{\tcN}_{\lambda} = \pi_B^* \mathcal{L}^{G/B}_{\lambda}$, so it suffices to note that $q^* (\mathcal{L}^{G/P}_{\lambda}) = \mathcal{L}^{G/B}_{\lambda}$. To see this, observe that
    \[
        \begin{tikzcd}
      G \times^B \mathbb{A}^1_{-\lambda} \ar[r] \ar[d] & G \times^{P} \mathbb{A}^1_{-\lambda} \ar[d] \\
        G / B \ar[r,"q"] & G/P 
    \end{tikzcd}
    \]
    is Cartesian--indeed, it is not difficult to check this after restricting to the big cell \begin{equation}\label{Big Cell for Partial Flag Variety}\overline{U}_{P} \cong \overline{U}_{P}P/P \subseteq G/P\end{equation} of the parabolic Bruhat decomposition and, since each map in the above diagram is $G$-equivariant, this implies that this diagram is Cartesian after restricting to any $G$-translate of the big cell \labelcref{Big Cell for Partial Flag Variety}.
\end{proof}

\begin{proof}[Proof of \cref{Generically Isomorphisms Are What We Want}] The variety $\tcN^{P}$ is normal by \cref{tcNP is normal}. Therefore, from \cref{For Proper Birational Maps Unit Map Is Iso on Line Bundles} and \cref{isotildeNP} we obtain an isomorphism 
\[\mathcal{L}_{\lambda}^{\tcN^{P}} \xrightarrow{\sim} (\nu^{P})_*(\nu^{P})^* (\mathcal{L}_{\lambda}^{\tcN^{P}}) \cong (\nu^{P})_*(\mathcal{L}_{\lambda}^{\tcN})\] and a similar isomorphism for every tensor power of $\mathcal{L}_{\lambda}^{\tcN^{P}}$, which therefore induces an isomorphism \[\eta_{P}: \Proj_\N(R(\tcN, \mathcal{L}_{\lambda}^{\tcN})) \xrightarrow{\sim} \Proj_\N(R(\tcN^{P}, \mathcal{L}_{\lambda}^{\tcN^{P}}))\] of varieties over $\mathcal{N}$. Observe that, by construction, \begin{equation}\label{The Map on Proj Induced by Nu}
        \begin{tikzcd}\
      \tcN^P  \ar[d, "i^{P}_{\lambda}"] & \tcN \ar[l, "\nu^{P}"] \ar[d, "i_{\lambda}"] \\
    \Proj_\N(R(\tcN^{P}, \L_{\lambda}^{\tcN^{P}})) &     \Proj_\N(R(\tcN, \L_{\lambda}^{\tcN})) \ar[l,"\eta_{P}"]  
    \end{tikzcd}
    \end{equation} commutes. 
    Since the dominant character $\lambda \in \characterlatticeforT$ is assumed to satisfy $\lambda(\alpha^{\vee}) = 0$ if and only if $\alpha^{\vee} \in S_{P}$, the line bundle $\mathcal{L}_{\lambda}^{\tcN^{P}}$ is ample by Lemma~\ref{Appropriate Bundles Are Ample or Basepoint Free}, and hence $i^P_{\lambda}$ is an isomorphism. Our desired isomorphism is given by the composite $\varphi_{\lambda} \circ \eta^{-1}_{P} \circ i^{P}_{\lambda}$, and the commutativity of the diagram \labelcref{Compatibility of Key Isomorphism} follows formally from combining the commutativity of \labelcref{Commutative Diagram from Restriction} and \labelcref{The Map on Proj Induced by Nu}. 

    Finally, we prove that \labelcref{Compatibility of Key Isomorphism For Any Parabolics} commutes. First, observe that we obtain a commutative diagram  \begin{equation}\xymatrix@R+2em@C+2em{\tcN \ar@/^1pc/[rr]^{\nu^{P'}} \ar[d]^{\zeta_{\theta}}\ar[r]_{\nu^{P}} & \tcN^{P} \ar[r]_{\nu^{P'}_{P}} & \tcN^{P'} \ar[d]^{\zeta_{\lambda'}}\\ Y_{\theta} \ar@/_1pc/[rr]_{\xi_{\theta}^{\lambda'}} \ar[r]^{\xi_{\theta}^{\lambda}} & 
Y_{\lambda} \ar[r]^{\xi^{\lambda'}_{\lambda}} & Y_{\lambda'} 
  }\end{equation} 
  by using the fact that the diagram analogous to \labelcref{Compatibility of Key Isomorphism} commutes when $P$ is replaced with the parabolic $P'$, the definition of $\nu^{P'}$, and the universal property of the categorical quotient. Now, the map $\nu^P$, and thus $\xi_{\theta}^{\lambda'}$, are birational. Therefore the locus where the diagram \labelcref{Compatibility of Key Isomorphism For Any Parabolics} commutes is open in $\tcN^{P}$ and nonempty. Moreover, since $\nu_{P}^{P'}$ is a separated map, 
  the locus where \labelcref{Compatibility of Key Isomorphism For Any Parabolics} commutes is also closed. Therefore, since $\tcN^{P}$ is irreducible, we see that the composites in the diagram \labelcref{Compatibility of Key Isomorphism For Any Parabolics} agree everywhere.
\end{proof}

\begin{Lemma}\label{Picard of tcnP}
Pullback $(\pi^P)^* \colon \Pic(G/P) \to \Pic(\tcN^P)$ is an isomorphism which identifies the Picard group of $\tcN^P$ with the subgroup of $ \characterlatticeforT$ consisting of all $\eta$ for which $\eta(\alpha^{\vee}) = 0$ for all $\alpha^{\vee} \in S_P$. 
\end{Lemma}

\begin{proof}
We will use diagram \eqref{eq:GBtoGP} to prove that $(\pi^P)^*$ is an isomorphism on Picard groups. Since $\pi^P \circ i_P = \mathrm{id}_{G/P}$, $(\pi^P)^*$ is an isomorphism if and only if the pullback $i_P^* \colon \Pic(\tcN^P) \to \Pic(G/P)$ is an isomorphism. We show that the latter holds.  

First we note that $(\nu^P)^* \colon \Pic(\tcN^P) \to \Pic(\tcN)$ is injective. Indeed, using that $\nu^P_* \mathcal{O}_{\tcN} \cong \mathcal{O}_{\tcN^P}$ (by \cref{cor:semiample-induced-morphism-corollary} applied to \cref{tcNP is normal}) and the projection formula \cite[Lemma 20.54.2, Tag 01E8]{StacksProject}, if $\mathcal{L} \in \Pic(\tcN^P)$ then 
\[
\nu^P_* (\nu^P)^* \mathcal{L} \cong \mathcal{L} \otimes (\nu^P)_* \mathcal{O}_{\tcN} \cong \mathcal{L}. 
\]
Moreover, in diagram \eqref{eq:GBtoGP}, $i_B^* \colon \Pic(\tcN) \to \Pic(G/B)$ and $\pi^*_B \colon \Pic(G/B) \to \Pic(\tcN)$ are isomorphisms, as explained in the proof of \cref{Picard Group of Vector BUndles on Partial Flag Variety}. Since $\pi^P \circ i_P = \mathrm{id}$, the pullback map $i_P^*$ is injective. Therefore, it suffices to show that if $\mathcal{L} \in \Pic(\tcN^P)$ then $i_B^* (\nu^P)^* (\pi^P)^* i_P^* \mathcal{L} \cong i_B^* (\nu^P)^*  \mathcal{L}$. But using the commutativity of \eqref{eq:GBtoGP} and the fact that $\pi^B \circ i_B = \mathrm{id}_{G/B}$, one has
\[
i_B^* (\nu^P)^* (\pi^P)^* i_P^* \mathcal{L} \cong i_B^* (\pi^P)^* q^* i_P^* \mathcal{L} \cong q^* i_P^* \mathcal{L} \cong i_B^* (\nu^P)^*  \mathcal{L}, 
\]
as required. The second claim of the lemma now follows from \cref{Picard Group of Vector BUndles on Partial Flag Variety}(2).
\end{proof}

\begin{Lemma}\label{lem:LtcNPiso}
    For any dominant weight $\lambda \in \characterlatticeforT$, the linearization map $L_{\Z} \colon \mathrm{St}(\lambda)^{\perp} \to \operatorname{Pic}(Y_{\lambda})$ and its realization $L \colon \mathrm{St}(\lambda)^{\perp} \otimes_{\mathbb{Z}} \mathbb{R} \to \operatorname{Pic}(Y_{\lambda} / Y_0)_{\mathbb{R}}$ are isomorphisms. 
\end{Lemma}

\begin{proof}
Let $P = P_{\lambda}$. Let $\mathrm{S}(\lambda) \subset \characterlatticeforT$ be the set of all $\sigma$ such that $\sigma(\alpha^{\vee}) = 0$ for all $\alpha^{\vee} \in S_P$. By \cref{cor:stthetaequalset}, we have an equality $\mathrm{St}(\lambda)^{\perp} = \mathrm{S}$. We claim that under the isomorphism $\zeta_{\lambda} \colon \tcN^P \to Y_{\lambda}$ of \cref{Generically Isomorphisms Are What We Want}, there is a commutative diagram 
\[
\begin{tikzcd}
  \mathrm{St}(\lambda)^{\perp} \ar[r,equal]  \ar[d,"L_{\Z}"'] & \mathrm{S}(\lambda) \ar[d] \\
   \Pic(Y_{\lambda}) \ar[r,"\zeta_{\lambda}^*"] & \Pic(\tcN^P) 
\end{tikzcd}
\]
where the right hand vertical map is $\lambda \mapsto \mathcal{L}_{\sigma}^{\tcN^P}$, as defined in Section~\ref{Line Bundle Notation Subsection}. Indeed, this follows from the fact that under the quotient morphism 
\[
\momentMapFromAFFINECLOSUREofCotangentSpaceWithGROUPT^{-1}(0)^{\lambda} \to Y_{\lambda} \stackrel{\zeta_{\theta}^{-1}}{\longrightarrow} \tcN^P
\]
the pullback of $\mathcal{L}_{\sigma}^{\tcN^P}$ is isomorphic to $\mathcal{O}_{\momentMapFromAFFINECLOSUREofCotangentSpaceWithGROUPT^{-1}(0)^{\lambda}} \otimes \sigma$, since the total space of the latter equals $\momentMapFromAFFINECLOSUREofCotangentSpaceWithGROUPT^{-1}(0)^{\lambda} \times \mathbb{A}^1_{-\sigma}$ as a scheme with $T$-action. Therefore, the fact that $L_{\Z}$ is an isomorphism follows from \cref{Picard of tcnP}, which says that the right hand vertical map is an isomorphism. 

Finally, \cref{Generically Isomorphisms Are What We Want} shows that $Y_0 \cong \mathcal{N}$ is a normal affine cone and hence has trivial Picard group. Therefore, the second isomorphism follows immediately from the first. 

\end{proof}

\begin{Corollary}\label{Movable Cone of tcnP}
Under the identification given in \cref{Picard of tcnP}, the movable cone of $\tcN^P$ is the subset of all weights $\eta$ that are dominant.
\end{Corollary}

\begin{proof}
The result will follow from \cref{cor:MoveconeGIT} once we check that all the hypothesis of that result hold. Choose dominant $\lambda \in \characterlatticeforT$ such that $P = P_{\lambda}$. Then $\tcN^P \cong Y_{\lambda}$ by Proposition~\ref{Generically Isomorphisms Are What We Want} and hence the variety $Y_{\lambda}$ is normal and $\Q$-factorial by Lemma~\ref{tcNP is normal}. Next we note that the linearization map $L$ is an isomorphism for $Y_{\lambda}$ by Lemma~\ref{lem:LtcNPiso}. 

Next, we note that $\xi^0_{\lambda}$ is birational by Lemma~\ref{tcNP is normal}  (3) and \cref{Generically Isomorphisms Are What We Want}. Let $C$ denote the set of all $\sigma$ in $\Pic(\tcN^P)$ such that $\sigma$ is dominant. This is a GIT cone by \cref{Wall and Chamber Structure for T action on affineClosureOfCotangentBundleofBasicAffineSpace}. \cref{cor:divisorialcontraction} shows that for $\sigma$ in the boundary of $C$, and $Q = P_{\sigma}$ the map $\xi^{\sigma}_{\lambda} = \nu^{Q}_P$ contracts at least one divisor. Therefore, the hypothesis of Corollary~\ref{cor:MoveconeGIT} hold for $Y_{\lambda} = \tcN^P$.
\end{proof}

\subsection{Proof of Main Theorem}

In this section, we state and prove a generalization of \cref{Intro Main Theorem}. 
If $f \colon Y \to X$ is a relative Mori dream space then the small $\Q$-factorial modifications $Y_i \to X$ of $f$ are all isomorphic in codimension one, $\phi_i \colon Y \dashrightarrow Y_i$, and hence there is an identification $\phi_i^* \colon \Pic(Y_i/X)_{\R} \cong \Pic(Y/X)_{\R}$. The Mori fan $\mathcal{M}(Y / X)$ is the fan whose cones consist of all faces of all the Nef cones $\phi_i^* \mathrm{Nef}(Y_i/X)_{\R}$, the union over all small $\Q$-factorial modifications. The support of $\mathcal{M}(Y / X)$ is the movable cone $\mathrm{Mov}(Y/ X)$. 

\begin{Theorem}\label{Copy of BCS Theorem in Our SetupP}
Let $\lambda \in \characterlatticeforT$ be a dominant weight and $P$ the parabolic subgroup corresponding to $S_P := \{\alpha^{\vee} \in \Pi^{\vee} : \lambda(\alpha^{\vee}) = 0\}$. 
\begin{enumerate}
\item There is an isomorphism $\momentMapFromAFFINECLOSUREofCotangentSpaceWithGROUPT^{-1}(0)\sslash_{\lambda} T \cong \tcN^{P}$ of varieties over $\mathcal{N}$.
 \item The variety $\tcN^P$ is a Mori dream space over $\mathcal{N}$.
    \item The linearization map $\sigma \mapsto \L_{\sigma}^{\tcN^P}$ induces an isomorphism \[S_P^{\perp} \xrightarrow{\sim} \mathrm{Pic}(\tcN^P)\] of abelian groups that, after realization, identifies the wall-and-chamber decomposition of $C := \{ \sigma \in (S_P^{\perp})_{\R} \, | \, \textrm{$\sigma$ dominant} \, \}$ with the Mori fan $\mathcal{M}(\tcN^P/ \mathcal{N})$.  
\end{enumerate}
\end{Theorem}

\begin{proof}[Proof of \cref{Copy of BCS Theorem in Our SetupP}]
Part (1) is \cref{Generically Isomorphisms Are What We Want} and part (2) is \cref{prop:Moridreamspace}. The fact that the linearization map is an isomorphism identifying the cone $C$ with the movable cone is \cref{Picard of tcnP} and \cref{Movable Cone of tcnP}. Therefore, the only thing to show is the last claim of part (3). However, combining \cref{Generically Isomorphisms Are What We Want} with \cref{Movable Cone of tcnP} shows that $\mathrm{Mov}(\tcN^P) = C = \mathrm{Nef}(\tcN^P)$. Therefore, the Mori fan is precisely the set of all faces of $C$. 

\end{proof}

It is shown in \cite[Theorem~4.4]{Ohta} that the cones in the Mori fan $\mathcal{M}(Y/X)$ are in bijection with the \textit{rational contractions} of $Y$ over $X$ (see \cite[Definition~2.8]{Ohta}). In the case $\mathrm{Mov}(\tcN^P/\mathcal{N}) = \mathrm{Nef}(\tcN^P/\mathcal{N})$ this implies that all rational contractions are regular and are precisely the surjective projective morphisms from $\tcN^P$ to a normal variety $Z$ that factor as $\tcN^P \to Z \to \mathcal{N}$. 

\begin{Remark}\label{Partial Crepant Resolutions Are Fces of Mori Fan}
If $\lambda$ is dominant regular then $\tcN^P = \tcN$. It is well known (see, for example, \cite[Corollary 4.10]{SanMiguelMalaneyPartialResolutionsOfAffineSymplecticSingularities}\footnote{See also the remark just after Lemma 4.5, where \lq partial resolutions\rq{} are defined to be partial symplectic resolutions and thus by \cite[Theorem 4.4]{SanMiguelMalaneyPartialResolutionsOfAffineSymplecticSingularities} are crepant.}) that crepant partial resolutions of $\mathcal{N}$ are in bijective correspondence with the faces of the Mori fan $\mathcal{M}(\tcN/ \mathcal{N})$. 
Therefore, \cref{Copy of BCS Theorem in Our SetupP} for $\lambda$ dominant regular implies \cref{Intro Main Theorem}. 
\end{Remark}

\section{Integrality and Normality of $\momentMapFromAFFINECLOSUREofCotangentSpaceWithGROUPT^{-1}(0)$}\label{sec:muinv0} In this section we study the geometry of $\momentMapFromAFFINECLOSUREofCotangentSpaceWithGROUPT^{-1}(0)$. Namely, we establish the following somewhat surprising:

\begin{Theorem}\label{Preimage of affinized moment map is integral}
$\momentMapFromAFFINECLOSUREofCotangentSpaceWithGROUPT^{-1}(0)$ is a normal integral Cohen-Macaulay scheme.
\end{Theorem}

We do not use these properties of $\momentMapFromAFFINECLOSUREofCotangentSpaceWithGROUPT^{-1}(0)$ elsewhere in the article. We prove that $\momentMapFromAFFINECLOSUREofCotangentSpaceWithGROUPT^{-1}(0)$ is Cohen-Macaulay in \cref{Cohen Macaulay Subsubsection}. After some preliminary work in \cref{Irreducible Components Meeting W Translates of Cotangent Bundle}, we prove that $\momentMapFromAFFINECLOSUREofCotangentSpaceWithGROUPT^{-1}(0)$ is irreducible in \cref{Proof of Irreduciblility}. Finally, we show that $\momentMapFromAFFINECLOSUREofCotangentSpaceWithGROUPT^{-1}(0)$ is normal in \cref{Normality Subsubsection}. 

\subsection{Cohen-Macaulay-ness}\label{Cohen Macaulay Subsubsection}To prove the Cohen-Macaulay-ness in \cref{Preimage of affinized moment map is integral}, we first recall the following result, which follows immediately from specialization of \cite[Proposition 3.2.3]{GinzburgRicheDifferentialOperatorsOnBasicAffineSpaceandtheAffineGrassmannian} to $\hbar = 0$:

\begin{Lemma}\label{Affinized moment map is flat}
    The map $\momentMapFromAFFINECLOSUREofCotangentSpaceWithGROUPT: \affineClosureOfCotangentBundleofBasicAffineSpace \to \LTd$ is flat.
\end{Lemma}

We also recall the main theorem of \cite{GannonProofOftheGinzburgKazhdanConjecture}: 

\begin{Theorem} The scheme $\affineClosureOfCotangentBundleofBasicAffineSpace$ has symplectic singularities and, in particular, is Cohen-Macaulay. 
\end{Theorem}

Thus, by the miracle flatness package, see for example \cite[Exercise 26.2.H]{VakilRisingSeaFoundationsofAlgebraicGeometry}, since $\affineClosureOfCotangentBundleofBasicAffineSpace$ is Cohen-Macaulay and $\momentMapFromAFFINECLOSUREofCotangentSpaceWithGROUPT$ is flat, we obtain:
\begin{Corollary}\label{Momemt map inverse is CM and pure dimension of G}
    The scheme $\momentMapFromAFFINECLOSUREofCotangentSpaceWithGROUPT^{-1}(0)$ is Cohen-Macaulay and pure of dimension $\mathrm{dim}(G)$.
\end{Corollary}

\subsection{Irreducible components that meet $w(T^*(G/U))$}\label{Irreducible Components Meeting W Translates of Cotangent Bundle}
Write $X := \momentMapFromAFFINECLOSUREofCotangentSpaceWithGROUPT^{-1}(0)$. Since the open subset \begin{equation}\label{Inverse image of zero is irreducible for moment map not affinized}G \times^U (\LG / \LB)^* \xrightarrow{\sim} (G \times^U (\LG / \LU)^*) \times_{\LT^*} \{0\} = \mu_T^{-1}(0) = \momentMapFromAFFINECLOSUREofCotangentSpaceWithGROUPT^{-1}(0) \cap T^*(G/U)\end{equation}
of $X$ is irreducible, we see that there exists a unique irreducible component $Z_1$ of $X$ such that $Z_1 \cap T^*(G/U)$ is nonempty. 
Moreover, since $X$ is closed under the $W$-action, for any $w \in W$ the subscheme $Z_w := w(Z_1)$ clearly satisfies the following Lemma:

\begin{Lemma}\label{Definition of Zw}
There exists a unique irreducible component $Z_w$ of $X$ such that $Z_w \cap w(T^*(G/U))$ is nonempty.
\end{Lemma} 

\begin{Proposition}
    For all $w, w' \in W$, we have $Z_w = Z_{w'}$.
\end{Proposition}

\begin{proof}
    It obviously suffices to show this when $w' = 1$. Observe that \[X\sslash T \xrightarrow{\sim} \left( \affineClosureOfCotangentBundleofBasicAffineSpace\sslash T \right)\times_{\LTd} \{0\} = (\LGd \times_{\LGd\sslash G} \LTd) \times_{\LTd} \{0\} \cong \N\] since tori are reductive and via a result of Ginzburg-Riche \cite[Lemma 3.6]{GinzburgRicheDifferentialOperatorsOnBasicAffineSpaceandtheAffineGrassmannian} stated explicitly in \cite[Lemma 2.2]{GannonProofOftheGinzburgKazhdanConjecture}. Therefore the restriction of the moment map  \[\momentMapFromAFFINECLOSUREofCotangentSpaceWithGROUPG: X \to X\sslash T \subset \mathfrak{g}^* \]
    for the Hamiltonian $G$-action on $X$ is a dominant map onto an irreducible scheme. Finally, since $Z_1 \cong G \times^{U} \LU$, the image of $\momentMapFromAFFINECLOSUREofCotangentSpaceWithGROUPG$ contains a regular element, and therefore $\momentMapFromAFFINECLOSUREofCotangentSpaceWithGROUPG |_{Z_1} \colon Z_1 \to X \sslash T$ is dominant. Now, by $W$-equivariance, since $\momentMapFromAFFINECLOSUREofCotangentSpaceWithGROUPG: Z_1 \to \N$ is dominant, $\momentMapFromAFFINECLOSUREofCotangentSpaceWithGROUPG: Z_w \to \N$ is dominant. In particular, \[Z_w \cap (\affineClosureOfCotangentBundleofBasicAffineSpace \times_{\LGd} \LGd_{\mathrm{reg}})\] is nonempty. However, by \cite[Proposition 5.1.4]{Gin}, $\affineClosureOfCotangentBundleofBasicAffineSpace \times_{\LGd} \LGd_{\mathrm{reg}} \subseteq T^*(G/U)$. Therefore, since the defining property of $Z_1$ is that it is the unique irreducible component whose intersection with $T^*(G/U)$ is nonempty, we see that $Z_w = Z_1$ as desired. 
\end{proof}

We hereafter omit the subscript and let $Z := Z_1$. Thus, by \cref{Definition of Zw}, $Z$ is the unique irreducible component of $\momentMapFromAFFINECLOSUREofCotangentSpaceWithGROUPT^{-1}(0)$ such that $Z \cap w(T^*(G/U))$ is non-empty for some (equivalently, any) $w \in W$.

\subsection{Proof of Irreducibility}
\label{Proof of Irreduciblility}

Now we establish that $X$ is irreducible. Assume that there is an irreducible component $V$ of $X$ other than $Z$. 

\begin{Lemma}\label{For extra irreducible component weight cone is small}
    The perpendicular $\omega(V)^{\perp}$ to the weight cone contains some non-zero $\gamma \in X_{\bullet}(T)$.
\end{Lemma}

\begin{proof}
First, we claim that the weight cone $\omega(V)$ does not contain any regular weight. Indeed, if $\lambda \in \omega(V)$ for some regular weight $\lambda$, then there exists $n > 0$ and $f \in \O(\affineClosureOfCotangentBundleofBasicAffineSpace)_{n \lambda}$ whose image in $\O(V)_{n \lambda}$ is non-zero. Then $D(f) \cap V$ is a nonempty open subset of $V$. Let $w \in W$ denote the unique element such that $w^{-1}\lambda$ is dominant. By \cite[Proposition 5.10]{GannonProofOftheGinzburgKazhdanConjecture}, $D(f) \subseteq w(T^*(G/U))$. Therefore $D(f) \cap V$ is a nonempty subset of $w(T^*(G/U))$. Thus, by \cref{Definition of Zw}, $V = Z$, which contradicts our choice of $V$. Thus the weight cone cannot span all of $\characterlatticeforT_{\Q}$. Since $\omega(V)$ is closed under scaling, our claim follows. 
\end{proof}

Assume now that the kernel of the action map of $T$ on $V$ is $K$, and let $T' := T/K$. By \cref{Elements in Interior of GIT Cone Are Generic}, since $V$ is irreducible, there exists $\lambda \in X^{\bullet}(T') \subset X^{\bullet}(T)$ such that $V$ contains a $\lambda$-stable point.\footnote{Note that this point is not stable with regards to $T$ since $T \neq T'$ by \cref{For extra irreducible component weight cone is small}.} Using the $W$-action, may assume $\lambda$ is dominant.

Now consider $V \sslash_{\lambda} T' = V \sslash_{\lambda} T$. Since there is a $\lambda$-stable point,
\begin{equation}\label{Dimension is too large for GIT reasons}
\dim V \sslash_{\lambda} T' = \dim V - \dim T'  = \dim G - \dim T' > \dim G - \dim T
\end{equation} 
by \cref{Momemt map inverse is CM and pure dimension of G}, and \cref{For extra irreducible component weight cone is small} (which implies $K$ has dimension at least one) respectively.

On the other hand, we have a quotient map 
\[
R' := \bigoplus_n \mathcal{O}(X)_{n \lambda} \longrightarrow \bigoplus_n \mathcal{O}(V)_{n \lambda} =: R''
\]
and we know that $\dim \Proj \, R' = \dim \N$ since $\Proj\, R' \cong \tcN^P$ by \cref{Generically Isomorphisms Are What We Want}, so $\dim R' = \dim \N + 1$. Thus $\dim R'' \leq \dim \N + 1$, and so \begin{equation}\label{dimension is less than nilpotent cone}\dim V\sslash_{\lambda} T' \leq \dim \N = \dim G - \dim T
\end{equation} as well. Combining \labelcref{Dimension is too large for GIT reasons} and \labelcref{dimension is less than nilpotent cone}, we obtain a contradiction. Therefore $Z = V$ and $\momentMapFromAFFINECLOSUREofCotangentSpaceWithGROUPT^{-1}(0)$ is irreducible.

Finally, we show that $\momentMapFromAFFINECLOSUREofCotangentSpaceWithGROUPT^{-1}(0)$ is reduced. Observe that, since $\momentMapFromAFFINECLOSUREofCotangentSpaceWithGROUPT^{-1}(0)$ is Cohen-Macaulay by \cref{Momemt map inverse is CM and pure dimension of G}, $\momentMapFromAFFINECLOSUREofCotangentSpaceWithGROUPT^{-1}(0)$ has property $(S_1)$. Moreover, the ring is generically reduced since the variety $G \times^U (\LG/\LB)^*$
 is an open subset of $\momentMapFromAFFINECLOSUREofCotangentSpaceWithGROUPT^{-1}(0)$, see \labelcref{Inverse image of zero is irreducible for moment map not affinized}. Therefore, by \cite[Lemma 28.12.4, Tag 0344]{StacksProject}, we obtain the following, which completes the proof that $\momentMapFromAFFINECLOSUREofCotangentSpaceWithGROUPT^{-1}(0)$ is integral:

\begin{Corollary}
    The scheme $\momentMapFromAFFINECLOSUREofCotangentSpaceWithGROUPT^{-1}(0)$ is reduced. 
\end{Corollary}

\subsection{Normality of  $\momentMapFromAFFINECLOSUREofCotangentSpaceWithGROUPT^{-1}(0)$}\label{Normality Subsubsection} We now show the following, which completes the proof of \cref{Preimage of affinized moment map is integral}:
\begin{Theorem}\label{Inverse of zero under momentMapFromAFFINECLOSUREofCotangentSpaceWithGROUPT is normal}
      The scheme $\momentMapFromAFFINECLOSUREofCotangentSpaceWithGROUPT^{-1}(0)$ is normal.
\end{Theorem}

Before proving \cref{Inverse of zero under momentMapFromAFFINECLOSUREofCotangentSpaceWithGROUPT is normal}, we give a sketch. First, by Serre's criterion for normality and our above results, it suffices to show that the singular locus of $\momentMapFromAFFINECLOSUREofCotangentSpaceWithGROUPT^{-1}(0)$ has codimension at least two. Fix a generic point of some irreducible component of the singular locus. Either the stabilizer of this point will have dimension larger than one or it does not. In the former case, one can use some elementary facts about $\affineClosureOfCotangentBundleofBasicAffineSpace$ and geometric invariant theory to argue that the codimension of this component is at least two. If the stabilizer of this point has dimension zero or one, then using computations performed in \cite{GannonProofOftheGinzburgKazhdanConjecture}, one can essentially reduce to the $\SL_2$ case:

\begin{Example}\label{Example for SL2}
    If $G = \SL_2$, then one can identify $\affineClosureOfCotangentBundleofBasicAffineSpace \cong \Spec(k[x, y, \partial_x, \partial_y])$ and $\LTd \cong \A^1$ in such a way where the moment map is given by the function $x\partial_x + y\partial_y$. We deduce that \[\momentMapFromAFFINECLOSUREofCotangentSpaceWithGROUPT^{-1}(0) \cong \Spec(k[x, y, \partial_x, \partial_y]/(x\partial_x + y\partial_y))\] is singular precisely at the point (0,0,0,0). 

    Although this is not used in the proof of \cref{Inverse of zero under momentMapFromAFFINECLOSUREofCotangentSpaceWithGROUPT is normal}, it is interesting to note that this computation implies that the complement of the inclusion $\mu_T^{-1}(0) \subseteq \momentMapFromAFFINECLOSUREofCotangentSpaceWithGROUPT^{-1}(0)$ has codimension \textit{one}.
\end{Example}

\begin{proof}[Proof of \cref{Inverse of zero under momentMapFromAFFINECLOSUREofCotangentSpaceWithGROUPT is normal}]
    We have already shown in \cref{Momemt map inverse is CM and pure dimension of G}  that $\momentMapFromAFFINECLOSUREofCotangentSpaceWithGROUPT^{-1}(0)$ is Cohen-Macaulay and have shown in \cref{Preimage of affinized moment map is integral} that $\momentMapFromAFFINECLOSUREofCotangentSpaceWithGROUPT^{-1}(0)$ is integral, so by Serre's criterion for normality, it suffices to show that the singular locus of $\momentMapFromAFFINECLOSUREofCotangentSpaceWithGROUPT^{-1}(0)$ has codimension at least 2. Let $S$ denote an irreducible component of the singular locus, and let $T'$ denote the connected component of the identity of the stabilizer of the generic point $s$ of $S$. 

    
We recall some facts about the ring $\ringOfFunctionsForCOTANGENTBUNDLEOfBasicAffineSpace := \O(T^*(G/U))$ and its relationship to the ring $\ringOfFunctionsForBasicAffineSpace := \O(G/U)$: \begin{enumerate}
    \item[(A)] The ring $R$ admits a grading $R = \oplus_i R^i$ by non-negative integers, which we will refer to as the \textit{conical grading}, which is compatible with the grading by $\characterlatticeforT$ and has the property that $R^0 = k$ and $\LG \oplus \LT \subseteq \ringOfFunctionsForCOTANGENTBUNDLEOfBasicAffineSpace^2$.
    \item[(B)] For any character $\lambda$, with $\lambda = w(\lambda^+)$ for some dominant weight $\lambda^+$ and $w \in W$, the isotypic component $\ringOfFunctionsForCOTANGENTBUNDLEOfBasicAffineSpace_{\lambda}$ is generated as a module for the ring $R_0 = \Sym(\LG) \otimes_{Z\LG} \Symt$ by $w(\ringOfFunctionsForBasicAffineSpace_{\lambda^+})$. In particular, elements of the form $w(\ringOfFunctionsForBasicAffineSpace_{\lambda^+})$ are minimal with respect to the conical grading.
    \item[(C)] For every $\lambda \in \characterlatticeforT$, the $\lambda$-isotypic component of $R/R\LT$ is nonzero.
\end{enumerate}
Indeed, (A) follows immediately from \cite[Proposition 3.6]{GannonProofOftheGinzburgKazhdanConjecture} and (B) follows immediately from the proof of \cite[Lemma 3.6.2]{GinzburgRicheDifferentialOperatorsOnBasicAffineSpaceandtheAffineGrassmannian}. Moreover, (B) also implies that the $\lambda$-isotypic component $R_{\lambda}$ is nonzero for every $\lambda$. Since the conical grading is compatible with the $\characterlatticeforT$ grading, (C) follows by taking the element of $R_{\lambda}$ of minimal conical grading.

We now assume that $T'$ has positive dimension. Since the generic point of $S$ is fixed by $T'$, $S$ is contained in the fixed locus $S'$ for $T'$. It is not difficult to check that the quotient map \[S' \to \momentMapFromAFFINECLOSUREofCotangentSpaceWithGROUPT^{-1}(0)\sslash T'\] is a closed embedding; this is done for example in \cite[Proposition 4.2]{GannonProofOftheGinzburgKazhdanConjecture}. Therefore $S$ itself is a closed subscheme of $\momentMapFromAFFINECLOSUREofCotangentSpaceWithGROUPT^{-1}(0)\sslash T'$. Since $T'$ has positive dimension, we have that \begin{equation}\label{We Dropped a Dimension One}
\mathrm{dim}(\momentMapFromAFFINECLOSUREofCotangentSpaceWithGROUPT^{-1}(0)\sslash T') \leq \mathrm{dim}(\momentMapFromAFFINECLOSUREofCotangentSpaceWithGROUPT^{-1}(0)) - 1
\end{equation} which is an immediate consequence of the fact that the GIT quotient by a nontrivial $\G_m$-action reduces the dimension (see for example \cite[Lemma 4.1]{GannonProofOftheGinzburgKazhdanConjecture}) and the fact that the $\G_m$-action is nontrivial, which in turn follows immediately from (C) above. If the dimension of $T'$ were at least two, we would moreover have that \begin{equation}\label{We Dropped a Dimension}
\mathrm{dim}(\momentMapFromAFFINECLOSUREofCotangentSpaceWithGROUPT^{-1}(0)\sslash T') \leq \mathrm{dim}(\momentMapFromAFFINECLOSUREofCotangentSpaceWithGROUPT^{-1}(0)) - 2
\end{equation} by the same reason. Therefore, the only remaining case is the case where the dimension of $T'$ is zero or one.

Assume the dimension of $T'$ is zero or one. Then, by \cite[Corollary 5.11]{GannonProofOftheGinzburgKazhdanConjecture}  and \cite[Corollary 5.13]{GannonProofOftheGinzburgKazhdanConjecture},  there exists a $w \in W$ such that $\projectionFromAffineClosureofCotangentBundleToAffineClosureofSpace(ws) \in V_{\alpha}$, where $V_{\alpha}$ is the union in $\affineClosureofBasicAffineSpace$ of the strata $G/U$ and $G/[P_{\alpha}, P_{\alpha}]$. Without loss of generality, we may replace $S$ with $w S$ and assume that $w = 1$. The same argument as in the proof of \cite[Proposition 4.3]{GannonProofOftheGinzburgKazhdanConjecture} gives an isomorphism \[T^*(V_{\alpha}) \cong \affineClosureOfCotangentBundleofBasicAffineSpace \times_{\affineClosureofBasicAffineSpace} V_\alpha\] and so it suffices to show that the codimension of $S$ in $T^*(V_\alpha) \cap \momentMapFromAFFINECLOSUREofCotangentSpaceWithGROUPT^{-1}(0)$ is at least 2.

To compute the codimension of $S$ in $\momentMapFromAFFINECLOSUREofCotangentSpaceWithGROUPT^{-1}(0) \cap T^*(V_{\alpha})$, we will construct a nonempty (and therefore dense by \cref{Preimage of affinized moment map is integral}) open subset $U'$ of $\momentMapFromAFFINECLOSUREofCotangentSpaceWithGROUPT^{-1}(0)$ such that either $U'$ or one of its $G(k)$-translates contains the generic point of $S$, and compute the codimension of $S \cap U'$ in $U'$. To this end, we recall that a construction of Kazhdan-Laumon \cite{KazhdanLaumonGluingofPerverseSheavesandDiscreteSeriesRepresentation} identifies $V_{\alpha}$ with a symplectic vector bundle over $G/[P_\alpha, P_\alpha]$, see \cite[Corollary 5.7]{GannonProofOftheGinzburgKazhdanConjecture}. Moreover, as we review in the proof of \cite[Lemma 5.21]{GannonProofOftheGinzburgKazhdanConjecture}, there is an open cover of $V_\alpha$, consisting of the preimage $\tilde{C} \subseteq G/[P_{\alpha}, P_{\alpha}]$ of the big cell under the quotient map to $G/P_{\alpha}$ and its $G(k)$-translates, trivializing this vector bundle compatibly the natural action of $T$. In other words, there is a $T$-equivariant isomorphism \begin{equation}\label{Open Subsets}V_{\alpha} \times_{G/[P_{\alpha}, P_{\alpha}]} \tilde{C} \cong \tilde{C} \times \A^2\end{equation} where $T$ acts on the right hand side through the isomorphism \[T \xrightarrow{\sim} \left(\prod_{\beta \neq \alpha} \G_m^{\beta} \right) \times \G_m^{\alpha}\] provided by the fundamental weights. A similar statement holds for the $G(k)$-translates of $\tilde{C}$. This induces an isomorphism on cotangent bundles, compatible with the moment map in the sense that \[
    \begin{tikzcd}
        U := T^*(V_{\alpha} \times_{G/[P_{\alpha}, P_{\alpha}]} \tilde{C}) \ar[r,"\sim"] \ar[d] & T^*(\tilde{C}) \times T^*(\A^2) \ar[d] \\
 \LTd \ar[r,"\sim"] & (\prod_{\beta \neq \alpha} \mathfrak{g}_m^{\beta}) \times \mathfrak{g}_m^{\alpha}
    \end{tikzcd}
    \] commutes, where the downward pointing maps are the moment maps and we identify $\mathfrak{g}_m = \mathrm{Lie} \, \G_m$ with its dual via the canonical identification. Replacing $S$ with a $G(k)$-translate if necessary, we may assume that $s \in U$. We deduce an isomorphism \begin{equation}\label{The Product Isomorphism}U \cap \momentMapFromAFFINECLOSUREofCotangentSpaceWithGROUPT^{-1}(0) \xrightarrow{\sim} (T^*(\tilde{C}) \times_{(\prod_{\beta \neq \alpha} \mathfrak{g}_m^{\beta})} \{0\}) \times (T^*(\A^2) \times_{\mathfrak{g}_m^{\alpha}} \{0\})\end{equation} by taking the intersection with $\momentMapFromAFFINECLOSUREofCotangentSpaceWithGROUPT^{-1}(0)$ in this diagram. But, since $\prod_{\beta \neq \alpha} \G_m^{\beta}$ acts freely on $\tilde{C}$ (because it acts freely on $G/[P_\alpha, P_\alpha]$) the first term in the product on the right hand side of \labelcref{The Product Isomorphism} is smooth. Moreover, the second term in the product \labelcref{The Product Isomorphism} is only singular at the point zero as shown in Example~\ref{Example for SL2}. Therefore $S \cap U$ is a closed subset of the codimension three subset \[(T^*(\tilde{C}) \times_{(\prod_{\beta \neq \alpha} \mathfrak{g}_m^{\beta})} \{0\}) \times \{0\}\] and thus has codimension at least two (in fact, at least 3), as desired.
\end{proof}

\section{Proof of \cref{Intro Main Theorem2}} \label{Proof of Analogue Section}
After proving some preliminary results, we prove \cref{Intro Main Theorem2} in \cref{Final Proof Subsection}. 
\subsection{Geometry of $\LG^* \times_{\LT^*\sslash W} \LT^*$} Let $\tgP := G \times^P (\LG / \LU_P)^*$. This scheme admits a map to $\LL^*\sslash L$ induced by the composite \[G \times (\LG / \LU_P)^* \to (\LG / \LU_P)^* \to \LL^* \to \LL^*\sslash L\] of the projection map and the quotient maps. The Chevalley isomorphism also gives $\LT^* \to \LT^* / W_L \cong \LL^*\sslash L$. 

We let $\tgTHISBaseChanged{P} := \tgP \times_{\LL^*\sslash L} \LT^*$ denote the base change of $\tgP$ along this map. Observe that if $P \subseteq Q$ are parabolic subgroups, then we have a canonical map $\tilde{\nu}_P^{Q}: \tgTHISBaseChanged{P} \to \tgTHISBaseChanged{Q}$.

\begin{Proposition}\label{Geometry of tgTHISBaseChanged New}
For any parabolic $P$, the scheme $\tgTHISBaseChanged{P}$ is integral, normal, and has dimension $\dim(\LG)$. Moreover, the maps $\tilde{\nu}_P^{Q}$ are projective and birational for any parabolic subgroups $P \subseteq Q$.
\end{Proposition}

\begin{proof}
    First, observe that there is a canonical isomorphism 
    \[
    G \times^P ((\LG / \LU_P)^* \times_{\LL^*\sslash L} \LT^*) \xrightarrow{\sim} \tgP \times_{\LL^*\sslash L} \LT^*.
    \]
    Therefore it suffices to show that the scheme $G \times^P ((\LG / \LU_P)^* \times_{\LL^*\sslash L} \LT^*)$ is integral and normal. By the Bruhat decomposition, it suffices to show that the scheme $(\LG / \LU_P)^* \times_{\LL^*\sslash L} \LT^*$ is normal and integral. Since $\LP \cong \LU_P \times \LL$, it suffices to show that the scheme $\LL^* \times_{\LL^*\sslash L} \LT^*$ is normal and integral. This is precisely \cite[Lemma~14]{BorhoBrylinski2}.
    Finally, it is not difficult to check that $\tgP$ itself has dimension $\mathrm{dim}(\LG)$ and, by base change, the map $\tgP \times_{\LL\sslash L} \LT^* \to \tgP$ is finite flat, so our claim follows.

    The fact that $\tilde{\nu}_P^{Q}$ is projective and birational follows in a completely parallel fashion to the proof of \cref{tcNP is normal}: the fact that the map $\tilde{\nu}_B^G$ is birational can be deduced from \cite[Proposition 3.1.36]{ChrissGinzburgRepresentationTheoryandComplexGeometry}.
\end{proof}

\subsection{Gelfand-Graev action and the Namikawa Weyl group}\label{Gelfand Graev and Namikawa Weyl Subsection} Using results of Namikawa, we can also show that any $\Q$-factorial terminalization of the variety $\LGd \times_{\LGd\sslash G} \LTd$ is obtained by a variation of GIT quotient for the $T$-action on $\affineClosureOfCotangentBundleofBasicAffineSpace$, and that, informally speaking, each of these resolutions are related through the Gelfand-Graev $W$-action on $\affineClosureOfCotangentBundleofBasicAffineSpace$. 
We first recall the following results of Namikawa, which we also use to set notation:

\begin{Theorem}\label{Summary of Namikawa Results New}Let $p: Y \to X$ denote a symplectic resolution of a variety $X$ with conical symplectic singularities.
\begin{enumerate}
    \item \cite{NamikawaPoissonDeformationsOfAffineSymplecticVarieties}, \cite{NamikawaPoissonDeformationsOfAffineSymplecticVarietiesII} (See also \cite[Section 2]{LosevDeformationsOfSymplecticSingularitiesAndOrbitMethodForSemisimpleLieAlgebras}) There are universal conic Poisson deformations $\mathcal{Y}$ and $\mathcal{X}$ of $Y$ and $X$, and we may identify the base of $\mathcal{Y}$ with $H^2(Y, \mathbb{C})$. Moreover, there exists a (unique) action of $W$ on $H^2(Y, \mathbb{C})$ such that the base of $\mathcal{X}$ identifies with $H^2(Y,\mathbb{C})/ W$ and such that if $\mathcal{X}' := \Spec (\Gamma(\mathcal{Y},\O_{\mathcal{Y}}))$ then $\mathcal{X}' \xrightarrow{\sim} \mathcal{X} \times_{H^2(Y,\mathbb{C}) / W} H^2(Y,\mathbb{C})$.
    \item \cite{NamikawaPoissonDeformationsOfAffineSymplecticVarietiesII} In the above notation, if $p$ is the Springer resolution, there is a $W$-equivariant isomorphism $\mathfrak{t}^* \xrightarrow{\sim} H^2(Y,\mathbb{C})$ which lifts to an isomorphism of $\widetilde{\LG} \to \mathfrak{t}^*$ with $\mathcal{Y} \to H^2(Y, \mathbb{C})$.
    \item \cite[Corollary 7]{NamikawaPoissonDeformationsAndBirationalGeometry} For a fixed universal Poisson deformation $\mathcal{Y'}$ of a symplectic resolution of $X$, define $\mathcal{Y}_w$ so that \begin{equation}\label{Definition of calYw}
\begin{tikzcd}
\mathcal{Y}'_w \ar[r] \ar[d, "\pi_w"] & \mathcal{Y}' \ar[d, "\pi"] \\
H^2(Y, \mathbb{C}) \ar[r, "w"] & H^2(Y, \mathbb{C})
\end{tikzcd}
\end{equation} is Cartesian for every $w \in W$. For every $\Q$-factorial terminalization of $\mathcal{X}'$, there exists some crepant projective resolution of $X$ and some universal Poisson deformation $\mathcal{Y}'$ such that this $\Q$-factorial terminalization is given by the obvious map $\mathcal{Y}'_w \to \mathcal{X}'$.
\end{enumerate} 
\end{Theorem}

For any $w \in W$, define $\tg_w$ so that \begin{equation}
\begin{tikzcd}\label{Definition of tgw new}
\tg_w \ar[r, "\tilde{w}"] \ar[d, "\pi'_w"] & \tg \ar[d, "\pi'"] \\
\LTd \ar[r, "w"] & \LTd
\end{tikzcd}
\end{equation} is Cartesian. Observe that, although $\tg_w$ and $\tg_{w'}$ are isomorphic as schemes for any $w, w' \in W$, they are not isomorphic as schemes over $\LTd$ if $w \neq w'$.

From the fact that \labelcref{Definition of calYw} and \labelcref{Definition of tgw new} are Cartesian, together with the $W$-equivariance of the isomorphism $\LTd \xrightarrow{\sim} H^2(Y, \mathbb{C})$, we obtain canonical isomorphisms $\tg_w \cong \mathcal{Y}_w$ compatible with the projection onto $\LTd \cong H^2(Y, \mathbb{C})$. For $\lambda \in \characterlatticeforT$, set
\[
Z_{\lambda} := \affineClosureOfCotangentBundleofBasicAffineSpace\sslash_{\lambda} T.
\]

\subsection{Proof of \cref{Intro Main Theorem2}}\label{Final Proof Subsection}
We now prove \cref{Intro Main Theorem2}. We first show \cref{Intro Main Theorem2}(1), which we give an expanded statement of here:
\begin{Proposition}\label{prop:parabolicpartialresGrothendieckNew}
For any parabolic subgroup $P$ and any dominant $\lambda \in \characterlatticeforT$ such that $\lambda(\alpha^{\vee}) = 0$ if and only if $\alpha^{\vee} \in S_{P}$, there is an isomorphism $\tilde{\zeta}_{\lambda}: \tgTHISBaseChanged{P} \xrightarrow{\sim} Z_{\lambda}$. Moreover, if $P'$ is any parabolic subgroup containing $P$ and $\lambda' \in \characterlatticeforT$ is a dominant weight such that $\lambda'(\alpha^{\vee}) = 0$ if and only if $\alpha^{\vee} \in S_{P'}$ then \begin{equation}\label{Compatibility of Tilded Key Isomorphism For Any Parabolics}\xymatrix@R+2em@C+2em{\tgTHISBaseChanged{P} \ar[r]^{\tilde{\nu}^{P'}_{P}} \ar[d]^{\tilde{\zeta}_{\lambda}} & \tgTHISBaseChanged{P'} \ar[d]^{\tilde{\zeta}_{\lambda'}}\\
Z_{\lambda} \ar[r]^{\tilde{\xi}^{\lambda'}_{\lambda}} & Z_{\lambda'} 
  }\end{equation} commutes.
\end{Proposition}

\begin{proof}
This is proved in a parallel manner to \cref{Generically Isomorphisms Are What We Want}. More specifically, one can use essentially identical arguments as in the proof of \cref{Appropriate Bundles Are Ample or Basepoint Free} to prove that $\L^{\tgTHISBaseChanged{P}}_{\lambda}$ is ample, respectively $\L^{\tgTHISBaseChanged{P}}_{\lambda'}$ is semiample, for $\lambda$ and $\lambda'$ as in \cref{Appropriate Bundles Are Ample or Basepoint Free}. One moreover proves the obvious variant of \cref{isotildeNP} in a parallel manner, constructing an isomorphism $(\tilde{\nu}_B^{P})^*(\L_{\lambda}^{\tgTHISBaseChanged{P}}) \cong \L_{\lambda}^{\tg}$. By the normality of $\tgTHISBaseChanged{P}$ (\cref{Geometry of tgTHISBaseChanged New}) we obtain, parallel to the proof of \cref{Generically Isomorphisms Are What We Want}, an isomorphism \[\tilde{\eta}_{P}: \Proj_{\LG^* \times_{\LT^*\sslash W} \LT^*}(R(\tg, \L^{\tg}_{\lambda})) \xrightarrow{\sim} \Proj_{\LG^* \times_{\LT^* \sslash W} \LT^*}(R(\tgTHISBaseChanged{P}, \L^{\tgTHISBaseChanged{P}}_{\lambda}))\] for which the analogue of \labelcref{The Map on Proj Induced by Nu} commutes. But since $\affineClosureOfCotangentBundleofBasicAffineSpace$ is the Cox ring of $\tg$ (see for example \cite[Section 3.6]{GinzburgRicheDifferentialOperatorsOnBasicAffineSpaceandtheAffineGrassmannian}) we see that \[\Proj_{\LG^* \times_{\LT^*\sslash W} \LT^*}(R(\tg, \L^{\tg}_{\lambda})) = \Proj_{\ringOfFunctionsForCOTANGENTBUNDLEOfBasicAffineSpace_0}(\oplus_{n = 0}^{\infty}\ringOfFunctionsForCOTANGENTBUNDLEOfBasicAffineSpace_{n\lambda}) =: \affineClosureOfCotangentBundleofBasicAffineSpace\sslash_{\lambda} T = Z_{\lambda}.\] Therefore our desired morphism is given by the composite of the isomorphism \[\tgTHISBaseChanged{P} \xrightarrow{\sim} \Proj_{\LG^* \times_{\LT^* \sslash W} \LT^*}(R(\tgTHISBaseChanged{P}, \L^{\tgTHISBaseChanged{P}}_{\lambda}))\] obtained from the fact that $\L_{\lambda}^{\tgTHISBaseChanged{P}}$ is ample and $\tilde{\eta}_P^{-1}$. Finally, the commutativity of \labelcref{Compatibility of Tilded Key Isomorphism For Any Parabolics} is completely parallel as in the proof of \cref{Generically Isomorphisms Are What We Want}.
\end{proof}

Now, before we prove \cref{Intro Main Theorem2}(2), we prove the following Proposition and derive a Corollary from it:

\begin{Proposition}\label{prop:smalGrothediecksimnew}
For regular $\theta \in \characterlatticeforT$ and any $\lambda \in \characterlatticeforT$ on the wall of the cone containing $\theta$ in its relative interior, the variation of GIT map $Z_\theta \to Z_\lambda$ does not contract any divisors.
\end{Proposition}

\begin{proof}
Using the Gelfand-Graev action, we may assume that $\theta$ is regular dominant.

Recall \cite{BorhoMacphersonPartialResolutionsOfNilpotentVarieties} that a proper map $Y \to X$ is said to be \textit{small} if 
\begin{equation}\label{Smallness Condition Spelled Out}
\mathrm{dim}(X_i) < \mathrm{dim}(X) - 2i \text{ }\textrm{for any integer $i \ge 0$, where } X_i := \{\mathfrak{q} \in X : \mathrm{dim}(p^{-1}(\mathfrak{q})) \geq i\}.
\end{equation} 
The Grothendieck-Springer resolution is small in this sense \cite[Section 3]{LusztigGreenPolynomialsAndSingularitiesofUnipotentClasses}, see also \cite[Lemma 8.2.5]{AcharPerverseSheavesandApplicationstoRepresentationTheory}. Therefore, the induced map $\tilde{\LG} \to \LGd \times_{\LGd\sslash G} \LTd$ is also a small map. This map, however, is equivalently the affinization map $\tilde{\LG} \to \Spec(\O(\tilde{\LG}))$ for $\tilde{\LG}$. Since the variation of GIT map $\tilde{\xi}_0^{\theta}$ is also given by the affinization map, by \cref{cor:simreschambersnew} we obtain that $\tilde{\xi}_0^{\theta}$ is a small map.

Now suppose for the sake of contradiction that the variation of GIT quotient map $\tilde{\xi}_\lambda^{\theta}$ does contract a divisor. By assumption, there exists some irreducible closed subscheme $Z \subseteq \affineClosureOfCotangentBundleofBasicAffineSpace\sslash_\theta T$ of codimension one such that the codimension of $\tilde{\xi}_\lambda^{\theta}(Z)$ is at least 2. Since $\tilde{\xi}_0^\theta$ is proper, the image $\tilde{\xi}_0^\theta(Z)$ is an irreducible closed subscheme of $\LGd \times_{\LGd\sslash G} \LTd$ and hence is the closure of a generic point $\mathfrak{q}$. Since $\tilde{\xi}_0^\theta(Z) = \tilde{\xi}_0^\lambda\tilde{\xi}_\lambda^\theta(Z)$, the codimension $c$ of $\tilde{\xi}_0^\theta(Z)$ is at least 2. Now consider the map $\tilde{\xi}_0^\theta: Z \to \tilde{\xi}_0^\theta(Z)$. By construction, this map is dominant. Therefore, by generic flatness of this map, the dimension of $(\tilde{\xi}_0^\theta)^{-1}(\mathfrak{q})$ is $(d - 1) - (d - c) = c - 1$. In particular, setting $X := \LGd \times_{\LGd\sslash G}\LTd$, we see that $\mathfrak{q} \in X_{c - 1}$. By the upper semicontinuity of the fiber dimension of the projective morphism $\tilde{\xi}_0^\theta$, the closure $\tilde{\xi}_0^\theta(Z)$ of $\mathfrak{q}$ lies in $X_{c - 1}$ as well. Therefore, \begin{equation}\label{Bigness of fiber dimension}\mathrm{dim}(X_{c - 1}) \geq \mathrm{dim}(\tilde{\xi}_0^\theta(Z)) = \mathrm{dim}(X) - c \geq \mathrm{dim}(X) - 2(c - 1)\end{equation} by this containment, the definition of codimension, and the fact that $c > 1$. On the other hand, since $c \geq 2$, $c - 1 \geq 1$, and so the smallness condition implies that \[\mathrm{dim}(X_{c -1}) < \mathrm{dim}(X) - 2(c - 1)\] which directly contradicts \labelcref{Bigness of fiber dimension}.
\end{proof}

\begin{Corollary}\label{GIT Region for affineclosureofBaseAffineSpace is All of real characterlattice}
A GIT region for the $T$-action on $\affineClosureOfCotangentBundleofBasicAffineSpace$ is given by the entirety of $\characterlatticeforT_\R$. 
\end{Corollary}

\begin{proof}
Recall \cite[Section 3.1]{BellamyCrawSchedlerBirationalGeometryofQuiverVarietiesandOtherGITQuotients} that a GIT region is constructed by taking the GIT fan and removing all walls where the two resolutions, corresponding to the two complements of the wall, are small. Then a GIT region is one of the resulting top-dimensional cones. The GIT chambers for $\affineClosureOfCotangentBundleofBasicAffineSpace$ are given by the Weyl chambers by \cref{Wall and Chamber Structure for T action on affineClosureOfCotangentBundleofBasicAffineSpace}. If we let $\theta$ be any element in a Weyl chamber (i.e. a regular, not necessarily dominant weight) and we let $\lambda \in \characterlatticeforT$ denote some element on a wall then \cref{prop:smalGrothediecksimnew} shows that the variation of GIT map is small. Therefore, to construct the GIT region, we remove \emph{all} walls, which means that the GIT region is the whole of $\characterlatticeforT_\R$. 
\end{proof}

\begin{proof}[Proof of \cref{Intro Main Theorem2}]
It remains to show that the above constructions give \textit{all} crepant partial resolutions of $\mathcal{X}'$, which will finish our proof of \cref{Intro Main Theorem2}(2). This will be similar to the proof of \cref{Intro Main Theorem}. Specifically, we will appeal to \cite[Theorem 1.1]{BellamyCrawSchedlerBirationalGeometryofQuiverVarietiesandOtherGITQuotients}, which, combined with \cite[Corollary 3.26]{BellamyCrawSchedlerBirationalGeometryofQuiverVarietiesandOtherGITQuotients} (see also \cite[Remark 3.27]{BellamyCrawSchedlerBirationalGeometryofQuiverVarietiesandOtherGITQuotients}) says that all crepant partial resolutions of $\mathcal{X}'$ can be exhibited as a variation of GIT map provided that there exists a chamber $C^{\circ}$ (which we will take to be the dominant Weyl chamber) such that: 
\begin{enumerate}
    \item For $\theta \in C^{\circ}$ the GIT quotient $Z_\theta$ is a $\Q$-factorial normal variety and the linearization map gives an isomorphism $\characterlatticeforT_\Q \xrightarrow{\sim} \mathrm{Pic}(Z_\theta/Z_0)_{\Q}$.
    \item Every interior wall of the GIT region $R_C$ is a \textit{flipping wall}, i.e. every interior wall separating two GIT chambers $C_{\pm}^{\circ}$ has the property that if $\theta_{\pm} \in C_\pm^{\circ}$ and $\theta_0$ lies on the wall separating these chambers, then $Z_{\theta_{\pm}}$ are normal, the unstable loci associated to $\tilde{\xi}_{\theta_0}^{\theta_+}$ and $\tilde{\xi}_{\theta_0}^{\theta_-}$ have codimension at least 2, and each of these morphisms contracts at least one curve.
    \item Each boundary wall of $R_C$ is divisorial or of fibre type.
\end{enumerate}
We first show (1). By \cref{Theta Stable Points Equals Theta Semistable Points is Cotangent Bundle}, we may identify $Z_\theta \cong \tilde{\LG}$. This space is smooth and therefore $\Q$-factorial normal. The fact that the linearization map is an isomorphism is precisely the content of \cref{Picard Group of Vector BUndles on Partial Flag Variety}(1). By \cref{GIT Region for affineclosureofBaseAffineSpace is All of real characterlattice}, there are no boundary walls of the GIT region associated to $C$, so (3) is vacuous. It remains to check (2) that every wall of $\characterlatticeforT_\R$ is a flipping wall. 

Using the Gelfand-Graev action, we see that it suffices to prove that, if $\theta$ is regular dominant and $\lambda$ lies on exactly one wall, the morphism $\tilde{\xi}_{\lambda}^{\theta}$ has unstable locus of codimension at least two and contracts a curve. 

We first prove that $\tilde{\xi}_{\lambda}^{\theta}$ has unstable locus of codimension at least two. We recall that, by \cite[Proposition 5.1.4]{Gin} the canonical map \[T^*(G/U)_{\mathrm{reg}} := T^*(G/U) \times_{\LGd} \LGd_{\mathrm{reg}} \xhookrightarrow{}\affineClosureOfCotangentBundleofBasicAffineSpace \times_{\LGd} \LGd_{\mathrm{reg}} =: \affineClosureOfCotangentBundleofBasicAffineSpace_{\mathrm{reg}}\] is an isomorphism. In other words, any point of $\affineClosureOfCotangentBundleofBasicAffineSpace$ which projects to a regular point of $\LGd$ in fact lies in $T^*(G/U)$ itself. Moreover, the $T$-orbit of any point of $T^*(G/U)_{\mathrm{reg}}$ is closed. Since any point of $T^*(G/U)$ is $\theta$-semistable by \cref{Theta Stable Points Equals Theta Semistable Points is Cotangent Bundle}, any point of $T^*(G/U)$ is $\lambda$-semistable as well. Moreover, since $T^*(G/U)_{\mathrm{reg}}$ is an open subset of $T^*(G/U)$, any closed point of $T^*(G/U)_{\mathrm{reg}}$ has trivial $T$-stabilizer. Therefore, any point of $\affineClosureOfCotangentBundleofBasicAffineSpace_{\mathrm{reg}}$ is $\lambda$-stable. Therefore the unstable locus for $\xi^{\theta}_\lambda$, which by definition is a subset of $Z_\theta \cong \tg$, must be contained in the complement of $\tg \times_{\LGd} \LGd_{\mathrm{reg}}$. But this complement has codimension two by \cite[Proposition 1.9.3]{BezrukavnikovRicheAffineBraidGroupActionsonDerivedCategoriesofSpringerResolutions}.

In this case, we observe that, since $\momentMapFromAFFINECLOSUREofCotangentSpaceWithGROUPT^{-1}(0)$ is a closed subscheme of $\affineClosureOfCotangentBundleofBasicAffineSpace$ compatibly with the $T$-action, we obtain a closed embedding \[\momentMapFromAFFINECLOSUREofCotangentSpaceWithGROUPT^{-1}(0)\sslash_{\theta} T \xhookrightarrow{} Z_\theta \] for \textit{any} character $\theta$, which intertwines the variation of GIT maps in the natural way. Since we have identified the variation of GIT map with the map $\nu^P$ by proving \cref{Intro Main Theorem}, the variation of GIT map contracts a divisor by \cref{Quotient Map is Divisorial Contraction}. It is standard that this implies it contracts a curve. (Indeed, since the variation of GIT map is projective, upper semicontinuity of the fiber dimension gives that the set of points in the codomain whose preimage has dimension at least one is closed. This closed set contains a $k$-point by our assumption on $k$. The preimage of this $k$-point thus by construction has dimension at least one, and thus contains a dimension one closed subscheme which is therefore contracted.) Therefore, the map $\tilde{\xi}^{\theta}_{\lambda}$ must also contract a curve.  
\end{proof}

\begin{Corollary}\label{cor:simreschambersnew} 
With the above notation, we have: \begin{enumerate}
    \item If $\theta \in \characterlatticeforT$ is dominant regular, there is an isomorphism $Z_{w\theta} \cong \tg_w$ compatible with the projection onto $\LTd$ for any $w \in W$.  
    \item For any $\lambda \in \characterlatticeforT$, the Gelfand-Graev action induces an isomorphism $Z_{w\lambda} \xrightarrow{\sim} Z_{\lambda}$ which are compatible with the isomorphisms $\tilde{w}$ in \labelcref{Definition of tgw new}.
    \item Every $\Q$-factorial terminalization of $\LGd \times_{\LGd\sslash G} \LTd$ is obtained from the variation of GIT map \[Z_{\theta} \to Z_0 = \affineClosureOfCotangentBundleofBasicAffineSpace\sslash T = \LGd \times_{\LGd\sslash G} \LTd\] for some regular (not necessarily dominant) $\theta$.
    \item The $\Q$-factorial terminalizations associated to regular $\theta, \theta'$ are the same if and only if $\theta$ and $\theta'$ lie in the same chamber.
\end{enumerate}
\end{Corollary} 

\begin{proof}
In the case $w = 1$, \cref{Theta Stable Points Equals Theta Semistable Points is Cotangent Bundle} implies that \[
Z_\theta = T^*(G/U)/T =: \tg\] by definition of $\tg$. Next observe that the diagram \begin{equation}\begin{tikzcd}\label{Gelfand Graev is Automatically Cartesian New}
Z_{w\lambda} \ar[d, "\momentMapFromAFFINECLOSUREofCotangentSpaceWithGROUPT"] \ar[r, "w"] & Z_{\lambda} \ar[d, "\momentMapFromAFFINECLOSUREofCotangentSpaceWithGROUPT"] \\
\LTd \ar[r, "w"] & \LTd
\end{tikzcd}
\end{equation} is Cartesian, where the top arrow is given by the Gelfand-Graev action. Since the isomorphism $\tilde{\zeta}_{\lambda}$ is compatible with the respective projections onto $\LTd$, our desired isomorphism is given by the universal property of the resulting Cartesian diagram. In other words, our desired isomorphism is given by $(\momentMapFromAFFINECLOSUREofCotangentSpaceWithGROUPT, \tilde{\zeta}_{\lambda}\circ w)$. Observe the fact that \labelcref{Gelfand Graev is Automatically Cartesian New} is Cartesian gives (2). By \cref{Summary of Namikawa Results New}(3), it suffices to show that any crepant projective resolution of $X$ is given by the Springer resolution, but this follows immediately from the fact that the only symplectic resolution of the nilpotent cone is the Springer resolution \cite[Corollary 3.2]{FuNamikawaUniquenessofCrepantResolutionsAndSymplecticSingularities} and a resolution of a symplectic variety is crepant if and only if it is symplectic \cite[Proposition 1.1]{FuSymplecticResolutionsforNilpotentOrbits}. The wall-and-chamber structure of the GIT quotient (\cref{Regular Dominant Weights are GIT Chamber}) gives (4). 
\end{proof}


    

    %
\printbibliography
\end{document}